\documentclass{amsart}
\usepackage[english]{babel}

\usepackage[usenames,dvipsnames]{xcolor}
\usepackage[bookmarksopen,bookmarksdepth=2]{hyperref}
\hypersetup{colorlinks=true,citecolor=NavyBlue,linkcolor=BrickRed,urlcolor=Green} 
\usepackage[font=footnotesize,labelfont=bf]{caption}
\usepackage[mathscr]{eucal}
\usepackage{tikz}
\usepackage{tikz-cd}
\usepackage{tikzsymbols} 
\usepackage{adjustbox}
\usepackage{enumitem}
\usepackage{dsfont}
\usepackage{verbatim}
\usepackage{mathtools}
\usepackage{fancyhdr}
\usepackage{threeparttable}

\usepackage{microtype}

\renewcommand{\mathbb}{\mathbf}

\usepackage{amssymb,amsmath,amsfonts,amsthm,epsfig,amscd}
\usepackage{stmaryrd}
\usepackage[all,cmtip,poly]{xy}
\usepackage{color}
\usepackage{xcolor}
\usepackage{array}   
\usepackage{cleveref}
\crefname{equation}{}{} 
\usepackage{relsize}
\newcolumntype{C}{>{$}c<{$}} 

\newcounter{rownumber}[table]
\renewcommand{\therownumber}{(\roman{rownumber})}
\newcolumntype{N}{>{\refstepcounter{rownumber}\therownumber}c}

\theoremstyle{plain}
\newtheorem{theorem}[equation]{Theorem}

\newtheorem{proposition}[equation]{Proposition}

\newtheorem{lemma}[equation]{Lemma}

\newtheorem{corollary}[equation]{Corollary}

\theoremstyle{definition}

\newtheorem{remark}[equation]{Remark}

\newtheorem{definition}[equation]{Definition}

\numberwithin{equation}{section}
\numberwithin{figure}{subsection}

\foreach \theoremenv in {theorem,proposition,lemma, remark, corollary, definition} {
  \AtBeginEnvironment{\theoremenv}{%
    \setlist[enumerate]{label={(\roman*)}}
  }
}

\newif\iffinalrun
\iffinalrun
\else
\fi

\iffinalrun
  \newcommand{\need}[1]{}
  \newcommand{\mar}[1]{}
\else
  \newcommand{\need}[1]{{\tiny *** #1}}
  \newcommand{\mar}[1]{\marginpar{\raggedright\tiny fixme #1}}
\fi

\newcommand{\F}{\mathbb{F}}
\newcommand{\Q}{\mathbb{Q}}
\newcommand{\Z}{\mathbb{Z}}
\newcommand{\Fbar}{\overline{\F}}
\newcommand{\Qbar}{\overline{\Q}}

\newcommand{\Fpbar}{\Fbar_p}

\newcommand{\red}{\operatorname{red}}

\newcommand{\Qp}{\Q_p}

\newcommand{\Qpbar}{\Qbar_p}
\newcommand{\ts}{\sigma_{\mathbf{t}, \mathbf{s}}}
\newcommand{\tss}{\sigma_{\mathbf{t}', \mathbf{s}'}}
\newcommand{\tsss}{\sigma_{\mathbf{t}'', \mathbf{s}''}}
\newcommand{\Gm}{\mathbb{G}_m}

\DeclareMathOperator{\Ext}{Ext}

\DeclareMathOperator{\Gal}{Gal}
\DeclareMathOperator{\GL}{GL}

\DeclareMathOperator{\Hom}{Hom}

\DeclareMathOperator{\Ind}{Ind}

\DeclareMathOperator{\Spec}{Spec}

\DeclareMathOperator{\Sym}{Sym}

\def\eqref #1{(\ref{#1})}

\newcommand{\cA}{{\mathcal A}}

\newcommand{\cE}{{\mathcal E}}
\newcommand{\cF}{{\mathcal F}}

\newcommand{\cI}{{\mathcal I}}

\newcommand{\cK}{{\mathcal K}}

\newcommand{\cO}{{\mathcal O}}

\newcommand{\cS}{{\mathcal S}}

\newcommand{\cX}{{\mathcal X}}
\newcommand{\cY}{{\mathcal Y}}
\newcommand{\cZ}{{\mathcal Z}}

\newcommand{\rhobar}{\overline{\rho}}

\begin{document}
\selectlanguage{english}
\title[Intersections of components of Emerton-Gee stack for $\GL_2$]{Codimension one intersections between components of the Emerton-Gee stack for $\GL_2$}
 \author{Kalyani Kansal}
 \address{Imperial College London, London SW7 2AZ, UK}
 \email{kalyani.kansal@gmail.com}
\keywords{Number theory, Langlands}

\subjclass[MSC Classification]{11S99}

\begin{abstract} Let $p$ be a fixed odd prime, and let $K$ be a finite extension of $\mathbb{Q}_p$ with ring of integers $\mathcal{O}_K$. The Emerton-Gee stack for $\GL_2$ is a stack of $(\varphi, \Gamma)$-modules. The stack, denoted $\mathcal{X}_2$, can be interpreted as a moduli stack of representations of the absolute Galois group of $K$ with $p$-adic coefficients. The reduced part of the Emerton-Gee stack, denoted $\mathcal{X}_{2, \red}$, is an algebraic stack defined over a finite field of characteristic $p$ and can be viewed as a moduli stack of Galois representations with mod $p$ coefficients. The irreducible components of $\mathcal{X}_{2, \red}$ are labelled in a natural way by Serre weights, which are the irreducible mod $p$ representations of $\GL_2(\mathcal{O}_K)$. Each irreducible component of $\mathcal{X}_{2, \red}$ has dimension $[K:\mathbb{Q}_p]$.

In this article, we compute $\GL_2(\mathcal{O}_K)$-extensions of pairs of Serre weights and, motivated by the conjectural categorical $p$-adic Langlands programme, 
we show that a non-trivial extension of a pair of non-isomorphic Serre weights implies a codimension $1$ intersection of the corresponding irreducible components. The converse of this statement is true if the Serre weights are chosen to be sufficiently generic. Furthermore, we show that the number of top-dimensional components in a codimension $1$ intersection is related to the depth of the extensions of the corresponding Serre weights. \end{abstract}

\maketitle
\tableofcontents

\section{Introduction}

Let $p$ be a fixed, odd prime and let $K/\mathbb{Q}_p$ be a finite extension, with ring of integers $\mathcal{O}_K$, residue field $k$ and absolute Galois group $G_K$. In \cite{emerton2019moduli}, Emerton and Gee constructed and studied the stack of rank $d$ \'etale $(\varphi, \Gamma)$-modules defined over the formal spectrum of a ring of integers $\mathcal{O}$ in a finite extension of $\mathbb{Q}_p$. The stack is denoted $\mathcal{X}_{d}$. Over Artinian coefficients, there exists an equivalence of categories between rank $d$ \'etale $(\varphi, \Gamma)$-modules and $d$-dimensional $G_K$-representations that allows one to view $\mathcal{X}_{d}$ as a moduli stack of Galois representations. 

We now recall the following theorem by Emerton and Gee on the geometry of the reduced part of $\mathcal{X}_{d}$, which in the case $d=2$ will provide the setting for this article.

\begin{theorem}{\cite[Thm.~1.2.1]{emerton2019moduli}}
The reduced part of $\mathcal{X}_{d}$, denoted by $\mathcal{X}_{d, \red}$, is an algebraic stack of finite type over $\mathbb{F}_p$. The irreducible components of $\mathcal{X}_{d, \red}$ admit a natural labeling by Serre weights. Each irreducible component of $\mathcal{X}_{d, \red}$ has dimension $[K:\mathbb{Q}_p]d(d-1)/2$.
\end{theorem}

Here, the Serre weights are the irreducible mod $p$ representations of $\mathrm{GL}_d(k)$, and by inflation of $\mathrm{GL}_d(\mathcal{O}_K)$.
The present work is an attempt towards relating the representation theory of $\mathrm{GL}_2(\mathcal{O}_K)$ in terms of extensions of Serre weights with the structure of the intersections of the irreducible components of $\mathcal{X}_{2, \red}$. Our investigation is motivated by a conjectural categorical $p$-adic Langlands correspondence, as we explain below.

\subsection{Categorical \texorpdfstring{$p$}{p}-adic Langlands}
The Emerton-Gee stack $\mathcal{X}_{d}$  is expected to play a central role in the $p$-adic Langlands program, occupying the position played by the moduli stack of $L$-parameters in the work of Fargues-Scholze on the classical Langlands correspondence. More precisely, \cite[Conj.~6.1.6]{egh-ihes} proposes the existence of an exact and fully faithful functor
$$\mathfrak{A}: D^{b}_{f.p.}(\text{sm.} \mathrm{GL}_2(K)) \to D^{b}_{\text{coh}}(\mathcal{X}_d),$$
satisfying a number of properties related to duality and support, and witnessing the $p$-adic Langlands correspondence. The domain $D^{b}_{f.p.}(\text{sm.} \mathrm{GL}_2(K))$ is a derived category of smooth representations of $\mathrm{GL}_d(K)$ while the codomain $D^{b}_{\text{coh}}(\mathcal{X}_d)$ is a derived category of quasicoherent sheaves on $\mathcal{X}_d$. We omit the details of how these categories are precisely defined and refer the interested reader to \cite[Ch.~6]{egh-ihes} instead. 

Specializing to $d=2$, the conditions relating to duality and support that the functor $\mathfrak{A}$ is expected to satisfy imply the following: If $\sigma$ is a Serre weight that is non-Steinberg (a notion that will be defined in \Cref{notation}), the support of $\mathfrak{A}(\text{c-Ind}_{\mathrm{GL}_2(\mathcal{O}_K)}^{\mathrm{GL}_2(K)}\sigma)$ is the irreducible component of $\mathcal{X}_{2, \red}$ labeled by $\sigma$, which we denote by $\mathcal{X}_{2, \red}^{\sigma}$. This follows from \cite[Conj.~6.1.34]{egh-ihes} and \cite[Thm.~1.2]{caraiani2022geometric}.
Now, let $\sigma$ and $\tau$ be two non-isomorphic non-Steinberg Serre weights, and let $$0 \to \sigma \to V \to \tau \to 0$$
be a short exact sequence of $\mathrm{GL}_2(\mathcal{O}_K)$-modules inducing the short exact sequence $$0 \to \text{c-Ind}_{\mathrm{GL}_2(\mathcal{O}_K)}^{\mathrm{GL}_2(K)} \sigma \to \text{c-Ind}_{\mathrm{GL}_2(\mathcal{O}_K)}^{\mathrm{GL}_2(K)} V \to \text{c-Ind}_{\mathrm{GL}_2(\mathcal{O}_K)}^{\mathrm{GL}_2(K)}\tau \to 0$$ of $\mathrm{GL}_2(K)$-modules. Since $\mathfrak{A}$ is (conjecturally) an exact and fully faithful functor, the extension class of $\text{c-Ind}_{\mathrm{GL}_2(\mathcal{O}_K)}^{\mathrm{GL}_2(K)} V$ will be witnessed precisely on the intersection of support of $\mathfrak{A} \circ \text{c-Ind}_{\mathrm{GL}_2(\mathcal{O}_K)}^{\mathrm{GL}_2(K)}\sigma$ and of $\mathfrak{A} \circ \text{c-Ind}_{\mathrm{GL}_2(\mathcal{O}_K)}^{\mathrm{GL}_2(K)}\tau$, that is on $\mathcal{X}_{2, \red}^{\sigma} \cap \mathcal{X}_{2, \red}^{\tau}$. 

Thus, away from $\mathcal{X}_{2, \red}^{\sigma} \cap \mathcal{X}_{2, \red}^{\tau}$, the sheaf coming from $V$ will be isomorphic to the direct sum of sheaves coming from $\sigma$ and $\tau$. In particular, if $\mathcal{X}_{2, \red}^{\sigma} \cap \mathcal{X}_{2, \red}^{\tau}$ is empty, we will have the following diagram of $\mathrm{GL}_2(\mathcal{O}_K)$-representations where the right downward arrow splits:

\[
  \begin{tikzcd}
  \sigma \arrow{r}{} \arrow[swap]{d}{} & \text{c-Ind}_{\mathrm{GL}_2(\mathcal{O}_K)}^{\mathrm{GL}_2(K)}\sigma \arrow{d}{}\\
   V \arrow{r}{} &\text{c-Ind}_{\mathrm{GL}_2(\mathcal{O}_K)}^{\mathrm{GL}_2(K)}V \arrow[dashed, bend right]{u}
   \end{tikzcd}
\]

The horizontal arrows split as maps of $\mathrm{GL}_2(\mathcal{O}_K)$-representations by Mackey's decomposition theorem. The left vertical arrow must then split as well, and $V$ must be isomorphic to $\sigma \oplus \tau$. Thus, if the conjectured functor $\mathfrak{A}$ exists, then 
$$\mathcal{X}_{2, \red}^{\sigma} \cap \mathcal{X}_{2, \red}^{\tau} = \varnothing \implies \Ext^{1}_{\Fbar[\GL_2(\cO_K)]}(\sigma, \tau) = 0.$$
This analysis only depends on the conjectured support  of the sheaves coming from $\sigma$ and $\tau$, and no other properties. As we will see, finer structural details of the intersections of components could be correlated to extensions of the corresponding Serre weights.

\subsection{Main results}

In this article, we specifically study codimension $1$ intersections of irreducible components of $\mathcal{X}_{2, \red}$ in conjunction with extensions of Serre weights. Our first main result is the following (see  \Cref{Ext-group-intersection}):

\begin{theorem}\label{high-level-summary}

If $\sigma$ and $\tau$ are a pair of non-isomorphic Serre weights, then 
$$\mathrm{Ext}^1_{\Fbar[\mathrm{GL}_2(\mathcal{O}_K)]}(\sigma, \tau) \neq 0 \implies \mathrm{dim} \; \mathcal{X}_{2, \red}^{\sigma} \cap \mathcal{X}_{2, \red}^{\tau} = [K:\mathbb{Q}_p] - 1.$$
The converse implication is not always true. However, when both $\sigma$ and $\tau$ are weakly regular, the converse is also true.
\end{theorem}
The notion of weak regularity mentioned in the theorem above is a certain genericity condition on the Serre weights and is defined precisely in \Cref{notation}. In the course of proving \Cref{high-level-summary}, we obtain the following result on extensions of Serre weights that may be of independent interest in the representation theory of $\mathrm{GL}_2(\mathcal{O}_K)$ (see \Cref{ext-GL_2(O_K)-unramified-generic} and \Cref{GL_2(O_K)-ramified-splits}).

\begin{theorem}\label{high-level-summary-3}
If $K/\mathbb{Q}_p$ is unramified and $\sigma$ and $\tau$ are a pair of non-isomorphic weakly regular Serre weights, then the natural map $$\mathrm{Ext}^1_{\Fbar[\mathrm{GL_2}(k)]}(\sigma, \tau) \hookrightarrow \mathrm{Ext}^1_{\Fbar[\mathrm{GL_2}(\mathcal{O}_K)]}(\sigma, \tau)$$ is an isomorphism.

If $K/\mathbb{Q}_p$ is ramified, then for any pair of non-isomorphic Serre weights $\sigma$ and $\tau$, there exists an exact sequence
$$0 \to \mathrm{Ext}^1_{\Fbar[\mathrm{GL_2}(k)]}(\sigma, \tau) \to \mathrm{Ext}^1_{\Fbar[\mathrm{GL_2}(\mathcal{O}_K)]}(\sigma, \tau) \to \Hom_{\Fbar[\GL_2(k)]}(\sigma, H^1(\cK_1, \tau)) \to 0$$
where $\cK_1$ is the kernel of the natural quotient map $\GL_2(\cO_K) \to \GL_2(k)$.
\end{theorem}

The results of \Cref{ext} show that when $p \neq 5$ and $\sigma$ and $\tau$ are non-isomorphic Serre weights with the space $\mathrm{Ext}^1_{\Fbar[\mathrm{GL}_2(\mathcal{O}_K)]}(\sigma, \tau)$ non-zero, then $\mathrm{Ext}^1_{\Fbar[\mathrm{GL}_2(\mathcal{O}_K)]}(\sigma, \tau)$ is $1$-dimensional. A generator of this space can either arise as a $\GL_2(k)$-extension or at a deeper level as an element of $\Hom_{\GL_2(k)}(\sigma, H^1(\cK_1, \tau))$. The depth at which the extensions arise curiously relates to the number of top dimensional components in $\cX_{2, \red}^{\sigma} \cap \cX_{2, \red}^{\tau}$. By \Cref{high-level-summary-3} above, generically, the differences in depth are relevant only in the ramified setting. This and other differences in the structure of intersections in the unramified and ramified settings are summarized below (see \Cref{thm-number}).
 
\begin{theorem}\label{high-level-summary-2}
Let $\sigma$ and $\tau$ be two weakly regular, non-isomorphic Serre weights  such that $\mathcal{X}_{2, \red}^{\sigma} \cap \mathcal{X}_{2, \red}^{\tau}$ is of codimension $1$. Then the following are true:

\begin{enumerate}
    \item When $K$ is unramified over $\mathbb{Q}_p$, the number of components of dimension $[K:\mathbb{Q}_p] - 1$ in $\mathcal{X}_{2, \red}^{\sigma} \cap \mathcal{X}_{2, \red}^{\tau}$ is $1$. When $K$ is ramified over $\mathbb{Q}_p$, this number is $2$ if the $\mathrm{GL}_2(k)$-extensions of $\tau$ by $\sigma$ are non-trivial, and $1$ otherwise.
    \item When $K$ is unramified over $\mathbb{Q}_p$, a component of dimension $[K:\mathbb{Q}_p] - 1$ in $\mathcal{X}_{2, \red}^{\sigma} \cap \mathcal{X}_{2, \red}^{\tau}$ is not contained in $\mathcal{X}_{2, \red}^{\sigma'}$ for any non-Steinberg $\sigma'$ not isomorphic to $\sigma$ or $\tau$. In the ramified case, for sufficiently generic $\sigma$ and $\tau$ (c.f. \Cref{thm-number}), a component of dimension $[K:\mathbb{Q}_p] - 1$ in $\mathcal{X}_{2, \red}^{\sigma} \cap \mathcal{X}_{2, \red}^{\tau}$ is contained in 
    $\cX_{2, \red}^{\sigma'}$ for some non-Steinberg $\sigma'$ not isomorphic to $\sigma$ or $\tau$.
\end{enumerate}
\end{theorem}

We emphasize that \Cref{high-level-summary-2} requires the ``weakly regular" hypothesis, and one can find counterexamples to each of the statements above in the non-generic case.

\subsection{Strategy}

Our strategy involves two separate investigations: the first determines the $\mathrm{GL}_2(\mathcal{O}_K)$-extensions of Serre weights yielding \Cref{high-level-summary-3} and the second determines by hand the pairs of Serre weights for which the corresponding irreducible components intersect in codimension $1$. In the end, we compare the results of the two investigations obtaining \Cref{high-level-summary}. A close examination of the representations that contribute to a codimension $1$ intersection of irreducible components allows us to further conclude \Cref{high-level-summary-2}. We now explain approximately the ideas that go into the two investigations.

In order to compute $\mathrm{GL}_2(\mathcal{O}_K)$-extensions of a pair of Serre weights $\sigma$ and $\tau$, we utilize a Grothendieck spectral sequence to construct a left exact sequence
$$0 \to \mathrm{Ext}^1_{\Fbar[\mathrm{GL_2}(k)]}(\sigma, \tau) \to \mathrm{Ext}^1_{\Fbar[\mathrm{GL_2(\mathcal{O}_K)}]}(\sigma, \tau) \to \Hom_{\Fbar[\GL_2(k)]}(\sigma, H^1(\cK_1, \tau)).$$

Explicit descriptions of the two terms flanking the group of $\mathrm{GL}_2(\mathcal{O}_K)$-extensions are available in \cite{Breuil-Paskunas}. The difficulty is in computing the image of the map $$\mathrm{Ext}^1_{\Fbar[\mathrm{GL_2(\mathcal{O}_K)}]}(\sigma, \tau) \to \Hom_{\Fbar[\GL_2(k)]}(\sigma, H^1(\cK_1, \tau)).$$
We resolve this (in most cases) by writing down cocycles of $\cK_1$ and attempting to extend them by hand to $\GL_2(\cO_K)$.


The starting point of the second investigation comes from the following key theorem {\cite[Thm.~1.2]{caraiani2022geometric} by Caraiani, Emerton, Gee and Savitt describing the finite type points of the irreducible components of $\cX_{2, \red}$. Recall that to each $\rhobar\colon G_K \to \GL_2(\Fpbar)$ one can associate a set $W(\rhobar)$ of Serre weights (see \Cref{notation} for details).

\begin{theorem}{\cite[Thm.~1.2]{caraiani2022geometric}}
    Let $\sigma$ be a non-Steinberg Serre weight. The $\Fbar_p$-points of $\mathcal{X}_{2, \red}^{\sigma}$ are precisely the representations $\rhobar\colon G_K \to \GL_2(\Fbar_p)$ satisfying $\sigma \in W(\rhobar)$.
\end{theorem}

Fixing non-Steinberg non-isomorphic Serre weights $\sigma$ and $\tau$, we proceed by looking for families of representations that have both $\sigma$ and $\tau$ as Serre weights, and therefore give points of $\mathcal{X}_{2, \red}^{\sigma} \cap \mathcal{X}_{2, \red}^{\tau}$. The sizes of these families can then be used to determine the dimension of $\mathcal{X}_{2, \red}^{\sigma} \cap \mathcal{X}_{2, \red}^{\tau}$. As employed in \cite{emerton2019moduli}, a source of families of representations is provided by extensions of fixed $G_K$-characters together with extensions of their unramified twists. Every irreducible component of $\mathcal{X}_{2, \red}$ can be obtained as the closure of such a family. Vector spaces of extensions of fixed $G_K$-characters are typically $[K:\mathbb{Q}_p]$-dimensional. Allowing various unramified twists of the fixed characters adds $2$ to the dimension, while one dimension is taken away because a $\Gm$-orbit of an extension class gives the same representation and yet another dimension is taken away because of a $\Gm$-worth of endomorphisms of each extension. Thus a codimension $1$ intersection of $\mathcal{X}_{2, \red}^{\sigma}$ and $\mathcal{X}_{2, \red}^{\tau}$ may be expected to correspond to the existence of a codimension $1$ family of extensions of fixed $G_K$-characters (as well as their unramified twists) with both $\sigma$ and $\tau$ as their Serre weights. We use detailed descriptions in \cite{ddr,steinmetz} of Serre weights associated to extensions of $G_K$-characters to explicitly compute such families.

\subsection{Outline of the paper}
In \Cref{ext}, we compute the group of extensions of Serre weights as $\GL_2(\mathcal{O}_K)$-representations. In \Cref{geometry}, we relate the dimensions of families of $G_K$-representations with both $\sigma$ and $\tau$ as Serre weights to the dimension of $\mathcal{X}_{2, \red}^{\sigma} \cap \mathcal{X}_{2, \red}^{\tau}$. We also relate the number of sufficiently large families to the number of components of maximal dimension inside $\mathcal{X}_{2, \red}^{\sigma} \cap \mathcal{X}_{2, \red}^{\tau}$. The objective of \Cref{serre-weights} is to recall explicit criteria for computations of Serre weights of representations as described in \cite{ddr,steinmetz}. Along with the results of \Cref{geometry}, these criteria are used to relate the existence of a codimension $1$ intersection between $\cX_{2, \red}^{\sigma}$ and $\cX_{2, \red}^{\tau}$ to a requirement that $\sigma$ and $\tau$ satisfy a precise numerical relationship. This numerical relationship is seen to manifest in two distinct types: type I and type II, defined in \Cref{geometry-to-ext}. In sections \Cref{type1,type2}, we compute all the pairs $\sigma$ and $\tau$ that satisfy the aforementioned numerical relationship of the two types. Finally, \Cref{conclusion} collects all the findings.

\subsection{Notation}\label{notation}
	Let $p>2$ be a fixed prime and let $K$ be a finite extension of $\mathbb{Q}_p$ with valuation ring $\mathcal{O}_K$, residue field $k$ and uniformizer $\pi$. 	The requirement on $p$ is to allow the key input of \cite[Cor.~7.2]{caraiani2022geometric}.
	
	We let $f := f(K/\mathbb{Q}_p)$ and $e := e(K/\mathbb{Q}_p)$. Let $G_K$ be the absolute Galois group of $K$, and $I_K$ the inertia group. Let $E$ be a finite extension of $\Qp$ in a fixed algebraic closure $\Qpbar$ of $\Qp$ so that all embeddings of $K$ into $\Qpbar$ are contained in $E$. Let $\cO$ be the ring of integers of $E$ with residue field $\F$. Denote by $\Fbar$ the residue field of the maximal unramified extension of $E$ inside $\Qpbar$. 
    We will consider representations of the groups $\GL_2(\cO_K)$, $\GL_2(k)$, $G_K$ and $I_K$ in various settings with coefficients in algebraic extensions of $\F$. Equipping algebraic extensions of finite fields with discrete topology, all representations will be continuous. 
	
	Fix an embedding $\theta_{f-1}: k \hookrightarrow \Fbar$. Let $\theta_{f-1-i} := \theta_{f-1}^{p^i}$ for $i \in \Z$. Since $\theta_i = \theta_{i+f}$, we will view the indices $i$ of the embeddings as elements of $\Z/f\Z$. Let $\omega_i$ be the $I_K$-character given by $$\omega_i(g) =  \theta_{i}(\frac{g(\sqrt[p^f - 1]{\pi})}{\sqrt[p^f - 1]{\pi}} ).$$
	
	Let $\mathbf{t}=(t_i)_{i}$ and $\mathbf{s}=(s_i)_{i}$ be tuples of length $f$ indexed over elements of $\Z/f\Z$. We let $\sigma_{\mathbf{t}, \mathbf{s}}$ denote the irreducible $\GL_2(k)$-representation $$\bigotimes_{i\in \Z/f\Z}(\mathrm{det}^{t_i} \otimes \mathrm{Sym}^{s_i}k^{2}) \otimes_{k, \theta_{i}} \Fbar$$
	where $k^2$ is the standard two-dimensional representation of $\GL_2(k)$ and each $s_i \in [0, p-1]$. All irreducible $\GL_2(k)$-representations with coefficients in $\Fbar$ are of this form and are called Serre weights. We can uniquely identify each Serre weight by $\mathbf{s}$ and $\mathbf{t}$ if we demand that $t_i \in [0, p-1] \; \forall i$ and at least one of the $t_i$'s is not $p-1$. Following \cite{gee08}, we say $\sigma_{\mathbf{t}, \mathbf{s}}$ is \textit{weakly regular}, if each $s_i \in [0, p-2]$. We say that $\sigma_{\mathbf{t}, \mathbf{s}}$ is \textit{Steinberg} if each $s_i$ equals $p-1$; we say it is \textit{non-Steinberg} otherwise.

Normalize Hodge-Tate weights in such a way that all Hodge-Tate weights of the cyclotomic character are equal to $-1$. Consistent with the conventions in \cite{emerton2019moduli},
we say that a representation $\rhobar\colon G_K \to \GL_2(\Fbar)$ has Serre weight $\sigma_{\mathbf{t}, \mathbf{s}}$ if $\rhobar$ has a crystalline lift $\rho\colon G_K \to \GL_2(\overline{\mathbb{Q}}_p)$ that satisfies the following condition: 
For each embedding $\theta_i: k \hookrightarrow \mathbb{F}$, there is an embedding $\tilde{\sigma}_i: K \hookrightarrow E$ lifting $\theta_i$ such that the $\tilde{\sigma}_i$ labeled Hodge-Tate weights of $\rho$ are $\{t_i, s_i + t_i + 1\}$, and the remaining $(e-1)f$ pairs of Hodge-Tate weights of $\rho$ are all $\{0, 1\}$. In this situation, we say $\sigma_{\mathbf{t}, \mathbf{s}} \in W(\rhobar)$. 

Let $\mathcal{X}_{2, \red}$, or simply $\mathcal{X}$, be the reduced part of the Emerton-Gee stack for two-dimensional representations of $G_K$. It is defined over $\mathbb{F}$ and is an algebraic stack of pure dimension $[K:\Qp]$. The irreducible components of $\mathcal{X}$ are indexed by the non-Steinberg Serre weights. For a non-Steinberg Serre weight $\sigma_{\mathbf{t}, \mathbf{s}}$, we denote the corresponding irreducible component by $\mathcal{X}_{\sigma_{\mathbf{t}, \mathbf{s}}}$. 

\subsection{Acknowledgements}
This work was done as part of my PhD dissertation. I thank my advisor David Savitt for sharing the problem with me along with his insights on it, as well as providing continual guidance and support. I am also grateful to Karol Koziol and Brandon Levin for helpful conversations, and to Laurence Kluge and the anonymous referee for their careful reading, corrections of typos, and suggestions improving the paper. Editing of this work was completed with the support of National Science Foundation under Grant No. DMS-1926686.

\section{Extensions of Serre weights}\label{ext}

Denote by $\cK$ the group $\GL_2(\mathcal{O}_K)$ 
and by $\cK_n$ the subgroup $1 + \pi^n M_{2}(\mathcal{O}_K)$ for $n \in \mathbb{Z}_{>0}$. For all representations and cohomology groups in this section, the field of coefficients will be $\Fbar$ and all homomorphisms we consider will at least be $\Fbar$-homomorphisms. 
We will suppress the field of coefficients from the notation. Additionally, we will view all $\GL_2(k)$-representations as $\cK/\cK_n$- and $\cK$-representations by inflation. Our objective in this section is to compute the group of $\cK$-extensions of Serre weights.

It is well-known that if $\sigma$ and $\tau$ are $\GL_2(k)$-representations, then there exists an isomorphism 
\begin{align}\label{ext-gp-coh-finite}\Ext^{1}_{\cK/\cK_n}(\sigma, \tau) \cong H^1(\cK/\cK_n, \sigma^{\vee} \otimes \tau).\end{align}
Indeed, in the category of $\cK/\cK_n$-representations, one can take a projective resolution of $\sigma$ by free $\Fbar[\cK/\cK_n]$-modules. As a $\cK/\cK_n$-representation, $$\Hom(\Fbar[\cK/\cK_n], \tau) \cong \Ind_{1}^{\cK/\cK_n}\tau.$$ The isomorphism of \Cref{ext-gp-coh-finite} follows from the observation that induced representations constitute a class of adapted objects (in the sense of homological algebra) for the functor that sends a $\cK/\cK_n$-representation to the vector space of $\cK/\cK_n$-invariants. This is because of Shapiro's lemma and the fact that if $V$ is a $\cK/\cK_n$-representation, then $V$ is a subobject of $\Ind_{1}^{\cK/\cK_n} \mathrm{res}_{1}^{\cK/\cK_n} V$. Taking colimits, we obtain
\begin{align}\label{ext-gp-coh}\Ext^{1}_{\cK}(\sigma, \tau) \cong H^1(\cK, \sigma^{\vee} \otimes \tau).\end{align} 

Henceforth, we will freely use the identifications in \Cref{ext-gp-coh-finite,ext-gp-coh}. Next, for each $n$, we consider the left exact sequence of low degree terms obtained from the inflation-restriction spectral sequence for $\cK/\cK_n$. Passing to colimits as $n$ varies, we obtain
\begin{align}\label{inf-res-gss}
0 \to H^1(\GL_2(k), \sigma) \xrightarrow{\mathit{inf}} H^1(\cK, \sigma) \xrightarrow{\mathit{res}} H^1(\cK_1, \sigma)^{\GL_2(k)}
\end{align}
where $\mathit{inf}$ is the map obtained by pre-composing a $\GL_2(k)$-cocycle with the natural quotient $\cK \to \GL_2(k)$ and the map $\mathit{res}$ is obtained by restricting a continuous $\cK$-cocycle to $\cK_1$. The following proposition describes the group $H^1(\GL_2(k), \sigma^{\vee} \otimes \tau)$ when $\sigma$ and $\tau$ are Serre weights.


\begin{proposition}\label{ext-gamma}
	 Let $\sigma_{\mathbf{t}, \mathbf{s}}$ and $\sigma_{\mathbf{t}', \mathbf{s}'}$ be a pair of Serre weights. Then, $$\dim H^1(\GL_2(k), \sigma_{\mathbf{t}, \mathbf{s}}^{\vee} \otimes  \sigma_{\mathbf{t'}, \mathbf{s'}}) \leq 1.$$ 
  The dimension is $1$ if and only if one of the following holds:
	\begin{enumerate}
				\item \label{ext-gamma-totally-ramified} $f=1$ and either
		\begin{enumerate}[label=(\alph*)]
			\item \label{ext-gamma-totally-ramified-1} $s_0 < p-2$, $s'_0 = p - s_0 - 3$ and $t'_0 \equiv t_0 + s_0 + 1 \mod p-1$; or
			\item \label{ext-gamma-totally-ramified-2} $s_0 \not\in \{0, p-1\}$, $s'_0 = p - s_0 - 1$ and $t'_0 \equiv t_0 + s_0 \mod p-1$.
   	\end{enumerate}
		\item \label{ext-k-i} $f > 1$ and either
		\begin{enumerate}[label=(\alph*)]
			\item\label{ext-k-i-a} $\exists j \in \mathbb{Z}/f\mathbb{Z}$ such that 
   \begin{align*}
   s'_i = \begin{cases} s_i &\text{ for } i \not\in \{j-1, j\}, \\
   s_{i} - 1 &\text{ for } i = j-1, \\
   p - s_i - 2 &\text{ for } i = j.
   \end{cases}
   \end{align*} 
   and $$\sum\limits_{i \in \Z/f\Z} t'_i p^{f-1-i} \equiv \sum\limits_{i \in \Z/f\Z} t_i p^{f-1-i} + (s_j + 1) p^{f-1-j} \mod p^f - 1;$$ or
			\item \label{ext-k-i-b} $\exists j \in \mathbb{Z}/f\mathbb{Z}$ such that 
     \begin{align*}
   s'_i = \begin{cases} s_i &\text{ for } i \not\in \{j-1, j\}, \\
   s_{i} + 1 &\text{ for } i = j-1, \\
   p - s_i - 2 &\text{ for } i = j.
   \end{cases}
   \end{align*}
    and $$\sum\limits_{i \in \Z/f\Z} t'_i p^{f-1-i} \equiv \sum\limits_{i \in \Z/f\Z} t_i p^{f-1-i} - (p-s_j-1) p^{f-1-j} \mod p^f - 1.$$
		\end{enumerate}


		
	\end{enumerate}
 
\end{proposition}

\begin{proof}
This is \cite[Cor.~4.6]{Breuil-Paskunas}. \end{proof}

In the subsequent text, we will use explicit bases for Serre weights which we now describe.
Recall that $k^2$ denotes the standard two-dimensional representation of $\GL_2(k)$.
If $\{x, y\}$ is a basis of $k^2$, then a basis of $\Sym^{s_i} k^2$ is given by homogeneous monomials in $x$ and $y$ of degree $s_i$, namely $\{x^{k_i} z^{s_i - k_i}\}_{k_i}$ where the indexing set is the set of integers $k_i \in [0, s_i]$. This induces a basis of $(\mathrm{det}^{t_i} \otimes \Sym^{s_i} k^2) \otimes_{k, \theta_i} \Fbar$ given by $\{(1 \otimes x^{k_i} y^{s_i - k_i}) \otimes 1\}_{k_i}$, which we will write simply as $\{x^{k_i} y^{s_i - k_i}\}_{k_i}$. We will call $\{x^{k_i} y^{s_i - k_i}\}_{k_i}$ the basis of $(\mathrm{det}^{t_i} \otimes \Sym^{s_i} k^2) \otimes_{k, \theta_i} \Fbar$ induced from $\{x, y\}$. Putting these bases together for each $i$, we obtain a basis of $\sigma_{\mathbf{t}, \mathbf{s}}$ given by 
$$\{\otimes_{i \in \Z/f\Z} x^{k_i} y^{s_i - k_i}\}_{\mathbf{k} \in \mathfrak{K}}$$ 
where the indexing set $\mathfrak{K}$ for the basis is the set of tuples $\mathbf{k} = (k_i)_i$ with $k_i \in [0, s_i]$ for each $i \in \Z/f\Z$. We will call this the basis of $\sigma_{\mathbf{t}, \mathbf{s}}$
induced from $\{x, y\}$.

\begin{proposition}\label{computation-of-H1}
Let $\sigma$ be a $\GL_2(k)$-representation. There exists an isomorphism of $\GL_2(k)$-representations
 \begin{align}\label{H1-eqn}
     H^1({\cK}_1, \sigma) \cong \bigoplus_{i\in \Z/f\Z} \bigl(\sigma \otimes \big((\mathrm{det}^{-1} \otimes \Sym^2 k^2) \otimes_{k, \theta_i} \Fbar\bigr)\bigr) \oplus \bigoplus_{i=1}^{d} \sigma
     \end{align}
where $d$ is the dimension of $\Hom(1 + \pi \mathcal{O}_K, \Fbar)$, such that the following hold true:

\begin{enumerate}
	\item Suppose $\{u, v\}$ is a basis of $k^{2}$ and $\{u^2, uv, v^2\}$ the induced basis of $(\mathrm{det}^{-1} \otimes \Sym^2 k^2) \otimes_{k, \theta_i} \Fbar$. Then, the inclusion $$\sigma \otimes \big((\mathrm{det}^{-1} \otimes \Sym^2 k^2) \otimes_{k, \theta_i} \Fbar\bigr) \: \hookrightarrow \: H^{1}(\cK_1, \sigma) \cong \Hom(\cK_1, \sigma) $$ induced by the isomorphism in \Cref{H1-eqn} is given by  
\begin{align}\label{cocyle-to-basis-element}
&\alpha \otimes u^2 \: &\mapsto \: \kappa_i^{l} \alpha \nonumber \\
&\alpha \otimes 2uv \: &\mapsto \: \epsilon_i \alpha \nonumber \\
&\alpha \otimes v^2 \: &\mapsto \: \kappa_i^{u} \alpha
\end{align}
	for any $\alpha \in \sigma$, where $\kappa_i^l, \epsilon_i, \kappa_i^u$ are group homomorphisms $\cK_1 \to \Fbar$ given by
	\begin{align*}
	&\kappa_i^{l}\biggl(\!\begin{pmatrix} 1+ \pi a & \pi b \\
		\pi c & 1 + \pi d \end{pmatrix}\!\biggr) = \theta_i(c)\\
	&\epsilon_i \biggl(\!\begin{pmatrix} 1+ \pi a & \pi b \\
		\pi c & 1 + \pi d \end{pmatrix}\!\biggr)  = \theta_i(a-d)\\
	&\kappa_i^{u}\biggl(\!\begin{pmatrix} 1+ \pi a & \pi b \\
		\pi c & 1 + \pi d \end{pmatrix}\!\biggr) =\theta_i(b)
\end{align*}

\item The inclusion $\bigoplus_{i=1}^{d} \sigma \hookrightarrow H^1({\cK}_1, \sigma)$ induced by \Cref{H1-eqn} corresponds to group homomorphisms $\cK_1 \to \sigma$ that factor through the determinant and are not given by any of the cocyles appearing in $\bigoplus_{i\in \Z/f\Z} \bigl(\sigma \otimes \big((\mathrm{det}^{-1} \otimes \Sym^2 k^2) \otimes_{k, \theta_i} \Fbar\bigr)\bigr)$.
\end{enumerate}
\end{proposition}
\begin{proof}
By \cite[Prop.~5.1]{Breuil-Paskunas}.
\end{proof}

\begin{corollary}\label{H1-multiple-ways}
If $\sigma$ and $\tau$ are $\GL_2(k)$-representations, then there exists an isomorphism 
$$\Hom (\sigma, H^1(\cK_1, \tau)) \xrightarrow{\sim} H^1(\cK_1, \sigma^{\vee} \otimes \tau).$$
\end{corollary}
\begin{proof}
Via the isomorphism in \Cref{H1-eqn}, one notes that $\sigma^{\vee} \otimes H^1(\cK_1, \tau) \cong H^1(\cK_1, \sigma^{\vee} \otimes \tau).$
\end{proof}
\begin{corollary}\label{ext-finite-level}
If $\sigma$ and $\tau$ are non-isomorphic Serre weights, then $$\Ext^{1}_{\cK}(\sigma, \tau) \cong \Ext^{1}_{\cK/\cK_2}(\sigma, \tau).$$
\end{corollary}
\begin{proof}
We need to verify that every continuous cocycle from $\cK$ to $\sigma^{\vee} \otimes \tau$ factors through $\cK/\cK_2$. It suffices to check this after restricting the cocycle to $\cK_1$. We see from \Cref{computation-of-H1} that since $\sigma$ and $\tau$ are non-isomorphic, 
\begin{align}\label{eqn:gl2-inv}
H^1(\cK, \sigma^{\vee} \otimes \tau)^{\GL_2(k)} \cong \bigoplus_{i\in \Z/f\Z} \bigl(\sigma^{\vee} \otimes \tau \otimes \big((\mathrm{det}^{-1} \otimes \Sym^2 k^2) \otimes_{k, \theta_i} \Fbar\bigr)\bigr)^{\GL_2(k)}.\end{align}
As described in \Cref{computation-of-H1} (i), cocycles in $\sigma^{\vee} \otimes \tau \otimes \big((\mathrm{det}^{-1} \otimes \Sym^2 k^2) \otimes_{k, \theta_i} \Fbar\bigr)$ are obtained by taking tensors with $\Fbar$-valued cocycles $\kappa_i^{l}$, $\epsilon_i$ and $\kappa_i^{u}$. Each of these three factors through $\cK_1/\cK_2$, finishing the proof.
\end{proof}

\begin{lemma}\label{dual-basis-isom}
Let $\sigma_{\mathbf{t}, \mathbf{s}}$ be a Serre weight. Suppose $\{x, y\}$ is the standard basis of $k^{2}$ inducing the basis $\{\otimes_{i \in \Z/f\Z} x^{k_i} y^{s_i - k_i}\}_{\mathbf{k}}$ of $\sigma_{\mathbf{t}, \mathbf{s}}$. Use the same notation to denote the induced basis of $\sigma_{\mathbf{-t-s}, \mathbf{s}}$.
Then $\sigma_{\mathbf{t}, \mathbf{s}}^{\vee} \cong \sigma_{\mathbf{-t-s}, \mathbf{s}}$ under the following map:
\begin{align*}
    &\sigma_{\mathbf{t}, \mathbf{s}}^{\vee} &\rightarrow &&\sigma_{\mathbf{-t-s}, \mathbf{s}} \\
    &\otimes_{i} (x^{k_i} y^{s_i - k_i})^{\vee} &\mapsto && \otimes_i \binom{s_i}{k_i} x^{s_i-k_i} (-y)^{k_i}
\end{align*}
\end{lemma}
\begin{proof}
By direct computation.
\end{proof}

\begin{corollary}\label{ext-cor-ext-dual}
	$\mathrm{Ext}^1_{\cK}(\sigma_{\mathbf{t}, \mathbf{s}}^{\vee}, \sigma_{\mathbf{t'}, \mathbf{s'}}^{\vee}) \neq 0$ if and only if $\mathrm{Ext}^1_{\cK}(\sigma_{\mathbf{t}, \mathbf{s}}, \sigma_{\mathbf{t'}, \mathbf{s'}}) \neq 0$.
\end{corollary}
\begin{proof}
The key observation is that $\sigma_{\mathbf{t}, \mathbf{s}}^{\vee} \cong \sigma_{-\mathbf{t}-\mathbf{s}, \mathbf{s}}$ is a twist of $\sigma_{\mathbf{t}, \mathbf{s}}$ by a power of the determinant character. Thus, the statement of the corollary will hold as long as taking duals of $\sigma_{\mathbf{t}, \mathbf{s}}$ and $\sigma_{\mathbf{t'}, \mathbf{s'}}$ involves twisting by the same character, or in other words, if $\forall a \in k^{\times}$,
$$\prod_{i \in \Z/f\Z}\theta_i (a)^{2t_i + s_i} = \prod_{i \in \Z/f\Z}\theta_i (a)^{2t'_i + s'_i}.$$
Note that the left and right hand sides above give the central characters of $\sigma_{\mathbf{t}, \mathbf{s}}$ and $\sigma_{\mathbf{t'}, \mathbf{s'}}$ respectively. The equality of the two is automatic if $\mathrm{Ext}^1_{\GL_2(k)}(\sigma_{\mathbf{t}, \mathbf{s}}, \sigma_{\mathbf{t'}, \mathbf{s'}}) \neq 0$ because the group algebra of the center of $\GL_2(k)$ is semisimple. Otherwise, using \Cref{computation-of-H1} and \Cref{H1-multiple-ways}, we note that if $H^1(\cK, \sigma_{\mathbf{t}, \mathbf{s}}^{\vee} \otimes \sigma_{\mathbf{t'}, \mathbf{s'}})^{\GL_2(k)}$ is nonzero, then $\sigma_{\mathbf{t}, \mathbf{s}}$ is either isomorphic to $\sigma_{\mathbf{t'}, \mathbf{s'}}$ or is a subrepresentation of $\sigma_{\mathbf{t'}, \mathbf{s'}} \otimes \big((\mathrm{det}^{-1} \otimes \Sym^2 k^2) \otimes_{k, \theta_i} \Fbar\bigr)$ for some $i \in \Z/f\Z$. In either case, the equality of the central characters follows from the fact that $\mathrm{det}^{-1} \otimes \Sym^2 k^2$ has trivial central character.
\end{proof}

\begin{corollary}
	$\mathrm{Ext}^1_{\cK}(\sigma_{\mathbf{t}, \mathbf{s}}, \sigma_{\mathbf{t'}, \mathbf{s'}}) \neq 0$ if and only if $\mathrm{Ext}^1_{\cK}(\sigma_{\mathbf{t'}, \mathbf{s'}}, \sigma_{\mathbf{t}, \mathbf{s}}) \neq 0$.
\end{corollary}
\begin{proof}
By \Cref{ext-cor-ext-dual}, we have $$\mathrm{Ext}^1_{\cK}(\sigma_{\mathbf{t}, \mathbf{s}}, \sigma_{\mathbf{t'}, \mathbf{s'}}) \neq 0 \iff \mathrm{Ext}^1_{\cK}(\sigma_{\mathbf{t}, \mathbf{s}}^{\vee}, \sigma_{\mathbf{t'}, \mathbf{s'}}^{\vee}) \neq 0.$$ Dualizing, the right hand side is equivalent to $\mathrm{Ext}^1_{\cK}(\sigma_{\mathbf{t'}, \mathbf{s'}}, \sigma_{\mathbf{t}, \mathbf{s}}) \neq 0$.\end{proof}

\begin{proposition} \label{H1-criteria}
Let $\sigma_{\mathbf{t}, \mathbf{s}}$ and $\sigma_{\mathbf{t'}, \mathbf{s'}}$ be a pair of non-isomorphic, non-Steinberg Serre weights. Then $\dim H^1(\cK, \sigma_{\mathbf{t}, \mathbf{s}}^{\vee} \otimes \sigma_{\mathbf{t'}, \mathbf{s'}})^{\GL_2(k)} \leq 1$. The dimension is $1$ if and only if there exists an $i \in \Z/f\Z$ such that
\begin{align*}
&s_j = \begin{cases}
    s'_j \pm 2 &\text{ if } j = i,\\
    s'_j &\text{ if } j \neq i,
\end{cases} \\
&\text{and}\\
&\sum\limits_{j \in \Z/f\Z} p^{f-1-j} t_j \equiv \mp p^{f-1-i} + \sum\limits_{j \in \Z/f\Z} p^{f-1-j} t'_j \mod p^f - 1.
\end{align*}
\end{proposition}

\begin{proof}
Proposition 5.4 and Corollary 5.5 in \cite{Breuil-Paskunas}.\end{proof}

\begin{corollary}
Let $\sigma_{\mathbf{t}, \mathbf{s}}$ and $\sigma_{\mathbf{t'}, \mathbf{s'}}$ be a pair of non-isomorphic, non-Steinberg Serre weights. Then $\dim H^1(\cK, \sigma_{\mathbf{t}, \mathbf{s}}^{\vee} \otimes \sigma_{\mathbf{t'}, \mathbf{s'}})^{\GL_2(k)} = \dim H^1(\cK, \sigma_{\mathbf{t'}, \mathbf{s'}}^{\vee} \otimes \sigma_{\mathbf{t}, \mathbf{s}})^{\GL_2(k)}$.
\end{corollary}
\begin{proof} The three criteria listed in \Cref{H1-criteria} are satisfied if and only if they are satisfied upon interchanging $\sigma_{\mathbf{t}, \mathbf{s}}$ and $\sigma_{\mathbf{t'}, \mathbf{s'}}$.\end{proof}


Our next step is to write down $\GL_2(k)$-invariant cocycles in $H^{1}(\cK_1, \ts^{\vee} \otimes \tss)$ as functions $\cK_1 \to \ts^{\vee} \otimes \tss$. This will allow for an explicit description of the map $\mathit{res}$ when $K$ is unramified over $\Qp$, as we will see later. For now, the following lemma will be useful.


\begin{lemma}\label{embedding-serre-weight-in-group-of-cocyles}
	Let $r \leq p-3$. The following are true:
	\begin{enumerate}
		\item We have an isomorphism of $\GL_2(k)$-representations:
\begin{align*}\mathrm{Sym}^{2}k^2 \otimes_k \mathrm{Sym}^{r} k^2 \cong \mathrm{Sym}^{r+2}k^2 \oplus (\Sym^r k^2 \otimes \mathrm{det}) \oplus  (\Sym^{r-2} k^2 \otimes \mathrm{det}^2).\end{align*} 
\item Let bases of $\mathrm{Sym}^{r+2}k^2$, $\mathrm{Sym}^{2}k^2$ and $\mathrm{Sym}^{r}k^2$ be given by
\begin{align*}
&\{w^k z^{r+2-k}\}_{k=0}^{r+2} &\text{ of } &\mathrm{Sym}^{r+2}k^2, &\\
&\{u^j v^{2-j}\}_{j=0}^{2} &\text{ of } &\mathrm{Sym}^{2}k^2, &\text{ and} \\
&\{x^k y^{r-k}\}_{k=0}^{r} &\text{ of } &\mathrm{Sym}^{r}k^2, \\
\end{align*} induced from the standard basis of $k^2$ denoted $\{w, z\}$, $\{u, v\}$ and $\{x, y\}$ respectively. The inclusion $$\Sym^{r+2}k^2 \hookrightarrow \mathrm{Sym}^{2}k^2 \otimes_k \mathrm{Sym}^{r} k^2$$ is given (uniquely upto scalar multiplication) as follows:
	\begin{align*}
		&&&\quad w^k z^{r+2-k} \quad  &\mapsto \qquad &u^2 \otimes \dfrac{k(k-1)}{(r+2)(r+1)} x^{k-2} y^{r+2-k}\\
		&&&&&+ 2uv \otimes \dfrac{k(r+2-k)}{(r+2)(r+1)} x^{k-1} y^{r+1-k} \\
		&&&&& + v^2 \otimes \dfrac{(r+2-k)(r+1-k)}{(r+2)(r+1)} x^{k} y^{r-k} &&\text{for $k\in [2, r]$}\\
		&&&\quad w z^{r+1} \quad &\mapsto \qquad &2uv \otimes \dfrac{1}{r+2} y^{r} + v^2 \otimes \dfrac{r}{r+2} x y^{r-1} && \text{if $r>0$}\\	
		&&&\quad w^{r+1} z \quad &\mapsto \qquad&u^2 \otimes \dfrac{r}{r+2} x^{r-1} y +2uv \otimes \dfrac{1}{r+2} x^{r} &&\text{if $r>0$}\\	
		&&&\quad w z \quad &\mapsto \qquad &uv \otimes 1 &&\text{if $r=0$}\\
		&&&\quad w^{r+2} \quad &\mapsto \qquad &u^2 \otimes x^r &&\\
		&&&\quad z^{r+2} \quad &\mapsto \qquad &v^2 \otimes y^r &&
	\end{align*}
\end{enumerate}

\begin{proof}
	The first statement is from \cite[Prop.~5.4]{Breuil-Paskunas}. The second statement is verified by direct comparison of $\GL_2(k)$-action.
\end{proof}
\end{lemma}

For the following proposition and theorem, we let $\sigma_{\mathbf{t}, \mathbf{s}}$ and $\sigma_{\mathbf{t'}, \mathbf{s'}}$ be a pair of non-isomorphic, non-Steinberg Serre weights such that for some $i \in \Z/f\Z$,
\begin{align}
&s_j = \begin{cases}
    s'_j + 2 &\text{ if } j = i,\\
    s'_j &\text{ if } j \neq i,
\end{cases} \nonumber \\
&\text{and} \nonumber \\
&\sum\limits_{j \in \Z/f\Z} p^{f-1-j} t_j \equiv - p^{f-1-i} + \sum\limits_{j \in \Z/f\Z} p^{f-1-j} t'_j \mod p^f - 1.
\end{align}
(These are the conditions from \Cref{H1-criteria}.) Further, let the bases of $\ts$, $\ts^{\vee}$ and $\tss$ be denoted by
\begin{align*}
&\{\otimes_{j \in \Z/f\Z} w^{k_j} z^{s_j - k_j}\}_{\mathbf{k} \in \mathfrak{K}} &\text{ for } &\ts \text{ as well as for } \sigma_{\mathbf{-t-s}, \mathbf{s}} \cong \ts^{\vee}, &\text{ and}\\
&\{\otimes_{j \in \Z/f\Z} x^{k'_j} y^{s'_j - k'_j}\}_{\mathbf{k'}\in \mathfrak{K'}} &\text{ for } &\tss, 
\end{align*}
which are the bases induced from the standard basis of $k^2$ denoted $\{w, z\}$ in the case of $\ts$ and $\sigma_{\mathbf{-t-s}, \mathbf{s}}$ and denoted $\{x, y\}$ in the case of $\tss$. To be clear, for $\ts^{\vee}$, we are choosing the induced basis of $\sigma_{\mathbf{-t-s}, \mathbf{s}}$ as explained in the paragraph preceding \Cref{computation-of-H1}, instead of the basis obtained by taking duals of the basis of $\ts$. The relationship between our chosen basis and the dual basis is given in \Cref{dual-basis-isom}.  Recall also that the set $\mathfrak{K}$ (resp.  $\mathfrak{K'}$) is the set of all $f$-tuples $\mathbf{k} = (k_j)_j$ (resp.  $\mathbf{k'}=(k'_j)_j$) indexed over $\Z/f\Z$ such that for each $j \in \Z/f\Z$, $k_j \in [0,s_j]$ (resp. $k'_j \in [0,s'_j]$). 


From these, we construct the obvious tensor product basis of $\ts^{\vee} \otimes \tss$ and for later convenience, we write the basis elements by positioning the $i$-th terms of the basis elements of $\ts^{\vee}$ and $\tss$ at the end. Thus, a basis of $\ts^{\vee} \otimes \tss$ is given by the set 
\begin{align}\label{exp-basis} \{\big((\otimes_{j\neq i} w^{k_j} z^{s_j - k_j}) \otimes (\otimes_{j\neq i} x^{k^{\prime}_j}y^{s_j - k'_j})\big) \mathlarger{\otimes} \big(w^{k_i} z^{s_i - k_i} \otimes x^{k'_i}y^{s_i - 2-k'_i}\big)\}_{(\mathbf{k}, \mathbf{k'})}\end{align}
indexed over the elements of $\mathfrak{K} \times \mathfrak{K'}$.

\begin{proposition}\label{gamma-invariant-cocycle}



For Serre weights $\ts$ and $\tss$ as above, the space of $\GL_2(k)$-invariant cocyles of $H^1({\cK}_1/{\cK}_2, \sigma_{\mathbf{t}, \mathbf{s}}^{\vee} \otimes \sigma_{\mathbf{t'}, \mathbf{s'}}) \cong H^1(\cK_1/\cK_2, \sigma_{-\mathbf{t}-\mathbf{s}, \mathbf{s}} \otimes \sigma_{\mathbf{t'}, \mathbf{s'}})$ is spanned by
	\begin{flalign*}
		\kappa_i^{l} A + \epsilon_i B + \kappa_i^{u} C
	\end{flalign*}
where $\kappa_i^{l}$, $\epsilon_i$ and $\kappa_i^{u}$ are homomorphisms ${\cK}_1/{\cK}_2 \to \Fbar$ defined in \Cref{computation-of-H1} and $A$, $B$ and $C$ are elements of $\sigma_{\mathbf{t}, \mathbf{s}}^{\vee} \otimes \sigma_{\mathbf{t'}, \mathbf{s'}}$ defined below.

If $s_i >2$,
\begin{align*}
	&A &=  &&- \left( \sum\limits_{(k_j)_{j\neq i}} (\otimes_{j\neq i} \binom{s_j}{k_j} w^{k_j} (-z)^{s_j - k_j}) \otimes (\otimes_{j \neq i} x^{s_j - k_j} y^{k_j}) \right) \mathlarger{\otimes} \\ 
	&&&&\left(
	\begin{aligned}
	\sum\limits_{k_i = 2}^{s_i-2} \left( \binom{s_i}{k_i}w^{k_i} (-z)^{s_i - k_i} \otimes \dfrac{(s_i - k_i)(s_i - k_i-1)}{s_i(s_i - 1)} x^{s_i - k_i-2} y^{k_i} \right) \\
	+ \left(w (-z)^{s_i - 1} \otimes (s_i -2) x^{s_i - 3} y \right) \\
	+ \left( (-z)^{s_i} \otimes x^{s_i-2}\right)
	\end{aligned}
	\right)
 \\[1cm]
	&B &=  && \left( \sum\limits_{(k_j)_{j\neq i}} (\otimes_{j\neq i} \binom{s_j}{k_j} w^{k_j} (-z)^{s_j - k_j})\otimes (\otimes_{j \neq i} x^{s_j - k_j} y^{k_j}) \right) \mathlarger{\otimes} \\ 
	&&&&\left(
	\begin{aligned}
	\sum\limits_{k_i = 2}^{s_i-2} \left( \binom{s_i}{k_i}w^{k_i} (-z)^{s_i - k_i} \otimes \dfrac{(s_i - k_i)k_i}{s_i(s_i - 1)} x^{s_i - k_i-1} y^{k_i -1} \right) \\
	+ \left( w^{s_i - 1} (-z) \otimes y^{s_i-2} \right) \\
	+ \left( w (-z)^{s_i - 1} \otimes x^{s_i-2} \right) \\
	\end{aligned}
	\right)
\\[1cm]
	&C &=  && \left( \sum\limits_{(k_j)_{j\neq i}} (\otimes_{j\neq i} \binom{s_j}{k_j} w^{k_j} (-z)^{s_j - k_j})\otimes (\otimes_{j \neq i} x^{s_j - k_j} y^{k_j}) \right) \mathlarger{\otimes} \\ 
	&&&&\left(
	\begin{aligned}
	\sum\limits_{k_i = 2}^{s_i-2} \left( \binom{s_i}{k_i}w^{k_i} (-z)^{s_i - k_i} \otimes \dfrac{k_i(k_i-1)}{s_i(s_i - 1)} x^{s_i - k_i} y^{k_i-2} \right) \\
	+ \left( w^{s_i - 1} (-z) \otimes (s_i - 2) x y^{s_i - 3} \right) \\
	+ \left( w^{s_i} \otimes y^{s_i-2}\right)
	\end{aligned}
	\right)
\end{align*}

If $s_i =2$,
\begin{align*}
	&A &=  &&- &\left( \sum\limits_{(k_j)_{j\neq i}} (\otimes_{j\neq i} \binom{s_j}{k_j} w^{k_j} (-z)^{s_j - k_j})\otimes (\otimes_{j \neq i} x^{s_j- k_j} y^{k_j}) \right) \mathlarger{\otimes} \left((-z)^{s_i} \otimes 1\right)\\
	&B &=  &&& \left( \sum\limits_{(k_j)_{j\neq i}} (\otimes_{j\neq i} \binom{s_j}{k_j} w^{k_j} (-z)^{s_j - k_j})\otimes (\otimes_{j \neq i} x^{s_j - k_j} y^{k_j}) \right) \mathlarger{\otimes} \left(  - w z \otimes 1 \right)\\
	&C &=  &&& \left( \sum\limits_{(k_j)_{j\neq i}} (\otimes_{j\neq i} \binom{s_j}{k_j} w^{k_j} (-z)^{s_j - k_j})\otimes (\otimes_{j \neq i} x^{s_j - k_j} y^{k_j}) \right) \mathlarger{\otimes}
	\left( w^{s_i} \otimes 1 \right)
\end{align*}
\end{proposition}

\begin{proof}
We use \Cref{embedding-serre-weight-in-group-of-cocyles} and \Cref{H1-multiple-ways} to identify a non-trivial element of the one-dimensional space $H^1(\cK_1, \ts^{\vee} \otimes \tss)^{\GL_2(k)}$. Writing the duals of the chosen basis of $\ts$ in terms of our chosen basis for $\ts^{\vee}$
using \Cref{dual-basis-isom}, we find that the element is given by $$A \otimes u^2 + 2B \otimes uv + C \otimes v^2$$ where $\{u^2, uv, v^2\}$ are the basis of $(\mathrm{det}^{-1} \otimes \Sym^2 k^2) \otimes_{k, \theta_i} \Fbar$ induced from the standard basis of $k^2$ denoted $\{u, v\}$ and $A$, $B$ and $C$ are as in the statement of the Proposition. Using \Cref{cocyle-to-basis-element}, $A \otimes u^2 + 2B \otimes uv + C \otimes v^2$ is seen to be the cocycle $\kappa_i^{l} A + \epsilon_i B + \kappa_i^{u} C$.\end{proof}

Armed with descriptions of the groups flanking $H^{1}(\cK, \ts^{\vee} \otimes \tss)$ in the left exact sequence \Cref{inf-res-gss}, our next order of business is to study the map denoted $\mathit{res}$. 
Continuing with the setup of \Cref{gamma-invariant-cocycle}, we will now check if a $\GL_2(k)$-invariant cocycle in $H^1({\cK}_1, \sigma_{\mathbf{t}, \mathbf{s}}^{\vee} \otimes \sigma_{\mathbf{t'}, \mathbf{s'}})$ is in the image of the $\mathit{res}$ map in \Cref{inf-res-gss}. Therefore, we will try and extend such a cocyle (necessarily factoring through $\cK_1/\cK_2$ by \Cref{ext-finite-level}) to $\cK/\cK_2$. 

\begin{theorem}\label{ext-GL_2(O_K)-unramified-generic}
Let $e=1$. Let $\ts$ and $\tss$ be a pair of Serre weights as in the setup of \Cref{gamma-invariant-cocycle}. Further, let $s_{i+1} < p-1$. Then $\mathit{res}$ is the zero map.
\end{theorem}
\begin{proof}




We will assume without loss of generality that $t'_j = 0$ for each $j \in \Z/f\Z$. Let $U$ be the group of upper unipotent matrices, and let $D$ be the group of diagonal matrices with $1$ in the bottom right entry. We identify $U$ with $\cO_K$ under the isomorphism
\begin{align}\label{eqn:U_OK}
    \cO_K  &\xlongrightarrow{\sim} U \\
   \alpha &\longmapsto U(\alpha) := \begin{pmatrix} 1 & \alpha \\
0& 1
\end{pmatrix} \nonumber 
\end{align}
and $D$ with $\cO_K^{\times}$ and $k^{\times} \times (1 + \pi \cO_K)$ under the isomorphisms
\begin{align}\label{eqn:D_OK}
&\cO_K^{\times} &&\xlongrightarrow{\sim} &\qquad k^{\times} \times (1 + \pi \cO_K)  &&\xlongrightarrow{\sim} &\qquad D \\
&t &&\longmapsto &\qquad \bar{t}, t - [\bar{t}] &&\longmapsto & \qquad D(t) := \begin{pmatrix} t & 0 \\
0& 1
\end{pmatrix} \nonumber
\end{align}
where $\bar{t}$ is the mod $\pi$ reduction of $t$ and $[\bar{t}]$ is the Teichm{\"u}ller lift of $\bar{t}$.

Let $W$ be the $\langle U, D \rangle$-representation obtained by taking the quotient of $\ts^{\vee} \otimes \sigma_{\mathbf{t'}, \mathbf{s'}}$ by the $\langle U, D \rangle$-invariant $\Fbar$-subspace spanned by those basis elements in \Cref{exp-basis} whose corresponding index $(\mathbf{k}, \mathbf{k'})$ satisfies $k'_j \neq 0$ for some $j$. 
For $\mathbf{k} = (k_j)_j \in \mathfrak{K}$ satisfying $k_j \in [0, s_j]$ for each $j$, we let $\mathfrak{e}_{\mathbf{k}}$ denote the image of $$\big((\otimes_{j \neq i} w^{k_j} z^{s_j - k_j}) \otimes (\otimes_{j \neq i} y^{s_j})\big) \mathlarger{\otimes} \big(w^{k_i} z^{s_i - k_i} \otimes y^{s_i - 2}\big)$$ in $W$.
Evidently, the set $\{\mathfrak{e}_{\mathbf{k}}\}_{\mathbf{k} \in \mathfrak{K}}$ gives a basis of $W$.

Consider the cocycle $\kappa_i^{l} A + \epsilon_i B + \kappa_i^{u} C$ defined in \Cref{gamma-invariant-cocycle}. Denote by $\psi$ the cocycle obtained by restricting the domain to $\langle U, D \rangle \cap \cK_1$ and composing the codomain with the quotient map $\ts^{\vee} \otimes \tss \twoheadrightarrow W$. 
Explicitly, since $\kappa_i^{l}$ vanishes on $\langle U, D \rangle \cap \cK_1$, the cocycle $\psi$ is given by 
\begin{align}\label{cocycle-psi}\psi = - \epsilon_i \mathfrak{e}_{\mathbf{a}} + \kappa_i^{u} \mathfrak{e}_{\mathbf{b}},\end{align}
where $\mathbf{a} = (a_j)_j$ and $\mathbf{b} = (b_j)_j$ satisfy $a_i = s_i - 1$, $a_j = s_j$ for all $j \neq i$ and $b_j = s_j$ for all $j$.  

The proof of the theorem will follow if we show that there does not exist an extension of $\psi$ to a $W$-valued cocycle defined on all of $\langle U, D \rangle$. Suppose, on the contrary that such an extension exists, denoted $q$.  Denote by $q_{\mathbf{k}}$ the coordinates of $q$ corresponding to the basis vector $\mathfrak{e}_{\mathbf{k}}$. 

We study first the cocycle $q|_{D}$. We note that for $t \in \cO_K^{\times}$,
$$D(t) \cdot \mathfrak{e}_{\mathbf{k}}= \theta_{f-1}(\bar{t})^{\lambda_{\mathbf{k}}} \mathfrak{e}_{\mathbf{k}}$$
where $\lambda_{\mathbf{k}}$ is the unique integer in $[0, p^f -1)$ that is equivalent mod $p^f - 1$ to 
\begin{align}\label{eigenvalue}
\sum_{\substack{j \in [0, f-1], \\ \not\equiv i}} p^{f-1-j}(k_j - s_j) + p^{f-1-i} (k_i - s_i +1).\end{align}

Thus, letting $\Fbar (\lambda_{\mathbf{k}})$ denote a one-dimensional $\Fbar$-vector space with action of $t \in \cO_K^{\times}$ given by multiplication with $\theta_{f-1}(\bar{t})^{\lambda_{\mathbf{k}}}$, $$q_{\mathbf{k}}|_{D}: D \to \Fbar (\lambda_{\mathbf{k}})$$ is a cocycle for each $\mathbf{k}$. 
Suppose $\lambda_{\mathbf{k}} \neq 0$. An examination of the expression in \Cref{eigenvalue} shows that this is equivalent to $k_i \neq s_i - 1$ or $k_j \neq s_j$ for some $j \neq i$. This is in turn equivalent to $q_{\mathbf{k}}|_{D \cap \cK_1} = 0$, using \Cref{cocycle-psi}. Thus, under the identification of $D$ with $\cO_K^{\times}$ in \Cref{eqn:D_OK}, $q_{\mathbf{k}}|_{D}$ can be viewed as a cocycle $k^{\times} \to \Fbar(\lambda_{\mathbf{k}})$. We claim that this cocycle is given by a coboundary. Indeed,
$$\sum\limits_{\xi \in k^{\times}} \xi^{\lambda_{\mathbf{k}}} = 0$$ 
because if $\xi_0$ is the generator of the cyclic group $k^{\times}$, $$\xi_0^{\lambda_{\mathbf{k}}} \sum\limits_{\xi \in k^{\times}} \xi^{\lambda_{\mathbf{k}}} =  \sum\limits_{\xi \in k^{\times}} (\xi_0 \xi)^{\lambda_{\mathbf{k}}} = \sum\limits_{\xi \in k^{\times}} \xi^{\lambda_{\mathbf{k}}}.$$ 
Hence, $$H^1(k^{\times}, \Fbar(\lambda_{\mathbf{k}})) \cong {\Fbar}/{(\theta_{f-1}(\xi_0)^{\lambda_{\mathbf{k}}} - 1)\Fbar} = 0.$$ 
Therefore, we can fix an element $e_{\mathbf{k}} \in \Fbar$ such that $$q_{\mathbf{k}}(D(t)) = D(t) \cdot e_{\mathbf{k}} - e_{\mathbf{k}}.$$ Having thus defined $e_{\mathbf{k}}$ when $\lambda_{\mathbf{k}} \neq 0$, we let $e_{\mathbf{k}} := 0$ when $\lambda_{\mathbf{k}} = 0$. Let $e \in W$ be the vector whose $\mathfrak{e}_{\mathbf{k}}$-th coordinate is given by 
$e_{\mathbf{k}}$. By subtracting from $q$ the coboundary $$D(t) \mapsto D(t) \cdot e - e$$ if necessary, we may assume that $q_{\mathbf{k}}|_{D} = 0$ whenever $\lambda_{\mathbf{k}} \neq 0$. When $\lambda_{\mathbf{k}} = 0$, $q_{\mathbf{k}}|_{D}: D \to \Fbar$ is a group homomorphism. Since the order of $k^{\times}$ is prime to $p$, $q_{\mathbf{k}} (D([\xi])) = 0$ for all $\xi \in k^{\times}$. Hence, $q([\xi]) = 0$ for all $\xi \in k^{\times}$.

Next, we study $q|_{U}$. Suppose $\mathbf{k} \neq \mathbf{s}$. Then by \Cref{cocycle-psi}, $q_{\mathbf{k}}|_{U(\pi \mathcal{O}_K)} = 0$ and $q_{\mathbf{k}}|_{U}$ can be seen as a map on $k$. Since $$D([\xi]) U(\alpha) = U([\xi]\alpha) D([\xi])$$ for $\xi \in k^{\times}$, $\bar{\alpha} \in k$ and $\alpha \in \cO_K$ any lift of $\bar{\alpha}$, we have
\begin{align*}
    \theta_{f-1}(\xi)^{\lambda_{\mathbf{k}}}q_{\mathbf{k}}(U(\alpha)) =  q_{\mathbf{k}}(D([\xi]) U(\alpha)) = q_{\mathbf{k}}(U([\xi]\alpha) D([\xi])) = q_{\mathbf{k}}(U([\xi]\alpha)).
\end{align*} 
If $\bar{\alpha} = 0$, then by direct calculation, and if $\bar{\alpha} \neq 0$, then by replacing $\alpha$ with $1$ and $\xi$ with $\bar{\alpha}$ in the equation above, we find that
\begin{align}\label{q(alpha)}
    \theta_{f-1}(\bar{\alpha})^{\lambda_{\mathbf{k}}} q_{\mathbf{k}}(U(1)) = q_{\mathbf{k}}(U(\alpha)).
\end{align}

Next, we will show through an inductive argument that if $\mathbf{k} \neq \mathbf{s}$, then $q_{\mathbf{k}}|_{U} = 0$. The argument will crucially use the following observation: 
If $\mathbf{k} \neq \mathbf{s}$, then there does not exist $m \in [0, f-1]$ such that $\lambda_{\mathbf{k}} = p^{f-1-m}$.
    To see this, suppose on the contrary that
    $$\sum\limits_{j \not\equiv i}p^{f-1-j}(k_j - s_j) + p^{f-1-i} (k_i - s_i +1) \equiv p^{f-1-m} \mod p^f - 1.$$
    If $m \equiv i$ mod $f$, then $k_j = s_j$ for all $j \in \Z/f\Z$, a contradiction. Now suppose $m \not\equiv i$ mod $f$. Identifying $i$ with its lift in $[0, f-1]$, if $i < m$ (resp. $i > m$), there exists a largest $l \in [i+1, m]$ (resp. $[i+1-f, m]$) such that viewing $l$ as an element of $\Z/f\Z$, $s_l < p-1$. This is because $s_{i+1} < p-1$ by hypothesis.
    Thus, mod $p^f - 1$, $$\sum_{\substack{j \in [l, m] \\ \mod f}} p^{f-1-j}s_j + p^{f-1-m} \equiv p^{f-1-l}(s_l + 1).$$ Consequently,
    $$\sum_{j}p^{f-1-j} k_j \equiv \sum_{\substack{j \not\in [l, m] \cup \{i\}  \\ \mod f}} p^{f-1-j}s_j + p^{f-1-i} (s_i - 1) + p^{f-1-l} (s_l + 1),$$ implying that $k_l = s_l + 1$, an impossibility as for each $j$, $k_j \in [0, s_j]$.    

For each $n \geq 0$, we define a relation $<_{n}$ on the set of indices $\mathbf{k}$ by saying $$\mathbf{m} = (m_j)_j <_{n} \mathbf{k} = (k_j)_j$$ if $m_j \leq k_j$ for each $j$ and $\sum_j (k_j - m_j) = n$. Let $\mathbf{k} \not\in \{(0)_j, \mathbf{s}\}$ be such that if there exists $\mathbf{m} <_n \mathbf{k}$ for some $n \geq 2$, then $q_{\mathbf{m}}|_{U} = 0$. Then for each $\alpha \in \mathcal{O}_K$,


\begin{align*}
    &q_{\mathbf{k}}(U(\alpha)) + q_{\mathbf{k}}(U(1)) + \\
    &\hspace{2.4cm}\sum_{\mathbf{m} <_1 \mathbf{k}}
    \left(
    \left(\prod_{j}\binom{s_j - m_j}{k_j - m_j}\right) \theta_{f-1}(\bar{\alpha})^{\sum\limits_j p^{f-1-j}(k_j - m_j)}q_{\mathbf{m}}(U(1))
    \right) \\
    \vspace{0.5cm}\\
    &= q_{\mathbf{k}}(U(\alpha+1)) \\
    \vspace{0.5cm}\\
    &= q_{\mathbf{k}}(U(1+\alpha))\\
    \vspace{0.5cm}\\
    &=q_{\mathbf{k}}(U(1)) + q_{\mathbf{k}}(U(\alpha)) + \sum_{\mathbf{m} <_1 \mathbf{k}}
    \left(
    \left(\prod_{j}\binom{s_j - m_j}{k_j - m_j}\right) q_{\mathbf{m}}(U(\alpha))
    \right)\\
\end{align*}

Therefore, 
\begin{align*}
    &\sum_{\mathbf{m} <_1 \mathbf{k}}
    \left(
    \left(\prod_{j}\binom{s_j - m_j}{k_j - m_j}\right) \theta_{f-1}(\bar{\alpha})^{\sum\limits_j p^{f-1-j}(k_j - m_j)}q_{\mathbf{m}}(U(1))
    \right)  \\
    &=\sum_{\mathbf{m} <_1 \mathbf{k}}
    \left(
    \left(\prod_{j}\binom{s_j - m_j}{k_j - m_j}\right) q_{\mathbf{m}}(U(\alpha))
    \right)\\
     &= \sum_{\mathbf{m} <_1 \mathbf{k}}
    \left(
    \left(\prod_{j}\binom{s_j - m_j}{k_j - m_j}\right) \theta_{f-1}(\bar{\alpha})^{\lambda_{\mathbf{m}}} q_{\mathbf{m}}(U(1))
    \right) &&\text{(by \Cref{q(alpha)})}
\end{align*}
In particular, each element in the image of $\theta_{f-1} \colon k \hookrightarrow \F$ satisfies the following polynomial in $x$:

\begin{flalign*}
    &\sum_{\mathbf{m} <_1 \mathbf{k}} \left(\prod_{j}\binom{s_j - m_j}{k_j - m_j}\right) q_{\mathbf{m}}(U(1)) x^{\lambda_{\mathbf{m}}} \\
    &- \sum_{\mathbf{m} <_1 \mathbf{k}}
    \left(\prod_{j}\binom{s_j - m_j}{k_j - m_j}\right) q_{\mathbf{m}}(U(1)) x^{\sum\limits_j p^{f-1-j}(k_j - m_j)} \\
\end{flalign*}
If there exists an $\mathbf{m} <_1 \mathbf{k}$ such that $q_{\mathbf{m}}(U(1))$ is non-zero, then this polynomial is nonzero because for each $\mathbf{m'} <_1 \mathbf{k}$, there exists a distinct $l(\mathbf{m'}) \in [0, f-1]$ such that $$\sum\limits_j p^{f-1-j}(k_j - m'_j) \equiv p^{f-1-l(\mathbf{m'})} \mod p^f - 1$$ while $\lambda_{\mathbf{m'}} \not\equiv p^{f-1-l}$ for any $l$, as noted earlier. Since a polynomial of degree less than $p^f - 1$ cannot have $|k|$ distinct roots, this creates a contradiction. Therefore, for all $\mathbf{m} <_1 \mathbf{k}$, $q_{\mathbf{m}}(U(1)) = 0$ and by \Cref{q(alpha)}, $q_{\mathbf{m}}|_{U} = 0$. 

Applying this argument to $\mathbf{k}$ satisfying $(0)_j <_1 \mathbf{k}$, we find that  $q_{(0)_j}|_{U} = 0$. Applying inductively to $\mathbf{k}$ satisfying $(0)_j <_n \mathbf{k}$ for $n \in [2, \sum_{j} s_j]$, we find that whenever $\mathbf{m} \neq \mathbf{s}$, $q_{\mathbf{m}}|_{U} = 0$.





Finally, we come to the last leg of the proof. Because $q_{\mathbf{k}}|_{U} = 0$ for each $\mathbf{k} \neq \mathbf{s}$, $$q_{\mathbf{s}}(U(\alpha + \beta)) = q_{\mathbf{s}}(U(\alpha)) + q_{\mathbf{s}}(U(\beta)).$$ Therefore $q_{\mathbf{s}}(U(p)) = p q_{\mathbf{s}}(U(1)) = 0$. From \Cref{cocycle-psi}, we know that $q_{\mathbf{s}}|_{U \cap \cK_1} = \kappa_{i}^{u}|_{U \cap \cK_1}$. As $p$ is the uniformizer of $\mathcal{O}_K$, $\kappa_{i}^{u}(U(p)) \neq 0$, giving a contradiction.
\end{proof}

\begin{theorem}\label{GL_2(O_K)-ramified-splits}
	 	Suppose $e > 1$. Let $\ts$ and $\tss$ be non-isomorphic Serre weights and let $\sigma$ denote the $\GL_2(k)$-representation $\ts^{\vee} \otimes \tss$. Then $\mathit{res}$ is a surjective map.
\end{theorem}

\begin{proof}
	Since $e > 1$, we have $p \in \pi^2 \mathcal{O}_K$. Let $\mathcal{O}_K^{\mathrm{ur}}$ be the ring of integers for the maximal unramified extension of $\Qp$ in $K$. The natural quotient map $\cO_K/\pi^2 \to k$ admits a splitting given by $k \cong \mathcal{O}_K^{\mathrm{ur}}/p \hookrightarrow \mathcal{O}_K/\pi^2$.  This induces a splitting of the exact sequence
\begin{center}
		\begin{tikzcd}
			1\arrow{r} & \cK_1/\cK_2\arrow{r} & \cK/\cK_2\arrow{r} & \GL_2(k)\arrow{r}\arrow[bend left=38, dashed]{l} & 1
		\end{tikzcd}
\end{center}
Therefore, $\cK/\cK_2 \cong \cK_1/\cK_2 \rtimes \GL_2(k)$.

Suppose $\psi$ is a cocycle representing a nonzero element of $H^1(\cK_1/\cK_2, \sigma)^{\GL_2(k)}$. As $\cK_1/\cK_2$ action is trivial on $\sigma$, $H^1(\cK_1/\cK_2, \sigma)^{\GL_2(k)} = Z^1(\cK_1/\cK_2, \sigma)^{\GL_2(k)}$. Invariance under ${\GL_2(k)}$-action means precisely that if $b \in \GL_2(k)$ and $a \in \cK_1/\cK_2$, then $b^{-1} \psi(a^b) = \psi(a)$, where $a^b$ denotes $b a b^{-1}$.

We define a function $\delta$ on $\cK_1/\cK_2 \rtimes \GL_2(k)$ by setting $\delta((a, b))$ equal to $\psi(a)$. We claim that $\delta$ is a cocyle, i.e., $\delta((a, b)(a',b')) = \delta((a, b)) + (a, b)\cdot \delta((a', b'))$. Evaluation of the left hand side gives us:
\begin{align*}
	\text{L.H.S.} &= \delta((a a'^b,bb')) \\
	&= \psi(a a'^b) \\
	&= \psi(a) + \psi(a'^b)
\end{align*}

Evaluation of the right hand side gives us:
\begin{align*}
	\text{R.H.S.} &= \psi(a) + (a, 1)(1, b)\cdot \psi(a') \\
	&= \psi(a) + (1, b)\cdot \psi(a') \\
	&= \psi(a) + (1, b)\cdot((1, b^{-1})\cdot \psi(a'^{b})) &\text{(as $\psi$ is $\GL_2(k)$-invariant)} \\
	&= \psi(a) + \psi(a'^{b}) \\
	&= \text{L.H.S.}
\end{align*}

This establishes that $\delta$ is a cocyle with $\mathit{res}(\delta) =\psi$. Using \Cref{ext-finite-level}, we conclude that the map $\mathit{res}$ is a surjection. \end{proof}


\section{Families of \texorpdfstring{$G_K$}{GK}-representations and stack dimensions}\label{geometry}
In this section, we study the relationship between families of representations and dimensions of closed substacks of $\mathcal{X}$.
We note first that there exists a map of sets \begin{align*}
  \mathit{rep}: \bigcup_{\chi_1, \chi_2} \mathrm{Ext}^{1}_{\Fbar[G_K]}(\chi_2, \chi_1) \to |\cX|
  \end{align*} where the union in the domain is over pairs of $G_K$-characters $\chi_1, \chi_2\colon G_K \to \Fbar^{\times}$ and the map is given by sending the equivalence class of a short exact sequence $1 \to \chi_1 \to V \to \chi_2 \to 1$ to the representation $V$.

For $x \in \Fbar^{\times}$, let $$\mathrm{ur}_x\colon G_K \to \Fbar^{\times}$$ be the character that factors through $G_K \twoheadrightarrow \Gal(k^{\text{sep}}/k)$ mapping the geometric Frobenius to $x$. We say that two $G_K$-representations with $\Fbar$-coefficients are unramified twists of each other if one can be obtained from the other by tensoring with $\mathrm{ur}_x$ for some $x$.

\begin{definition}
  A set $\cF \subset |\cX|$ of isomorphism classes of two-dimensional $G_K$-representations with $\Fbar$-coefficients is a ``family" of representations if there exists a pair of $G_K$-characters $$\chi_1, \chi_2\colon G_K \to \Fbar^{\times}$$ so that for each $V \in \cF$, there exist $a, b \in \Fbar^{\times}$ such that $V = \mathit{rep}(\mathfrak{a})$ for some $\mathfrak{a} \in \mathrm{Ext}^{1}_{\Fbar[G_K]}(\chi_2 \otimes \mathrm{ur}_{b}, \chi_1 \otimes \mathrm{ur}_{a})$.
  


  \end{definition}
  To indicate the data of $\chi_1$ and $\chi_2$ in such a family, we will denote the family by $\mathcal{F}_{\chi_1, \chi_2}$. Note that the choice of $\chi_1$ and $\chi_2$ is non-unique, and in particular, can be changed by unramified twists.

  


 \begin{definition}
	We say that two families $\mathcal{F}_{\chi_1, \chi_2}$ and $\mathcal{F}_{\chi'_1, \chi'_2}$ are \textit{separated} if the set $\cF_{\chi_1, \chi_2} \cup \cF_{\chi'_1, \chi'_2}$ is not itself a family. Otherwise, we say that $\mathcal{F}_{\chi_1, \chi_2}$ and $\mathcal{F}_{\chi'_1, \chi'_2}$ are not separated.
\end{definition}

\begin{definition}
Let $\cE$ be a closed substack of $\cX$. A family $\cF \subset |\cE|$ is said to be maximal in $\cE$ if any family $\cF' \subset |\cE|$ that is not separated from $\cF$ satisfies $\cF' \subset \cF$.
\end{definition}

Consider $\Gm \times \Gm = \Spec \F[x, x^{-1}, y, y^{-1}]$ as parametrizing the unramified twists of $\chi_1$ and $\chi_2$ in the following way:
If $p = (p(x), p(y)) \in \Gm \times \Gm (\Fbar)$, then $p$ corresponds to the pair of characters $\chi_1 \otimes \mathrm{ur}_{p(x)}$ and $\chi_2 \otimes \mathrm{ur}_{p(y)}$.
	
	\begin{definition}
	
	 We say that a family $\mathcal{F}_{\chi_1, \chi_2}$ is \textit{of dimension $\leq d$} (resp. \textit{of dimension $d$)} if there exists a dense open subset $W$ of $\Gm \times \Gm$ such that if $p \in W(\Fbar)$, then 
  $$\{\mathfrak{a} \in \mathrm{Ext}^{1}_{\Fbar[G_K]}(\chi_2 \otimes \mathrm{ur}_{p(y)}, \chi_1 \otimes \mathrm{ur}_{p(x)}) | \mathit{rep}(\mathfrak{a}) \in \cF_{\chi_1, \chi_2}\}$$ %
  is a subspace 
of dimension $\leq d$ (resp. of dimension $d$).

	\end{definition}
	

Let $\{\sigma_i\}_{i\in I}$ be a set of non-Steinberg weights for some indexing set $I$. We record a fact from \cite{gls} that we will use in this section. Suppose $\chi_1$ and $\chi_2$ are distinct $G_K$-characters with coefficients in $\Fbar$. Whenever non-empty, the set 
\begin{align}\label{defn:L-sigma}
L_{\sigma}(\chi_1, \chi_2) := \{\mathfrak{a} \in \Ext^{1}_{\Fbar[G_K]}(\chi_2, \chi_1) | \sigma \in W(\mathit{rep}(\mathfrak{a}))\}
\end{align} is an $\Fbar$-subspace of $\Ext^{1}_{\Fbar[G_K]}(\chi_2, \chi_1)$. 
Further, suppose $\chi'_1$ and $\chi'_2$ are unramified twists of $\chi_1$ and $\chi_2$ respectively. If $\chi'_1 \neq \chi'_2$,
\begin{align}\label{inv-L-sigma} \dim_{\Fbar} \big(\bigcap_{i\in I} L_{\sigma_i}(\chi_1, \chi_2)\big) = \dim_{\Fbar} \big(\bigcap_{i\in I} L_{\sigma_i}(\chi'_1, \chi'_2)\big)\end{align}
and
if $\chi'_1 = \chi'_2$, then \begin{align}\label{special-L-sigma}\dim_{\Fbar} \big(\bigcap_{i\in I} L_{\sigma_i}(\chi_1, \chi_2)\big) + 1= \dim_{\Fbar} \big( \bigcap_{i\in I} L_{\sigma_i}(\chi'_1, \chi'_2)\big).\end{align}





\subsection{Scheme-theoretic families of \texorpdfstring{$G_K$}{GK}-representations}
There exist finitely many $\Fbar$-valued characters of $I_K$ that admit extensions to $G_K$. Each such character is in fact valued in $\mathbb{F}$ and is described uniquely by $a = (a_i)_{i\in \Z/f\Z}$ with each $a_i \in [0, p-1]$ and at least some $a_i < p-1$. Let $\cA$ be the set of such $f$-tuples indexed over $\Z/f\Z$. Then for $a \in \cA$, the corresponding $I_K$-character is given by $$\prod_{i \in \Z/f\Z}\omega_i^{a_i}.$$ Fix an extension of such a character to $G_K$, and denote it by $\psi_{a}\colon G_K \to \F^{\times}$. When each $a_i = 0$, take $\psi_{a} = 1$. When $a$ satisfies
$$\prod_{i \in \Z/f\Z}\omega_i^{a_i} = \epsilon|_{I_{K}}$$ where $\epsilon$ is the mod $p$ cyclotomic character, take $\psi_{a}$ to be $\epsilon$.
By applying the functor $\mathbf{D}$ defined in \cite[Sec.~3.6]{emerton2019moduli}, we obtain a set of $(\varphi, \Gamma)$-modules $\{\mathbf{D}(\psi_{a})\}_{a \in \cA}$ with $\mathbb{F}$-coefficients. Recall that if $A$ is an $\F$-algebra, a $(\varphi, \Gamma)$-module $N$ with $A$-coefficients is a module over the ring $\mathbf{A}_{K, A}$, defined in \cite[Sec.~2.2]{emerton2019moduli}. We will say that $N$ is defined over $A$ or $\Spec A$. If there exists a map $f: \Spec B \to \Spec A$ of $\F$-schemes, then we will denote by $N_B$ or $N|_{\Spec B}$ or $f^{*}N$ the $(\varphi, \Gamma)$-module obtained by changing coefficients from $A$ to $B$, that is, the $(\varphi, \Gamma)$-module $N \otimes_{\mathbf{A}_{K, A}} \mathbf{A}_{K, B}$.
For  \'etale $(\varphi, \Gamma)$-modules $N_1$ and $N_2$ defined over schemes $\Spec A_1$ and $\Spec A_2$ respectively, we let $N_1 \boxtimes N_2$ denote the \'etale $(\varphi, \Gamma)$-module $\mathrm{pr}_1^{*}N_1 \otimes_{\mathbf{A}_{K, A_1 \otimes A_2}} \mathrm{pr}_2^{*}N_2$ defined over $\Spec A_1 \times \Spec A_2$ where for $i \in \{1, 2\}$, $\mathrm{pr}_i \colon \Spec A_1 \times \Spec A_2 \to \Spec A_i$ is the projection onto the $i$-th factor. 

Let $M$ be a rank $1$ {\'e}tale $(\varphi, \Gamma)$-module over $\Gm = \Spec \mathbb{F}[x, x^{-1}]$ such that it is  generated as a free module over $\mathbf{A}_{K, \mathbb{F}[x, x^{-1}]}$ by $v \in M$ with $$\varphi(v) = xv$$ and trivial $\Gamma$-action. As described in \cite[Sec.~5.3]{emerton2019moduli}, if $\alpha \in \Fbar^{\times}$, then $$\mathbf{D}(\mathrm{ur}_{\alpha}) = p^{*}M$$ where $p: \Spec \Fbar \to \Gm$ is the map which at the level of rings is given by mapping $x \mapsto \alpha$. In this sense, $M$ interpolates the unramified characters of $G_K$.    

For $a \in \cA$, let $M_a$ denote the {\'e}tale $(\varphi, \Gamma)$-module $\mathbf{D}(\psi_{a}) \otimes_{\mathbf{A}_{K, \F}} M$ defined over $\Gm$. Thus $M_a$ interpolates the unramified twists of $\psi_{a}$.
Let $X_a = \Gm$ when $\psi_{a}$ is not trivial or cyclotomic, and let $X_a = \Gm \smallsetminus \{1\} = \Spec \F[x, x^{-1}][\frac{1}{x-1}]$ when $\psi_{a}$ is trivial or cyclotomic.

We now make some constructions following \cite[Sec.~5]{emerton2019moduli}. For details refer to \textit{loc. cit.}

\subsubsection*{Extensions of a character by a generic character}
The cohomology groups of $\mathcal{C}^{\bullet}(M_a|_{X_a})$, the Herr complex associated to $M_a|_{X_a}$, vanish in degrees $0$ and $2$ (the latter by Tate local duality). Following the arguments in \cite[Sec.~5.4]{emerton2019moduli}, the cohomology group in degree $1$ is compatible with arbitrary finite type base-change and gives a coherent sheaf on $X_a$. By the local Euler characteristic formula, the rank of this sheaf is the constant $[K:\mathbb{Q}_p]$ at every point. Thus, the reducedness of the base implies that it is a locally free sheaf. Denoting the total space of this sheaf by $V_a$, we find using \cite[Lem.~7.1.2]{emerton2020moduli} that $V_a$ parameterizes extensions of $\mathbf{D}(1)|_{X_a}$ by $M_a|_{X_a}$. That is, $V_a$ represents the functor that assigns to any $\F$-algebra $A$ the data of isomorphism classes of pairs $f = (\underline{f}, \mathfrak{a})$ where $\underline{f} \colon \Spec A \to X_a$ is a map and $\mathfrak{a} \in \Ext^{1}_{(\varphi, \Gamma)/\mathbf{A}_{K, A}}(\mathbf{D}(1)_A, \underline{f}^{*}M_a|_{X_a})$ is an extension class.
For each $b \in \cA$, we can consider $$V_{a, b} := V_a \times \Gm$$ as parameterizing extensions of $\mathbf{D}(1)|_{X_a} \boxtimes M_b$ by 
$M_a|_{X_a} \boxtimes M_b$, \'etale $(\varphi, \Gamma)$-modules defined over $X_a \times \Gm$. That is, if $f = (\underline{f}, \mathfrak{a}) \in V_a(A)$ and $g \in \Gm(A)$, then $(\underline{f},g)$ gives a map $\Spec A \to X_a \times \Gm$ and if $\mathfrak{a}$ is the equivalence class of a short exact sequence
$$1 \to \underline{f}^{*}M_a|_{X_a} \xrightarrow{\iota} E \xrightarrow{\pi} \mathbf{D}(1)_A\to 1,$$ then the extension class encoded by $(f, g) \in V_{a,b}(A)$ is the equivalence class of the short exact sequence 
$$1 \to \underline{f}^{*}M_a|_{X_a} \otimes_{\mathbf{A}_{K, A}} g^{*} M_b \xrightarrow{\iota \otimes 1} E \otimes_{\mathbf{A}_{K, A}} g^{*} M_b\xrightarrow{\pi \otimes 1} \mathbf{D}(1)_A\otimes_{\mathbf{A}_{K, A}} g^{*} M_b \to 1.$$
Furthermore, $E \otimes_{\mathbf{A}_{K, A}} g^{*} M_b$ gives an $A$-point of $\cX$, giving rise to a morphism
$$f_{a,b}\colon V_{a,b} \xrightarrow{} \mathcal{X}$$
mapping an extension of a pullback of $\mathbf{D}(1)|_{X_a} \boxtimes M_b$ by a pullback along the same morphism of 
$M_a|_{X_a} \boxtimes M_b$ to the corresponding \'etale $(\varphi, \Gamma)$-module.
There also exists a map $$V_{a,b} \xrightarrow{\pi_{a, b}} B_{a,b} := X_a \times \Gm$$ induced by the structure map $V_a \to X_a$ and the identity map $\Gm \to \Gm$. On closed points, the map $\pi_{a, b}$ corresponds to choices of unramified twists of $\psi_{a}$ and $\psi_{b}$ respectively.

As described in \cite[Sec.~7.3]{emerton2020moduli}, there exists an action of $\Gm \times \Gm$ on extensions of $\mathbf{D}(1)|_{X_a} \boxtimes M_b$ by $M_a|_{X_a} \boxtimes M_b$ which induces a monomorphism
\begin{align}\label{eqn:generic-action-pr-map}
\Gm \times \Gm \times V_{a,b} \to V_{a,b} \times_{\cX} V_{a,b}
\end{align}
given in the following way: Let $(r, s, f)$ be an $A$-point of the domain such that $f$ corresponds to the data of $\underline{f}\colon \Spec A \to B_{a,b}$ and an extension class $\mathfrak{a}$ represented by a short exact sequence
\begin{align}\label{eqn:ext-class}1 \to \underline{f}^{*}(M_a|_{X_a} \boxtimes M_b) \xrightarrow{\iota} E \xrightarrow{\pi} \underline{f}^{*}(\mathbf{D}(1)|_{X_a} \boxtimes M_b) \to 1. \end{align} Then, the image of $(r, s, f)$ is
\begin{itemize}
\item $f$ in the first coordinate,
\item $rs^{-1} \cdot f$ in the second coordinate (giving the action), where $rs^{-1}\cdot f$ is defined to correspond to the data of the map $\underline{f}: \Spec A \to B_{a,b}$ and the extension class $rs^{-1}\mathfrak{a}$, which is the equivalence class of the short exact sequence
\begin{align}\label{eqn:ext-class-scaled} 1 \to \underline{f}^{*}(M_a|_{X_a} \boxtimes M_b) \xrightarrow{r^{-1}s\iota} E \xrightarrow{\pi} \underline{f}^{*}(\mathbf{D}(1)|_{X_a} \boxtimes M_b) \to 1, \end{align} and
\item the unique automorphism of $E$ as a $(\varphi, \Gamma)$-module that induces the map from the sequence in \Cref{eqn:ext-class} to that in \Cref{eqn:ext-class-scaled} given by multiplication by $r$ on $\underline{f}^{*}(M_a|_{X_a} \boxtimes M_b)$ and $s$ on $\underline{f}^{*}(\mathbf{D}(1)|_{X_a} \boxtimes M_b)$. The uniqueness of this automorphism is a consequence of the fact that by definition of $X_a$, $\mathbf{D}(1)$ is not isomorphic to $M_a$ after restriction to any field-valued point of $X_a$, and therefore, not isomorphic to $M_a$ after any base change.
\end{itemize}
Note that $f_{a,b}(f) = E = f_{a,b}(rs^{-1} \cdot f)$, and so, along with the datum of the automorphism of $E$ in the third bullet, we do indeed get a point of the stacky fiber product $V_{a,b} \times_{\cX} V_{a,b}$.
\subsubsection*{Extensions of a character by itself}
Next, we consider the Herr complex $\mathcal{C}^{\bullet}(\mathbf{D}(1))$ associated to $\mathbf{D}(1)$ defined over $\F$. 
Each of the cohomology groups is a finite-dimensional vector space over $\mathbb{F}$. As before, the considerations in \cite[Sec.~5.4]{emerton2019moduli} show that for any finite type $\mathbb{F}$-algebra $R$, $$H^1(\mathcal{C}^{\bullet}(\mathbf{D}(1)_{R})) = H^1(\mathcal{C}^{\bullet}(\mathbf{D}(1)) \otimes R.$$ 
Thus $V_{1} = \Spec \Sym^{\bullet}(H^1(\mathcal{C}^{\bullet}(\mathbf{D}(1))^{\vee})$ parameterizes extensions of $\mathbf{D}(1)$ by itself, and for each $b \in \cA$, we can consider $V_{1,b}:= V_1 \times \Gm$ as parameterizing extensions of $\mathbf{D}(1) \boxtimes M_b$ by itself.
Similar to the definition of $f_{a,b}$, we can define a map $$f_{1,b}\colon V_{1,b} \xrightarrow{} \mathcal{X}$$
mapping extensions of pullbacks of $\mathbf{D}(1) \boxtimes M_b$ by themselves to the corresponding $(\varphi, \Gamma)$-modules. The analogue of $\pi_{a,b}$ in this setting is the map $\pi_{1,b}\colon V_{1,b} \to \Gm$ which is simply projection onto the second factor.
Departing a little from our construction in \Cref{eqn:generic-action-pr-map}, we specify a monomorphism
\begin{align}\label{eqn:borel-action}
\Gm \times \Gm \times \mathbb{G}_a \times V_{1,b} \to V_{1,b} \times_{\cX} V_{1,b},
\end{align}
by first making a choice for a basis of the \'etale $(\varphi, \Gamma)$-module corresponding to the map $f_{1,b}$. 


The universal point of $V_{1, b}$ corresponds to the map $\pi_{1,b}: V_{1,b} \to \Gm$ and a universal extension of $(\varphi, \Gamma)$-modules over $V_{1,b}$. Fix a representative short exact sequence of this universal extension:
$$1 \to \pi_{1,b}^{*}(\mathbf{D}(1) \boxtimes M_b) \xrightarrow{\iota} E \xrightarrow{\pi} \pi_{1,b}^{*}(\mathbf{D}(1) \boxtimes M_b) \to 1,$$
where $E$ is the $(\varphi, \Gamma)$-module corresponding to $f_{1,b}$.
Let $\cO(V_{1,b})$ be the ring of global functions on $V_{1,b}$. Since $\mathbf{D}(1) \boxtimes M_b$ is projective as a $\mathbf{A}_{K, \cO(V_{1,b})}$-module, the short exact sequence admits a splitting of $\mathbf{A}_{K, \cO(V_{1,b})}$-modules. Fix such a splitting and using it, fix a basis of $E$ as a rank $2$ free $\mathbf{A}_{K, \cO(V_{1,b})}$-module, with the first basis element being the image of a fixed free generator of $\mathbf{D}(1) \boxtimes M_b$ under the inclusion $\iota$ and the second basis element being the image of a fixed free generator of $\mathbf{D}(1) \boxtimes M_b$ under the splitting of $\pi$. Denote the basis by $\{e_1, e_2\}$. Suppose the action of $\varphi$ and $\gamma \in \Gamma$ with respect to this basis is given by upper triangular $2 \times 2$ matrices $M(\varphi)$ and $M(\gamma)$ respectively. Given a map $f \colon \Spec A \to V_{1, b}$, we can pull back along $f$ to obtain an \'etale $(\varphi, \Gamma)$-module $f^{*}E$ over $\Spec A$ with a basis $\{f^{*}e_{1}, f^{*}e_{2}\}$, and matrices $f^{*}M(\varphi)$ and $f^{*}M(\gamma)$ describing the action of $\varphi$ and $\gamma \in \Gamma$ on $f^{*}E$.

Armed with these choices, we now describe the map in \Cref{eqn:borel-action}. Let $(r,s,b,f)$ be an $A$-point of the domain. Suppose the map $f\colon \Spec A \to V_{1,b}$ corresponds to the isomorphism class of a tuple $(\underline{f}, \mathfrak{a})$, where $\underline{f}$ is a map $\Spec A \to \Gm$ and $$\mathfrak{a} \in \Ext^{1}_{(\varphi, \Gamma)/\mathbf{A}_{K, A}}(\mathbf{D}(1)_A\otimes_{\mathbf{A}_{K, A}} \underline{f}^{*}M_b, \mathbf{D}(1)_A\otimes_{\mathbf{A}_{K, A}} \underline{f}^{*}M_b)$$ is an extension class. Then the image of $(r,s,b,f)$ is given by 
\begin{itemize}
\item $f$ in the first coordinate, 
\item $f' = rs^{-1} \cdot f$ in the second coordinate, where $rs^{-1} \cdot f$ is defined to correspond to the tuple $(\underline{f}, rs^{-1}\mathfrak{a})$, and 
\item the datum of an isomorphism $f^{*}E \to f^{\prime *}E$ with respect to the basis $\{f^{*}e_{1}, f^{*}e_{2}\}$ of the domain and $\{f^{\prime *}e_{1}, f^{\prime *}e_{2}\}$ of the codomain given by the matrix
$$\begin{pmatrix}
    r & b \\
    0 & s
\end{pmatrix}.$$
\end{itemize}

\subsubsection*{Extensions of a character by its cyclotomic twist}
Finally, 
we consider the Herr complex associated to $\mathbf{D}(\epsilon)$, denoted $\mathcal{C}^{\bullet}(\mathbf{D}(\epsilon))$. 
 As before, viewing the finite-dimensional degree $1$ cohomology group as an invertible sheaf on a point, the total space gives a vector bundle $V_\epsilon$ defined over $\mathrm{Spec}\: \mathbb{F}$ that parameterizes extensions of $\mathbf{D}(1)$ by $\mathbf{D}(\epsilon)$.
For each $b \in \cA$, we can consider $V_{\epsilon, b}:= V_{\epsilon} \times \Gm$ as parameterizing extensions of $\mathbf{D}(1) \boxtimes M_b$ by $\mathbf{D}(\epsilon) \boxtimes M_b$. Thus, we have a map
$$f_{\epsilon,b} \colon V_{\epsilon,b} \xrightarrow{} \mathcal{X}$$
given by sending an extension of a pullback of $\mathbf{D}(1) \boxtimes M_b$ by a pullback along the same morphism of $\mathbf{D}(\epsilon) \boxtimes M_b$ to the corresponding \'etale $(\varphi, \Gamma)$-module. Let $$\pi_{\epsilon, b}\colon V_{\epsilon,b} \to \Gm$$ be the projection onto the second factor.
Exactly as in the setting of $\Cref{eqn:generic-action-pr-map}$, there exists a monomorphism 
\begin{align}\label{eqn:cyclotomic-action-pr-map}
\Gm \times \Gm \times V_{\epsilon,b} \to V_{\epsilon,b} \times_{\cX} V_{\epsilon,b}
\end{align}
given in precisely the same way.


\subsubsection*{Irreducible representations}
Next, we construct a version of the maps $f_{a,b}$, $f_{1,b}$ and $f_{\epsilon, b}$ to obtain the \'etale $(\varphi, \Gamma)$-modules corresponding to the irreducible mod $p$ representations of $G_K$.
For each irreducible $2$-dimensional representation $\rhobar$ defined over a finite extension $\mathbb{F}'$ of $\F$,
there exists a map $\mathbf{D}(\rhobar) \colon \Spec\F' \to \mathcal{X}$. In order to capture the unramified twists of $\rhobar$, we construct a map
\begin{align}\label{defn:irreducible-maps}
f_{\rhobar}: \mathbb{G}_{m,\F'} \to \cX_{\F'} \to \mathcal{X}
\end{align}
corresponding to $\mathbf{D}(\rhobar) \boxtimes M$, where $M$ is the rank $1$ \'etale $(\varphi, \Gamma)$-module we constructed earlier to interpolate the unramified characters of $G_K$. There exists a monomorphism
\begin{align}\label{eqn:irred-action-pr}
\mathbb{G}_{m,\F'} \times_{\F'} \mathbb{G}_{m,\F'} \to \mathbb{G}_{m,\F'} \times_{\cX_{\F'}} \mathbb{G}_{m,\F'}
\end{align}
given by mapping a point $(r,s)$ of the domain to $s$ in both the first and second coordinates of the codomain and the automorphism of the $(\varphi, \Gamma)$-module corresponding to $f_{\rhobar} \circ s$ given by scalar multiplication by $r$.


\begin{lemma}\label{finite-type-points-fiber}
Suppose $\F''$ is a finite extension of $\F$ (resp. of $\F'$ appearing in \Cref{eqn:irred-action-pr}), then\Cref{eqn:generic-action-pr-map,eqn:borel-action,eqn:cyclotomic-action-pr-map} (resp. \Cref{eqn:irred-action-pr}) induce bijections on finite type points.
\end{lemma}
\begin{proof}
This is easily verified on applying the functor $T_{\F''}$ defined in \cite[Sec.~3.6]{emerton2019moduli} to pass from projective \'etale $(\varphi, \Gamma)$-modules with $\F''$-coefficients to the equivalent category of $G_K$-representations with $\F''$-coefficients. 
\end{proof}

\subsubsection*{Top-dimensional irreducible components in \texorpdfstring{$\cap_i \cX_{\sigma_i}$}{X-sigma-cap-X-tau}}
Let $\{\sigma_i\}_{i \in I}$ be a fixed set of non-Steinberg, pairwise non-isomorphic Serre weights. Let $$\mathcal{E} := \bigcap_{i \in I} \mathcal{X}_{\sigma_i}.$$ By \cite[Thm.~1.2]{caraiani2022geometric}, the set $|\cE|$ is precisely the set of isomorphism classes of $G_K$-representations 
$$\{\rhobar\colon G_K \to \GL_2(\Fbar)| \; \forall i \in I, \sigma_i \in W(\rhobar)\}.$$ From \Cref{inv-L-sigma}, it follows that any maximal family of representations contained in $|\cE|$ has a well-defined dimension.


	\begin{theorem}
		\label{geometry-proposition-dimension-intersection}
		Let $d\geq 0$. Suppose all maximal families of representations contained in $|\mathcal{E}|$ are of dimension $\leq d$, and moreover, $|\mathcal{E}|$ contains at least one maximal family of dimension $d$. Then the following are true:
		\begin{enumerate}
		    \item $\mathcal{E}$ has dimension $d$.
		    \item If $d>0$, the number of $d$-dimensional components in $\mathcal{E}$ equals the number of $d$-dimensional pairwise separated maximal families contained in $\mathcal{E}$.
		    \item Let $d=0$, and let
		    $$C := \{\rhobar\colon G_K \to \GL_2(\mathbb{F}) \: | \: \rhobar \text{ is semisimple}\}/\sim$$ where $\rhobar \sim \rhobar'$ if $\rhobar$ and $\rhobar'$ are isomorphic as $I_K$-representations. Then the number of $d$-dimensional components in $\mathcal{E}$ equals $|C|$.
		\end{enumerate}
	\end{theorem}

The proof of this proposition will be presented after \Cref{exclude-contr-of-triv}. First, for each $a \in \cA$, let $Y_{a,b} := \mathcal{E} \times_{\mathcal{X}} V_{a,b}$. We also define  $Y_{1,b} := \mathcal{E} \times_{\mathcal{X}} V_{1, b}$ and $Y_{\epsilon,b} := \mathcal{E} \times_{\mathcal{X}} V_{\epsilon,b}$.

\begin{proposition}\label{geometry-lemma-dimension-max}
        Let $\chi_1$ and $\chi_2$ be fixed $G_K$-characters. Suppose $|\mathcal{E}|$ contains a maximal family $\mathcal{F}_{\chi_1, \chi_2}$ of representations of dimension $d$.
        Let $a, b \in \cA$ be such that $\psi_b$ is an unramified twist of $\chi_2$, while $\psi_a \otimes \psi_b$ is an unramified twist of $\chi_1$. Then the scheme-theoretic image of $Y_{a,b}$ under $f_{a,b}$ has dimension $d$ and the number of $d$-dimensional irreducible components in this scheme-theoretic image is exactly one.
        
\end{proposition}

\begin{proof}
Let $q$ be a closed point of $B_{a,b}$, and after fixing an embedding $\kappa(q) \hookrightarrow \Fbar$, let $\overline{q}$ be the corresponding $\Fbar$-point of $B_{a,b}$.  
By \cite[Thm.~1.2]{caraiani2022geometric},
the set $(Y_{a,b})_{\overline{q}}(\Fbar)$ equals $$\bigcap_{i \in I} L_{\sigma_i}(\psi_a \otimes \psi_b \otimes \mathrm{ur}_{\overline{q}(x)}, \psi_b \otimes \mathrm{ur}_{\overline{q}(y)})$$ 
and in particular, forms a vector space. By the construction of $B_{a,b}$, representations corresponding to the \'etale ($\varphi, \Gamma$)-modules obtained by pushing forward $Y_{a,b} (\Fbar)$ under $f_{a,b}$ are never an extension of a character by itself, and therefore, by \Cref{inv-L-sigma}, $\dim_{\Fbar} (Y_{a,b})_{\overline{q}}(\Fbar)$ does not depend on the choice of $q$ or $\overline{q}$. Since $|\mathcal{E}|$ contains a maximal family $\mathcal{F}_{\chi_1, \chi_2}$ of representations of dimension $d$, $\dim_{\Fbar} (Y_{a,b})_{\overline{q}}(\Fbar) = d$. This implies furthermore that $\pi_{a,b}(|Y_{a,b}|) = |B_{a,b}|$.

Since the $\Fbar$-points of $(Y_{a,b})_{\overline{q}}$ form a $d$-dimensional vector space, the reduced induced closed subscheme of $(Y_{a,b})_q$ must be cut out by homogeneous linear equations in $V_{a,b} \times_{B_{a,b}} \kappa(q)$ and thus be irreducible of dimension $d$.


Let $S$ be an irreducible component of $Y_{a, b}$. Denote by $\overline{f_{a,b}(S)}$ and $\overline{\pi_{a,b}(S)}$ the scheme-theoretic images of $S$ under $f_{a,b}$ and $\pi_{a,b}$ respectively. By \cite[\href{https://stacks.math.columbia.edu/tag/0DS4}{Tag 0DS4}]{stacks-project}, there exists a dense open $U  \subset S$ such that for any $p \in U(\Fbar)$, 
\begin{align}\label{eqn:fiber-dim-eqn-1}
    \dim \overline{f_{a,b}(S)} = 
    \dim S - \dim_p(S_{f_{a,b}(p)})
\end{align} and 
\begin{align}\label{eqn:fiber-dim-eqn-2}
\dim (\overline{\pi_{a,b}(S)}) = 
\dim S - \dim_{p} (S_{\pi_{a,b}(p)}).\end{align}



 Fix $p \in U(\Fbar)$. Since $(Y_{a,b})_{\pi_{a,b}(p)}$ is irreducible, its reduced induced closed subscheme $((Y_{a,b})_{\pi_{a,b}(p)})^{\text{red}}$  is contained entirely in some irreducible component of $Y_{a,b}$. Restrict $U$ if necessary so that it is disjoint from all irreducible components of $Y_{a, b}$ except $S$. We may thus assume that $((Y_{a,b})_{\pi_{a,b}(p)})^{\text{red}} \subset S$ and $\dim (Y_{a,b})_{\pi_{a,b}(p)} = \dim S_{\pi_{a,b}(p)} = d$ . 
 
 If $r \in \Fbar^{\times}$, then $f_{a,b}(p)$ and $f_{a,b}(r\cdot p)$ are isomorphic $(\varphi, \Gamma)$-modules. Therefore, $r\cdot p \in V_{a,b}(\Fbar)$ is also a point of $Y_{a,b}$.
 Thus, there exists a monomorphism
\begin{align}\label{eqn:generic-fiber-dim} \Gm \times \Gm \times \kappa(p)\to Y_{a,b} \times_{\cX} \kappa(p) \cong (Y_{a,b})_{f_{a,b}(p)}\end{align} induced by \Cref{eqn:generic-action-pr-map}, which is a bijection on finite type points as in \Cref{finite-type-points-fiber}. The existence of \Cref{eqn:generic-fiber-dim} uses the fact that the finite type points of the domain land in $(Y_{a,b})_{f_{a,b}(p)}$ and the domain, being reduced, is the Zariski closure of its finite type points. The finite type points of $(Y_{a,b})_{f_{a,b}(p)}$ are evidently contained in the finite type points of $((Y_{a,b})_{\pi_{a,b}(p)})_{f_{a,b}(p)}$, which in turn are contained in $S_{f_{a,b}(p)}$. Therefore the map in \Cref{eqn:generic-fiber-dim} can be factored as
$$\Gm \times \Gm \times \kappa(\mathfrak{a})\to S_{f_{a,b}(p)} \hookrightarrow (Y_{a,b})_{f_{a,b}(p)}$$ where both the arrows are monomorphisms and bijective on finite type points. Using \cite[\href{https://stacks.math.columbia.edu/tag/0DS4}{Tag 0DS4}]{stacks-project}, we infer that $\dim S_{f_{a,b}(p)} = 2$.

 
We conclude from \Cref{eqn:fiber-dim-eqn-1,eqn:fiber-dim-eqn-2} that  
\begin{align*}\dim \overline{f_{a,b}(S)} = d - (2 - \dim \overline{\pi_{a,b}(S)}) \leq d\end{align*}
and $$\dim \overline{f_{a,b}(S)} = d \iff \dim \overline{\pi_{a,b}(S)} = 2.$$
Since $\pi_{a,b}(|Y_{a,b}|) = |B_{a,b}|$ as noted earlier, for at least one of the finitely many irreducible components of $Y_{a,b}$, the scheme theoretic image under $\pi_{a,b}$ has dimension $2$, showing that its scheme theoretic image under $f_{a,b}$ has dimension $d$.

Next, we show that there exists exactly one irreducible component of $Y_{a,b}$ whose scheme theoretic image under $\pi_{a,b}$ has dimension $2$.
For $i \in \{1, 2\}$, suppose $S^i$ is an irreducible component of $Y_{a,b}$ so that $\overline{\pi_{a,b}(S)}$ has dimension $2$. Let $U^i$ be the dense open subscheme of $S^i$ obtained by taking the complement of all other irreducible components of $Y_{a,b}$. Therefore, $\overline{\pi_{a,b}(U^i)} = \overline{\pi_{a,b}(S^i)} = B_{a,b}$. Since $\pi_{a,b}(U^i)$ is constructible, it contains a dense open $W^i$ of $B_{a,b}$. Let $W = W^1 \cap W^2$. If $q$ is a closed point of $W$, the  irreducible scheme $((Y_{a,b})_q)^{\red}$ is contained entirely in at least one irreducible component of $Y_{a,b}$. But since for each $i$, $(U^i)_q$ is a non-empty locally closed subscheme of $(Y_{a,b})_q$ that is disjoint from all irreducible components of $Y_{a,b}$ apart from $S^i$, $S^1$ must be the same as $S^2$. 

Thus, there exists exactly one irreducible component of dimension $d$ in the scheme-theoretic image of $Y_{a,b}$ under $f_{a,b}$. \end{proof}



\begin{proposition}\label{exclude-contr-of-triv}

        Suppose all maximal families of representations contained in $|\mathcal{E}|$ have dimension $\leq d$. 
        Then, for each $b \in \cA$, the scheme theoretic images of $Y_{1,b}$ and $Y_{\epsilon, b}$ under  $f_{1, b}$ and $f_{\epsilon, b}$ respectively have dimension strictly less than $d$.
\end{proposition}
\begin{proof}The proof follows the same ideas as the proof of \Cref{geometry-lemma-dimension-max}. The reduction in dimension for the scheme-theoretic image of $f_{1, b}|_{Y_{1,b}}$ arises from the fact that the codomain of $\pi_{1, b}: Y_{1,b} \to \Gm$ has dimension $1$ less than the codomain of $\pi_{a,b}$ along with the fact that the fibers of the map $f_{1,b}|_{Y_{1,b}}$ are $3$-dimensional, using \Cref{eqn:borel-action}.

The reduction in dimension for the scheme-theoretic image of $f_{\epsilon, b}|_{Y_{\epsilon, b}}$ arises from the fact that the codomain of the map $\pi_{\epsilon, b}: V_{\epsilon,b} \to \Gm$ has dimension $1$ less than the codomain of $\pi_{a,b}$. 

We leave the details to the reader.
\end{proof}

\begin{proof}[Proof of \Cref{geometry-proposition-dimension-intersection}] By construction, each finite type point of $\cE$ corresponding to a reducible representation is in the image of one of the (finitely many) $f_{a, b}|_{Y_{a,b}}$, $f_{1, b}|_{Y_{1,b}}$ and  $f_{\epsilon, b}|_{Y_{\epsilon, b}}$ maps. Further, each finite type point of $\mathcal{E}$ corresponding to an irreducible representation is in the image of a map $f_{\rhobar}$ where $\rhobar$ is an irreducible representation. Since there are finitely many irreducible representations of $G_K$ with $\Fbar$-coefficients upto unramified twists, one can take a finite set of irreducible representations $I_{\cE}$ so that the underlying topological space of the union of the scheme-theoretic images of $f_{\rhobar}$ for $\rhobar \in I_{\cE}$ contains all the irreducible representations in $|\cE|$.
By an application of \cite[\href{https://stacks.math.columbia.edu/tag/0DS4}{Tag 0DS4}]{stacks-project} and using \Cref{finite-type-points-fiber} to calculate fiber dimension, the maps $f_{\rhobar}$ have irreducible scheme-theoretic images of dimension $0$. Thus, $\dim \cE$ is the maximum of the dimensions of the scheme-theoretic images of the various $f_{a, b}|_{Y_{a,b}}$, $f_{1, b}|_{Y_{1,b}}$ and  $f_{\epsilon, b}|_{Y_{\epsilon, b}}$ maps. 

The first statement then follows from \Cref{geometry-lemma-dimension-max,exclude-contr-of-triv}, which also show that each $d$-dimensional family contains precisely one $d$-dimensional component in its closure.

Now assume that $\mathcal{Y}$ is a $d$-dimensional irreducible component of $\mathcal{E}$ contained in the closure of two $d$-dimensional maximal families $\mathcal{F}_{\chi_1, \chi_2}$ and $\mathcal{F}_{\chi'_1, \chi'_2}$. Then there exist unique $a,a',b, b' \in A$ so that $(\psi_a, \psi_b)$ and $(\psi_{a'}, \psi_{b'})$ are unramified twists of $(\chi_1, \chi_2)$ and $(\chi'_1, \chi'_2)$ respectively. Therefore, $\mathcal{Y}$ is the scheme-theoretic image of an irreducible component $S$ of $Y_{a,b}$ as well as an irreducible component $S'$ of $Y_{a', b'}$. 

Since the images of $|S|$ and of $|S'|$ are constructible sets dense in $|\mathcal{Y}|$, there exists a dense open $V$ of $|\mathcal{Y}|$ contained in the image of both $|S|$ and $|S'|$. If $(a,b) \neq (a', b')$, then this means that $(a,b) = (b',a')$ and $V$ contains only split reducible representations. Therefore, families of split reducible representations are dense in $\mathcal{Y}$. 

Consider the closed subschemes $W$ and $W'$ of $Y_{a, b}$ and $Y_{a',b'}$ respectively obtained by setting the extension classes to $0$. Clearly, $V$ is in the images of both $|W|$ and $|W'|$, and therefore $\cY$ is the scheme-theoretic image of both $W$ and $W'$ under $f_{a,b}$ and $f_{a',b'}$ respectively. For each $\overline{q} \in B_{a,b}(\Fbar)$ and $\bar{q}' \in B_{a',b'}(\Fbar)$, $W_{\bar{q}}(\Fbar)$ and $W'_{\bar{q}'}(\Fbar)$ are $\Fbar$-vector space of dimension $0$. Thus, arguing exactly as in the proof of \Cref{geometry-lemma-dimension-max}, for each closed point $q$ of $B_{a,b}$ and $q'$ of $B_{a',b'}$, $\dim W_{q} = \dim W'_{q'} = 0$, and as a result, the scheme theoretic image of $W$ under $f_{a,b}$, which is also the scheme-theoretic image of $W'$ under $f_{a',b'}$, has dimension $0$. If $d>0$, this is an impossibility as $\cY$ has dimension $d$. This settles the second statement.

Finally, if $d=0$, then since $\cE$ doesn't contain a family of dimension $>0$, $W=Y_{a,b}$ and $W' = Y_{a',b'}$. Maximality of $\cF_{\chi_1, \chi_2}$ and $\cF_{\chi'_1, \chi'_2}$ forces them to be the same family, that is uniquely identified by taking the isomorphism class of $\rhobar|_{I_K}$ where $\rhobar$ is any (reducible) point in this family. Next, suppose $\rhobar'$ and $\rhobar''$ are irreducible $G_{K}$-representations defined over finite field extensions $\F'$ and $\F''$ of $\F$ respectively. Suppose, further, that the $0$-dimensional, irreducible scheme-theoretic image $\cY$ of $f_{\rhobar'}$ is also the same as that of $f_{\rhobar''}$. As before, since $f_{\rhobar'}(|\mathbb{G}_{m, \F'}|)$ and $f_{\rhobar''}(|\mathbb{G}_{m, \F''}|)$ is constructible, there exists a dense open $V$ in $|\cY|$ which is contained in $f_{\rhobar'}(|\mathbb{G}_{m, \F'}|) \cap f_{\rhobar''}(|\mathbb{G}_{m, \F''}|)$. Therefore, $\rhobar'$ is an unramified twist of $\rhobar''$, or equivalently, $\rhobar' \sim \rhobar''$. On the other hand, by construction, if $\rhobar' \sim \rhobar''$, then $f_{\rhobar'}(|\mathbb{G}_{m, \F'}|) = f_{\rhobar''}(|\mathbb{G}_{m, \F''}|)$, showing that $f_{\rhobar'}$ and $f_{\rhobar''}$ have the same scheme-theoretic images. Using essentially the same argument, for any $a,b \in \cA$ and any irreducible $\rhobar$, the scheme-theoretic image of any irreducible component of $Y_{a,b}$ under $f_{a,b}$ is necessarily not isomorphic to that of $f_{\rhobar}$. This settles the third statement.
\end{proof}

\section{Serre weights of \texorpdfstring{$G_K$}{GK}-representations}\label{serre-weights}
In this section, the field of coefficients for all $I_K$ and $G_K$-representations will be $\Fbar$, and we will omit the base field from the notations. If $\chi_1, \chi_2$ are $G_K$-characters, then for any Serre weight $\sigma$, explicit recipes to compute the set $L_{\sigma}(\chi_1, \chi_2)$ (defined in \Cref{defn:L-sigma}) were provided by \cite{ddr} in the unramified setting and later, by \cite{steinmetz} in the general setting. We will recall these recipes 
and compute conditions for $L_{\sigma}(\chi_1, \chi_2)$ to equal the entire space of extensions or be of codimension $1$ in it. When $K=\Qp$, we will recall a criterion from \cite{gls} that determines when an irreducible representation $\rhobar \colon G_K \to GL_2(\Fbar)$ satisfies $\sigma \in W(\rhobar)$.

\subsection{Serre weights of reducible representations}
For this subsection, let $$\chi_1, \chi_2 \colon I_K \to \Fbar^{\times}$$ be fixed characters and let $\ts$ be a fixed Serre weight. Suppose $\chi_1$ and $\chi_2$ admit extensions to $G_K$-characters. This is equivalent to each of $\chi_1$ and $\chi_2$ being an exponent of $\omega_i$ for some (any) $i \in \Z/f\Z$. Fix $G_K$-characters $\tilde{\chi}_1$ and $\tilde{\chi}_2$ satisfying $\tilde{\chi}_1|_{I_K} = \chi_1$, $\tilde{\chi}_2|_{I_K} = \chi_2$. 

We note first that using Tate local duality and Euler characteristic formula, 
\begin{align*}
    \dim \Ext^{1}_{G_K}(\tilde{\chi}_1, \tilde{\chi}_2) =\begin{cases}
    [K:\Qp] &\text{ if } \tilde{\chi}_2^{-1}\tilde{\chi}_1 \not\in \{1, \epsilon\},\\
    [K:\Qp] + 1 &\text{ otherwise}
    \end{cases}
\end{align*} where $\epsilon$ denotes the mod $p$ cyclotomic character.  

Let $\rhobar \colon G_K \to \GL_2(\Fbar)$ be a $G_K$-representation. Then $\ts \in W(\rhobar)$ if and only if  $\sigma_{\mathbf{t}, \mathbf{s}}^{\vee}$ is a Serre weight of $\rhobar^{\vee}$ in the sense of \cite{gls,ddr,steinmetz}. Therefore, our definition of $L_{\ts}(\tilde{\chi}_1, \tilde{\chi}_2) \subset \Ext^{1}_{G_K}(\tilde{\chi}_1, \tilde{\chi}_2)$ is not the same as the space denoted $L_{\ts}$ in \cite{ddr} where the data of the $G_K$-characters $\tilde{\chi}_1, \tilde{\chi}_2$ is suppressed from the notation, or the space denoted $L_{\ts}(\tilde{\chi}_1, \tilde{\chi}_2)$ in \cite{steinmetz}. To make a distinction, we will denote the aforementioned objects appearing in \cite{ddr,steinmetz} by $L_{\ts}^{*}(\tilde{\chi}_1, \tilde{\chi}_2)$ instead. We have
\begin{align}\label{notation-dual}
\mathfrak{a} \in L_{\ts}(\tilde{\chi}_2^{-1}, \tilde{\chi}_1^{-1}) \iff \mathfrak{a}^{\vee} \in L_{\ts}^{*}(\tilde{\chi}_1, \tilde{\chi}_2) \end{align}
where the extension class $\mathfrak{a}^{\vee}$ is obtained by taking the equivalence class of the dual of a short exact sequence representing $\mathfrak{a}$.

\begin{lemma}\label{necessary-but-not-suff}
The set $L^{*}_{\ts}(\tilde{\chi}_1, \tilde{\chi}_2)$ is non-empty if and only if there exists a subset $J \subset \Z/f\Z$, and for each $i \in \Z/f\Z$ there exists $x_i \in [0, e-1]$ such that	
	\begin{align}\label{screening-1}
		\chi_1 &= \prod_{i \in \Z/f\Z} \omega_i^{t_i} \prod_{i \in J} \omega_i^{s_i + 1 + x_i} \prod_{i \not\in J} \omega_i^{x_i}, \text{ and}\\
	\label{screening-2}
		\chi_2 &= \prod_{i \in \Z/f\Z} \omega_i^{t_i} \prod_{i \in J} \omega_i^{e - 1 - x_i} \prod_{i \not\in J} \omega_i^{s_i + e - x_i}
	\end{align}
\end{lemma}
\begin{proof}
This is the content of the remark following \cite[Defn.~4.1.4]{gls}.
\end{proof}

\subsubsection*{A basis of \texorpdfstring{$L^{*}_{\ts}(\tilde{\chi}_1, \tilde{\chi}_2)$ using \cite{steinmetz}}{Lsigma}.}
Now, suppose $L^{*}_{\ts}(\tilde{\chi}_1, \tilde{\chi}_2) \neq \varnothing$. What follows is the recipe in \cite{steinmetz} for writing down a basis of $L^{*}_{\ts}(\tilde{\chi}_1, \tilde{\chi}_2)$, although we change the symbols used. Refer to \textit{loc. cit.} for justifications of the statements and more details. We first write 
\begin{align}\label{defn:m_i}
\chi_2 = \prod_{i \in \Z/f\Z} \omega_i^{t_i} \prod_{i \in \Z/f\Z} \omega_i^{m_i}\end{align} with each $m_i \in [0, p-1]$ and not all $m_i$ equal to $p-1$. This requirement uniquely determines $m_i$ for each $i$.
Let $\mathcal{S}$ be the necessarily non-empty set of $f$-tuples of non-negative integers $(a_0, \dots, a_{f-1})$ indexed over $\Z/f\Z$ and
satisfying $$\chi_2 = \prod_{i \in \Z/f\Z} \omega_i^{t_i} \prod_{i \in \Z/f\Z} \omega_i^{a_i}$$ and $a_i \in [0, e-1] \cup [s_i + 1, s_i + e]$ for all $i$. 

For $i \neq f-1$, let $v_{i}$ be the $f$-tuple $(0, \dots, 0, -1, p, 0, \dots, 0)$ with $-1$ in $i$ position, $p$ in $i+1$ position and $0$ everywhere else. Let $v_{f-1}$ be $(p, 0, \dots, 0, -1)$. Then there exists a subset $A \subset \Z/f\Z$ such that 
\begin{align}\label{ramified-condition-J}
	(m_0, \dots, m_{f-1}) + \sum_{i \in A} v_i \in \mathcal{S}
\end{align}
There exists a minimal subset of $\Z/f\Z$ that is contained in any other subset $A \subset \Z/f\Z$ satisfying \Cref{ramified-condition-J}. 

\begin{definition}\label{maximal-model}
	Define $A_{\min}$ to be the minimal $A$ satisfying \Cref{ramified-condition-J}.
\end{definition}

\begin{definition}\label{definition-y-z}
	Given $(m_0, \dots, m_{f-1})$ and $A_{\min}$ as above. Define
	\begin{align*}
		(y_{0}, \dots, y_{f-1}) &:= (m_{0}, \dots, m_{f-1}) + \sum_{i \in A_{\min}} v_i \in \mathcal{S}
	\end{align*}
 and for each $i \in \Z/f\Z$,
 \begin{align*}
		z_i &:= s_i + e - y_i.
 \end{align*}
\end{definition}

\begin{remark}\label{chi-terms-y-z} We have 
\begin{align*}
    \chi_1 &=\prod_{i \in \Z/f\Z} \omega_i^{t_i} \prod_{i \in \Z/f\Z} \omega_i^{z_i} ,\\
    \chi_2 &=\prod_{i \in \Z/f\Z} \omega_i^{t_i} \prod_{i \in \Z/f\Z} \omega_i^{y_i} , \: \text{ and} \\
    \chi_2^{-1} \chi_1&= \prod_{i \in \Z/f\Z} \omega_i^{z_i - y_i}.
\end{align*}
\end{remark}

\begin{definition} \label{definition-interval}
Define
\begin{align*}
    \cI_i := \begin{cases}
    [0, z_i - 1] &\text{ if } y_i \geq s_i + 1, \\
    \{y_i\} \cup [s_i + 1, z_i - 1] &\text{ if } y_i < s_i + 1.
    \end{cases}
\end{align*}
 Here, an interval $[a, b] \subset \Z$ is interpreted to be empty if $b < a$. 
\end{definition}
\begin{remark}
    When $e=1$, $\mathcal{I}_i = \{0\}$ if $y_i = 0$ and $\mathcal{I}_i = \varnothing$ if $y_i = s_i + 1$. 
\end{remark}

\begin{remark} \label{intervals} 
Note that for each $i \in \Z/f\Z$, $|\cI_i| \leq e$. Furthermore,
\begin{align*}
|\cI_i| = e &\iff y_i = 0. \\
|\cI_i| = e- 1 &\iff \begin{cases}
y_i = s_i + 1, &\quad \text{ if } e=1,\\
y_i \in \{1, s_i + 1\} &\quad \text{ if } e>1. 
 \end{cases}
\end{align*}

\end{remark}

    


\begin{definition}\label{definition-xi}
	For each $i \in \Z/f\Z$, let
	\begin{align}
		\lambda_i &:= \sum_{j=0}^{f-1}p^{j}(z_{i-j} - y_{i-j}) \in \Z \\
		\xi_i &:= (p^f - 1)z_i + \lambda_i \in \Z
\end{align}
\end{definition}


Suppose $\chi_2^{-1} \chi_1 = \prod_{i \in \Z/f\Z} \omega_i^{a_i}$ for $a_i \in [1, p]$ and not all $a_i = p$. 
We will extend the indices of the $a_i$ from $\Z/f\Z$ to all of $\mathbb{Z}$ by setting  $a_{j} = a_{j'}$ if $j \equiv j' \mod f$. 
We call the tuple $(a_{0}, \dots, a_{f-1})$ the \textit{tame signature} of $\chi_2^{-1} \chi_1$. The arithmetic Frobenius element $\mathit{Frob}_k \in \Gal(k/\mathbb{F}_p)$ acts on $f$-tuples $(a_{0}, \dots, a_{f-1})$ via 
\begin{align}
	\mathit{Frob}_k \cdot (a_{0}, \dots, a_{f-1}) = (a_1, \dots, a_{f-1}, a_0).
\end{align}

 Let $f'$ be the cardinality of the orbit of $(a_{0}, \dots, a_{f-1})$ under the action of $\langle\mathit{Frob}\rangle$, and let $f'' := f/f'$.

Let $i \in \Z/f\Z$ and $u \in \cI_i$. Let $\nu$ be the $p$-adic valuation of $\xi_i - u(p^f - 1)$. We let 
\begin{align}\label{defn:alpha-i-u}
\alpha(i, u) := (n, \kappa) \in \mathbb{Z} \times [0, f''-1]
\end{align} where 
\begin{align}\label{defn:m}
			n = \frac{\xi_i - u(p^f - 1)}{p^{\nu}}
		\end{align}
and $\kappa$ is defined as follows: if $i_n \in [0, f'-1]$ is such that $i_n \equiv i - \nu \mod f',$ then 
 \begin{align}\label{formula-for-kappa}\kappa \equiv \frac{i - \nu - i_n}{f'} \mod f''.\end{align}
 
 \begin{definition}\label{definition-J-V}
We define 
$$J^{\mathrm{AH}}_{\sigma_{\mathbf{t}, \mathbf{s}}}(\chi_1, \chi_2):=\{\alpha(i, u) \: | \: i \in \Z/f\Z, u \in \cI_i\}.$$ 

	
		
\end{definition}
We emphasize that $J^{\mathrm{AH}}_{\sigma_{\mathbf{t}, \mathbf{s}}}(\chi_1, \chi_2)$ is defined if and only if \Cref{screening-1,screening-2} are true.

By \cite[Thm.~3.16]{steinmetz}, each $\alpha \in J^{\mathrm{AH}}_{\sigma_{\mathbf{t}, \mathbf{s}}}(\chi_1, \chi_2)$ specifies a unique element $c_{\alpha}$ of a basis of $L^{*}_{\sigma_{\mathbf{t}, \mathbf{s}}}(\tilde{\chi}_1, \tilde{\chi}_2)$. 
The space $L^{*}_{\sigma_{\mathbf{t}, \mathbf{s}}}(\tilde{\chi}_1, \tilde{\chi}_2)$ is the span of 
$$\{c_{\alpha} \:|\: \alpha \in J^{\mathrm{AH}}_{\sigma_{\mathbf{t}, \mathbf{s}}}(\chi_1, \chi_2)\}$$
together with additional, distinguished basis elements 
$c_{\text{un}}$ if $\tilde{\chi}_2^{-1} \tilde{\chi}_1 =1$, and
$c_{\text{tr}}$ if $\tilde{\chi}_2^{-1} \tilde{\chi}_1 = \epsilon$, $\prod_{i \in \Z/f\Z} \omega^{-t_i} \otimes\chi_2 = 1$ and $s_i = p-1$ for all $i$.
Thus, whenever $\tilde{\chi}_2^{-1} \tilde{\chi}_1 \not\in \{1, \epsilon\}$,
\begin{align}\label{dim-L}
	\dim L^{*}_{\sigma_{\mathbf{t}, \mathbf{s}}}(\tilde{\chi}_1, \tilde{\chi}_2)
	= |J^{\mathrm{AH}}_{\sigma_{\mathbf{t}, \mathbf{s}}}(\chi_1, \chi_2)| 
 = \sum_{i \in \Z/f\Z} |\mathcal{I}_i|. 
\end{align}

Let $\sigma_{\mathbf{t'}, \mathbf{s'}}$ be another Serre weight, not isomorphic to $\ts$. If $L^{*}_{\sigma_{\mathbf{t'}, \mathbf{s'}}}(\tilde{\chi}_1, \tilde{\chi}_2)$ is also non-empty, then
a basis for $L^{*}_{\sigma_{\mathbf{t}, \mathbf{s}}}(\tilde{\chi}_1, \tilde{\chi}_2) \cap L^{*}_{\sigma_{\mathbf{t'}, \mathbf{s'}}}(\tilde{\chi}_1, \tilde{\chi}_2)$ 
is given by
\begin{align}\label{basis-intersection}
\{c_{\alpha} \:|\: \alpha \in J^{\mathrm{AH}}_{\sigma_{\mathbf{t}, \mathbf{s}}}(\chi_1, \chi_2) \cap J^{\mathrm{AH}}_{\sigma_{\mathbf{t'}, \mathbf{s'}}}(\chi_1, \chi_2)\}
\end{align}
together with $c_{\text{un}}$ if $\tilde{\chi}_2^{-1} \tilde{\chi}_1$ is trivial. The basis element $c_{\text{tr}}$ doesn't come into the picture because if it was contained in $L^{*}_{\sigma_{\mathbf{t}, \mathbf{s}}}(\tilde{\chi}_1, \tilde{\chi}_2) \cap L^{*}_{\sigma_{\mathbf{t'}, \mathbf{s'}}}(\tilde{\chi}_1, \tilde{\chi}_2)$, then $\ts$ and $\tss$ would be forced to be isomorphic, a contradiction.
In particular, whenever $\tilde{\chi}_2^{-1} \tilde{\chi}_1 \not=1$,
\begin{align}\label{dim-L-cap}
	\dim L^{*}_{\sigma_{\mathbf{t}, \mathbf{s}}}(\tilde{\chi}_1, \tilde{\chi}_2) \cap L^{*}_{\sigma_{\mathbf{t'}, \mathbf{s'}}}(\tilde{\chi}_1, \tilde{\chi}_2)
	= |J^{\mathrm{AH}}_{\sigma_{\mathbf{t}, \mathbf{s}}}(\chi_1, \chi_2) \cap J^{\mathrm{AH}}_{\tss}(\chi_1, \chi_2)|. 
\end{align}

\begin{proposition}\label{ramified-full-dim}
	Suppose $\ts$ is non-Steinberg. Then $$|J^{\mathrm{AH}}_{\sigma_{\mathbf{t}, \mathbf{s}}}(\chi_1, \chi_2)| = ef$$ if and only if 
 \begin{align}\label{eqn:highest-wt-criterion}
 \chi_1 &= \prod_{i \in \Z/f\Z} \omega_i^{s_i + e} \prod_{i \in \Z/f\Z} \omega_i^{t_i}, \text{ and} \nonumber\\ \chi_2 &= \prod_{i \in \Z/f\Z} \omega_i^{t_i}.  \end{align}
\end{proposition}
\begin{proof}
By \Cref{dim-L} and \Cref{intervals}, $|J^{\mathrm{AH}}_{\sigma_{\mathbf{t}, \mathbf{s}}}(\chi_1, \chi_2)| = ef$ if and only if for each $i \in \Z/f\Z$, $$|\cI_i| = e \iff y_i = 0.$$ Thus, \Cref{chi-terms-y-z} shows that $|J^{\mathrm{AH}}_{\sigma_{\mathbf{t}, \mathbf{s}}}(\chi_1, \chi_2)| = ef$ implies \Cref{eqn:highest-wt-criterion}. On the other hand, if $\chi_1$ and $\chi_2$ satisfy \Cref{eqn:highest-wt-criterion}, then $A_{\mathrm{min}}$ is seen to be the empty set and we obtain $y_i =0$ for each $i$ as desired.
\end{proof}

\begin{definition}\label{defn:highest-wt}
If $\ts$ is non-Steinberg and the conditions in \Cref{eqn:highest-wt-criterion} are satisfied, we say that $\sigma_{\mathbf{t}, {\mathbf{s}}}$ is the \textit{highest weight} for the pair $(\chi_1, \chi_2)$.
\end{definition}
Note that because of the requirement that the highest weight be non-Steinberg, the highest weight for any pair of $G_K$-characters is well defined. Moreover, the highest weight only depends on the restrictions of the $G_K$-characters to $I_K$.


\begin{proposition}\label{ramified-codim-1}
	Suppose $\ts$ is non-Steinberg. Then $$|J^{\mathrm{AH}}_{\ts}(\chi_1, \chi_2)| = ef-1$$ if and only if one of the following conditions is satisfied:
	\begin{enumerate}
		\item \label{ramified-codim-1-a} There exists $j \in \Z/f\Z$ such that \begin{align}\label{eqn:ram-codim-1-a}
  \chi_1 = \omega_j ^{e-1} \prod_{i \neq j} \omega_i^{s_i + e} \prod_{i \in \Z/f\Z} \omega_i^{t_i}, \qquad \chi_2 = \omega_j^{s_j + 1} \prod_{i \in \Z/f\Z} \omega_i^{t_i},
  \end{align} and
  \begin{align}\label{eqn:s-ram-codim-1-a}
      s_j \leq \begin{cases}
          p-3 \text{ if } f=1,\\
          p-2 \text{ if } f>1.
      \end{cases}
  \end{align}
		\item \label{unramified-codim-1} $e=1, f>1$ and there exists $j \in \Z/f\Z$ such that \Cref{eqn:ram-codim-1-a} holds, 
  $s_{j-1} \neq 0$ and $s_j = p-1$.
		\item \label{ramified-codim-1-b} $e>1$ and there exists $j \in \Z/f\Z$ such that 
  \begin{align}\label{eqn:ram-codim-1-b}
      \chi_1 = \omega_j ^{s_j + e-1} \prod_{i \neq j} \omega_i^{s_i + e} \prod_{i \in \Z/f\Z} \omega_i^{t_i}, \qquad
      \chi_2 &= \omega_j \prod_{i \in \Z/f\Z} \omega_i^{t_i}
  \end{align} and $s_j \neq 0$.
	\end{enumerate}

	
\end{proposition}

\begin{proof}
By \Cref{dim-L} and \Cref{intervals}, $|J^{\mathrm{AH}}_{\sigma_{\mathbf{t}, \mathbf{s}}}(\chi_1, \chi_2)| = ef-1$ if and only there exists $j \in \Z/fZ$ so that
\begin{align*}
    |\cI_i| = \begin{cases}
        e &\text{ if } i \neq j, \\
        e-1 &\text{ if } i=j.
    \end{cases}
\end{align*} which happens if and only if
either
\begin{align}\label{eqn:y_1}
    y_i = \begin{cases}
        0 &\text{ if } i \neq j,\\
        s_i + 1 &\text{ if } i = j,
    \end{cases}
\end{align} or, 
\begin{align}\label{eqn:y_2}
    &e >1 \text{ and} \nonumber \\ 
    &y_i = \begin{cases}
        0 &\text{ if } i \neq j,\\
        1 &\text{ if } i = j.
    \end{cases}
\end{align}

Note that \Cref{eqn:y_1} implies \Cref{eqn:ram-codim-1-a}, but the conditions on $e$ and $s_i$'s in \Cref{ramified-codim-1-a,unramified-codim-1} are not obvious. Similarly, when \Cref{eqn:y_2} is true, then \Cref{eqn:ram-codim-1-b} is satisfied, but the condition $s_j \neq 0$ in \Cref{ramified-codim-1-b} is not clear. 

To see how the extra conditions arise, suppose first that $\chi_1$ and $\chi_2$ satisfy \Cref{eqn:ram-codim-1-a}. We need to compute $m_i$'s, $A_{\mathrm{min}}$ and $y_i$'s and check when $y_i$'s satisfy either \Cref{eqn:y_1} or \Cref{eqn:y_2}. We consider three cases:

\begin{itemize}
    \item If $s_j \leq p-2$ with $f>1$ or $s_j \leq p-3$ with $f=1$, then $m_i = 0$ for $i \neq j$ and $m_j = s_j + 1$. Thus, the tuple $(m_0, \dots, m_{f-1})$ is in $\cS$ implying $A_{\mathrm{min}} = \varnothing$. This implies further that $y_i = m_i$ for each $i$ and \Cref{eqn:y_1} holds. 
    \item If $s_j = p-1$ (with $f>1$, since $\ts$ is assumed to be non-Steinberg), then $m_i = 0$ for $i \neq j-1$ and $m_{j-1} = 1$. We can obtain the desired values of $y_i$'s only if $(m_0, \dots, m_{f-1}) \not\in \cS$ which happens if and only if $e=1$ and $s_{j-1} \neq 0$. When that happens, $A_{\mathrm{min}} = \{j-1\}$ and we obtain \Cref{eqn:y_1}.
    \item If $s_j = p-2$ with $f=1$, then $m_j = 0$. Therefore, $A_{\mathrm{min}} = \varnothing$ and $y_j = 0$, contradicting both \Cref{eqn:y_1} and \Cref{eqn:y_2}.
\end{itemize}


Now, suppose $e>1$ and $\chi_1$ and $\chi_2$ satisfy \Cref{eqn:ram-codim-1-b} but not \Cref{eqn:ram-codim-1-a}. Thus, we assume $s_j \neq 0$ simply to avoid redundancy with \Cref{ramified-codim-1-a,unramified-codim-1}. In this setting, $m_i = 0$ if $i \neq j$ and $m_j = 1$. Evidently, $(m_0, \dots, m_j) \in \cS$. Therefore, $y_i = m_i$ for each $i$ and \Cref{eqn:y_2} holds, but \Cref{eqn:y_1} is false. 
\end{proof}




\subsubsection*{A basis of \texorpdfstring{$L^{*}_{\ts}(\tilde{\chi}_1, \tilde{\chi}_2)$ using \cite{ddr}}{Lsigma}.}
When $e=1$, \cite{ddr} provides an algorithm to specify a basis of $L^{*}_{\sigma_{\mathbf{t}, \mathbf{s}}}(\tilde{\chi}_1, \tilde{\chi}_2)$, assuming it to be non-empty. It will be significantly easier to use this algorithm for some of the calculations in the unramified case, and so, we recall some essentials of this algorithm starting with a few definitions, modified slightly where necessary to be consistent with our conventions. As before, let $(a_0, ..., a_{f-1})$ denote the tame signature of $\chi_2^{-1} \chi_1$ with the indices of the $a_i$ extended from $\Z/f\Z$ to all of $\Z$ in the usual way. Consider the function
$$\delta: \mathbb{Z} \to \mathbb{Z}$$ defined as follows: For $j \in \Z$, $\delta(j) = j$ unless there exists a (necessarily unique) integer $i > j$ such that for each $k \in [j, i-2]$ (taken to be the empty interval if $j =i-1$]), $a_k = p-1$ and $a_{i-1} = p$.  When this exception happens, $\delta(j) = i$. The function $\delta$ induces a function $\mathbb{Z}/f\mathbb{Z} \to \mathbb{Z}/f\mathbb{Z}$, also denoted by $\delta$. 

Next, consider a function $\mu$ on subsets of $\Z/f\Z$ defined as follows: 
If $\delta(J) \subset J$, $\mu(J) := J$. Else, choose some $i_1 \in \Z$ so that $\bar{i_1} \in \delta(J) \smallsetminus J$ and let $j_1$ be the largest integer such that $j_1 < i_1$, $\bar{j_1} \in J$ and $\delta(j_1) = i_1$. If $J = \{\bar{j_1}, \dots, \bar{j_r}\}$ with $j_1 > j_2 > ... > j_r > j_1 - f$, define $i_{\kappa}$ for $\kappa \in [2, r]$ inductively as follows:
 \begin{align*}
     i_{\kappa} = \begin{cases}
     \delta(j_{\kappa}) \: &\text{ if } i_{\kappa-1} > \delta(j_{\kappa}), \\
     j_{\kappa} \: &\text{ otherwise. } 
     \end{cases}
 \end{align*}

Then $\mu(J) := \{\bar{i_1}, \dots, \bar{i_r}\}$. Note, in particular, that the definitions of the functions $\delta$ and $\mu$ depend on $\chi_2^{-1}\chi_1$.


Next, let
\begin{align}\label{J-max}
    J_{\mathrm{max}} := \{i \in \mathbb{Z}/f\mathbb{Z} \; | \; y_i = 0\}\end{align} where $y_i$ are as defined in \Cref{definition-y-z}.

When $e=1$, each $\alpha \in \mu(J_{\mathrm{max}})$ specifies a distinguished basis element of $L^{*}_{\sigma_{\mathbf{t}, \mathbf{s}}}(\tilde{\chi}_1, \tilde{\chi}_2)$ and if $\tilde{\chi}_2^{-1} \tilde{\chi}_1 \not\in \{1, \epsilon\}$, then by \cite[Conj.~7.2]{ddr} and the main result of \cite{calegari2017explicit}, these distinguished basis elements labelled by elements of $\mu(J_{\mathrm{max}})$ give a complete basis of $L^{*}_{\sigma_{\mathbf{t}, \mathbf{s}}}(\tilde{\chi}_1, \tilde{\chi}_2)$. Thus, assuming $e=1$,
$$|J^{\text{AH}}_{\ts}(\chi_1, \chi_2)| = |J_{\mathrm{max}}| = |\mu(J_{\mathrm{max}})|.$$ Further, the constructions of \cite{ddr,steinmetz} imply that if $\tss$ is another non-Steinberg Serre weight not isomorphic to $\ts$, then
\begin{align}\label{dim-cap-unram}
|J^{\text{AH}}_{\ts}(\chi_1, \chi_2) \cap J^{\text{AH}}_{\tss}(\chi_1, \chi_2)| = |\mu(J_{\mathrm{max}}) \cap \mu(J'_{\mathrm{max}})|
\end{align}
where $J'_{\mathrm{max}}$ is the analogue of $J_{\mathrm{max}}$ for $\tss$.

\subsection{Serre weights of irreducible \texorpdfstring{$G_{\Qp}$}{GQp}-representations}
When $K=\Qp$, we will need results on Serre weights of irreducible representations.
Let $\Q_{p^2}$ be the quadratic unramified extension of $\Qp$ in its fixed algebraic closure with residue field $\F_{p^2}$. Fix a $(p^{2} - 1)$-th root of $p$ in the algebraic closure of $\Qp$, denoted $\sqrt[p^2 - 1]{p}$, and consider the map \begin{align*}
    I_K &\longrightarrow \F_{p^2}^{\times} \\
    g &\longmapsto \frac{g(\sqrt[p^2 - 1]{p})}{\sqrt[p^2 - 1]{p}} \mod p.
\end{align*}
On composing with the two distinct embeddings of $\F_{p^2}$ into $\Fbar$, we obtain two characters $\eta_1, \eta_2: I_K \to \Fbar^{\times}$. These are the fundamental characters of level $2$.


\begin{lemma}\label{wt-irred}
    Let $\rhobar \colon G_{\Qp} \to \GL_2(\Fbar)$ be an irreducible representation. Then $\ts \in W(\rhobar^{\vee})$ if and only if
    $$\rhobar|_{I_K} \cong \eta_1^{t+s+1} \eta_2^{t} \oplus \eta_1^{t} \eta_2^{t+s+1}$$
\end{lemma} 
\begin{proof}
    This is immediate from \cite[Defn.~4.1.4, Thm.~4.1.6]{gls} keeping in mind that $\ts \in W(\rhobar^{\vee})$ if and only if $\ts \in W^{\text{explicit}}(\rhobar) = W^{\text{cris}}(\rhobar)$ in the sense of \cite{gls}.
\end{proof}

\begin{proposition}\label{irred-intersection}
Let $K = \Qp$ and let $\ts$ and $\tss$ be non-isomorphic non-Steinberg Serre weights. The stack $\cX_{\ts} \cap \cX_{\tss}$ has finite type points corresponding to irreducible representations if and only if $s' = p-1-s$ and $t' \equiv t + s \mod p-1$.
\end{proposition}
\begin{proof}
By \cite[Thm.~1.2]{caraiani2022geometric} and \Cref{wt-irred}, we need to determine when $$\eta_1^{t+s+1} \eta_2^{t} \oplus \eta_1^{t} \eta_2^{t+s+1}= \eta_1^{t'+s'+1} \eta_2^{t'} \oplus \eta_1^{t'} \eta_2^{t'+s'+1}.$$
Taking ratios of the direct summands, we have two possibilities. Either 
\begin{align*}
    &&\eta_1^{s+1} \eta_2^{-(s+1)}\quad &= \quad \eta_1^{s'+1} \eta_2^{-(s'+1)} \\
    &\iff &\eta_1^{s} \eta_2^{s'} \quad &= \quad \eta_1^{s'} \eta_2^{s}  
\end{align*} or 
\begin{align*}
    &&\eta_1^{s+1} \eta_2^{-(s'+1)} \quad &= \quad \eta_1^{-(s'+1)} \eta_2^{s+1} \\
    &\iff &\eta_1^{s} \eta_2^{p-s'-1} \quad &= \quad \eta_1^{p-s'-1} \eta_2^{s}.
\end{align*}
Since $s, s' < p-1$, the first case forces $\ts \cong \tss$, a contradiction. The second case implies the desired relationship between $(t,s)$ and $(t', s')$.
\end{proof}
\begin{remark}
The criterion in the statement of \Cref{irred-intersection} is the same as that in \Cref{ext-gamma}\ref{ext-gamma-totally-ramified}\ref {ext-gamma-totally-ramified-2}.
\end{remark}

\subsection{Types of codimension one intersections}\label{geometry-to-ext}
We combine the results of \Cref{geometry} and this section to obtain the following theorem, true for arbitrary $K$.
\begin{theorem}\label{criteria-codim-arbitrary}
    Let $\sigma_{\mathbf{t}, \mathbf{s}}$ and $\sigma_{\mathbf{t'}, \mathbf{s'}}$ be non-isomorphic, non-Steinberg Serre weights. Then $\cX_{\ts} \cap \cX_{\tss}$ has dimension $[K:\Qp] - 1$ if and only if one of the following is satisfied.
    \begin{enumerate}
        \item\label{codim-1-generic} There exist $I_{K}$-characters $\chi_1$ and $\chi_2$ such that
        \begin{align}\label{criteria-codim}
        |J^{\mathrm{AH}}_{\sigma_{\mathbf{t}, \mathbf{s}}}(\chi_1, \chi_2) \cap J^{\mathrm{AH}}_{\sigma_{\mathbf{t'}, \mathbf{s'}}}(\chi_1, \chi_2)| = ef-1.\end{align}
        \item\label{codim-1-irred} $K = \Qp$ and
        \begin{align}\label{eqn:irred-intersection}
        s' =p-1-s, \qquad t'\equiv t+ s \mod p-1.\end{align}
    \end{enumerate}
Let $n$ be the number of irreducible components in $\cX_{\ts} \cap \cX_{\tss}$ of dimension $[K:\Qp]-1$. If $K \neq \Qp$, then $n$ is the number of isomorphism classes of ordered pairs of $I_K$-characters $(\chi_1, \chi_2)$ satisfying \Cref{codim-1-generic}. If $K = \Qp$, then $n$ equals the number of isomorphism classes of $I_K$-representations $\chi_1 \oplus \chi_2$, where $\chi_1$ and $\chi_2$ satisfy \Cref{codim-1-generic}, plus $1$ if \Cref{eqn:irred-intersection} holds.

\end{theorem}
\begin{proof}
   The existence of $\chi_1$ and $\chi_2$ satisfying \Cref{codim-1-generic} 
   is equivalent to the existence of $G_K$-characters $\tilde{\chi}_1$ and $\tilde{\chi}_2$, so that for all $a, b \in \Fbar^{\times}$ such that
$\tilde{\chi}_2^{-1}\tilde{\chi}_1 \otimes \mathrm{ur}_{b^{-1}a} \neq 1$, we have 
   \begin{align*}
   &\dim \; L_{\ts}(\tilde{\chi}^{-1}_2 \otimes \mathrm{ur}_{b^{-1}}, \tilde{\chi}^{-1}_1\otimes \mathrm{ur}_{a^{-1}}) \cap L_{\tss}(\tilde{\chi}^{-1}_2\otimes \mathrm{ur}_{b^{-1}}, \tilde{\chi}^{-1}_1\otimes \mathrm{ur}_{a^{-1}}) \\
    \stackrel{\Cref{notation-dual}}{=} \: &\dim \; L^{*}_{\ts}(\tilde{\chi}_1\otimes \mathrm{ur}_{a}, \tilde{\chi}_2\otimes \mathrm{ur}_{b}) \cap L^{*}_{\tss}(\tilde{\chi}_1\otimes \mathrm{ur}_{a}, \tilde{\chi}_2\otimes \mathrm{ur}_{b}) \\\stackrel{\Cref{dim-L-cap}}{=} \: &ef-1.
   \end{align*}
   Thus, by \cite[Thm.~1.2]{caraiani2022geometric}, \Cref{codim-1-generic} is equivalent to the existence of a maximal family of representations $\cF_{\tilde{\chi}^{-1}_2, \tilde{\chi}^{-1}_1}$ of dimension $[K:\Qp]-1$ inside $|\cX_{\ts} \cap \cX_{\tss}|$. Clearly, when $K \neq \Qp$, the number of pairwise separated maximal families of dimension $[K:\Qp]-1$ in $|\cX_{\ts} \cap \cX_{\tss}|$ equals the number of isomorphism classes of ordered pairs of $I_K$-characters $(\chi_1, \chi_2)$ satisfying \Cref{codim-1-generic}.
    Further, when $K=\Qp$, \Cref{codim-1-irred} is equivalent to the existence of points corresponding to irreducible representations in $|\cX_{\ts} \cap \cX_{\tss}|$ by \Cref{irred-intersection}. All such points corresponding to irreducible representations have the same isomorphism class after restricting to $I_K$ by \Cref{wt-irred}. An application of \Cref{geometry-proposition-dimension-intersection} finishes the proof.
\end{proof}

\begin{definition}
    We say that an unordered pair of non-isomorphic non-Steinberg Serre weights $\sigma_{\mathbf{t}, \mathbf{s}}$ and $\sigma_{\mathbf{t'}, \mathbf{s'}}$ have a ``type I" intersection witnessed by an ordered pair of $I_K$-characters $(\chi_1, \chi_2)$ if after exchanging $\ts$ and $\tss$ if necessary, $\ts$ is the highest weight for $(\chi_1, \chi_2)$ while $|J^{\mathrm{AH}}_{\sigma_{\mathbf{t'}, \mathbf{s'}}}(\chi_1, \chi_2)| = ef-1$.
\end{definition}
\begin{proposition}\label{lem:type-1-codim}
If $\sigma_{\mathbf{t}, \mathbf{s}}$ and $\sigma_{\mathbf{t'}, \mathbf{s'}}$ have a type I intersection, then $|J^{\mathrm{AH}}_{\sigma_{\mathbf{t}, \mathbf{s}}}(\chi_1, \chi_2) \cap J^{\mathrm{AH}}_{\tss}(\chi_1, \chi_2)| = ef-1$.
\end{proposition}
\begin{proof} After exchanging $\ts$ and $\tss$ if necessary, we have $$|J^{\mathrm{AH}}_{\sigma_{\mathbf{t}, \mathbf{s}}}(\chi_1, \chi_2)| = ef.$$
Therefore, there exist $G_K$-characters $\tilde{\chi}_1$ and $\tilde{\chi}_2$, so that for all $a, b \in \Fbar^{\times}$ such that $\tilde{\chi}_2^{-1}\tilde{\chi}_1 \otimes \mathrm{ur}_{b^{-1}a} \not\in \{1,\epsilon\}$,
$L^{*}_{\ts}(\tilde{\chi}_1\otimes \mathrm{ur}_{a}, \tilde{\chi}_2\otimes \mathrm{ur}_{b}) = \Ext^{1}_{G_K}(\tilde{\chi}_1\otimes \mathrm{ur}_{a}, \tilde{\chi}_2\otimes \mathrm{ur}_{b})$. As a result,
\begin{align*}
&|J^{\mathrm{AH}}_{\sigma_{\mathbf{t}, \mathbf{s}}}(\chi_1, \chi_2) \cap J^{\mathrm{AH}}_{\tss}(\chi_1, \chi_2)| \\
=\quad &\dim L^{*}_{\ts}(\tilde{\chi}_1\otimes \mathrm{ur}_{a}, \tilde{\chi}_2\otimes \mathrm{ur}_{b}) \cap L^{*}_{\tss}(\tilde{\chi}_1\otimes \mathrm{ur}_{a}, \tilde{\chi}_2\otimes \mathrm{ur}_{b}) \\
=\quad &\dim L^{*}_{\tss}(\tilde{\chi}_1\otimes \mathrm{ur}_{a}, \tilde{\chi}_2\otimes \mathrm{ur}_{b}) \\
=\quad &|J^{\mathrm{AH}}_{\tss}(\chi_1, \chi_2)| = ef-1.
\end{align*}\\[-2.5\baselineskip] \end{proof}

\begin{proposition}
The Serre weights $\ts$ and $\tss$ can have a type I intersection witnessed by at most two pairs of $I_K$-characters.
\end{proposition}
\begin{proof}
    We have two possibilities, either $|J^{\mathrm{AH}}_{\ts}(\chi_1, \chi_2)| = ef$ or $|J^{\mathrm{AH}}_{\tss}(\chi_1, \chi_2)| = ef$. In both cases, \Cref{ramified-full-dim} determines $(\chi_1, \chi_2)$ completely.
\end{proof}
\begin{definition}
     We say that an unordered pair of non-isomorphic non-Steinberg Serre weights $\sigma_{\mathbf{t}, \mathbf{s}}$ and $\sigma_{\mathbf{t'}, \mathbf{s'}}$ have a ``type II" intersection witnessed by an ordered pair of $I_K$-characters $(\chi_1, \chi_2)$ if $J^{\mathrm{AH}}_{\sigma_{\mathbf{t}, \mathbf{s}}}(\chi_1, \chi_2) = J^{\mathrm{AH}}_{\sigma_{\mathbf{t'}, \mathbf{s'}}}(\chi_1, \chi_2)$ of cardinality $ef-1$.
\end{definition}

Note that by definition of types and \Cref{criteria-codim-arbitrary}, the stack $\cX_{\ts} \cap \cX_{\tss}$ has dimension $[K:\Qp] - 1$ if and only if $\ts$ and $\tss$ have either a type I intersection or a type II intersection.
\begin{proposition}\label{prop:triple}
    Suppose $\sigma_{\mathbf{t}, \mathbf{s}}$ and $\sigma_{\mathbf{t'}, \mathbf{s'}}$ have a type II intersection witnessed by $(\chi_1, \chi_2)$. Then there exists a unique non-Steinberg Serre weight $\tsss$ such that $\ts$ and $\tsss$, as well as $\tss$ and $\tsss$, have a type I intersection witnessed by $(\chi_1, \chi_2)$. Moreover, $\cX_{\ts} \cap \cX_{\tss} \cap \cX_{\tsss}$ has dimension $[K:\Qp] -1$. On the other hand, if there exist pairwise non-isomorphic, non-Steinberg Serre weights $\ts$, $\tss$ and $\tsss$ satisfying $\dim \cX_{\ts} \cap \cX_{\tss} \cap \cX_{\tsss} = [K:\Qp]-1$, then at least two of $\{\ts, \tss, \tsss\}$ have a type II intersection between them.
\end{proposition}
\begin{proof}
Suppose first that $\sigma_{\mathbf{t}, \mathbf{s}}$ and $\sigma_{\mathbf{t'}, \mathbf{s'}}$ have a type II intersection witnessed by $(\chi_1, \chi_2)$. Existence of $\tsss$ such that $\ts$ and $\tsss$, as well as $\tss$ and $\tsss$, have a type I intersection witnessed by $(\chi_1, \chi_2)$ is equivalent to $\tsss$ being the highest weight for $(\chi_1, \chi_2)$. This in turn is equivalent to finding a solution for the equations
\begin{align*}
 \chi_1 &= \prod_{i \in \Z/f\Z} \omega_i^{s''_i + e} \prod_{i \in \Z/f\Z} \omega_i^{t''_i}, \text{ and} \nonumber\\ \chi_2 &= \prod_{i \in \Z/f\Z} \omega_i^{t''_i}.  \end{align*}
 where each $s''_i \in [0, p-1]$ and not all $s''_i$ equal $p-1$. Such a solution clearly exists and is unique, settling the first part of the statement. 

    By the proof of \Cref{criteria-codim-arbitrary}, $|\cX_{\ts} \cap \cX_{\tss}|$ contains a maximal family $\cF_{\tilde{\chi}_2^{-1}, \tilde{\chi}_1^{-1}}$ of dimension $ef-1$. Using \cite[Thm.~1.2]{caraiani2022geometric} and similar arguments as in \Cref{lem:type-1-codim}, we find that $\cF_{\tilde{\chi}_2^{-1}, \tilde{\chi}_1^{-1}}$ is contained in $|\cX_{\tsss}|$. An application of \Cref{geometry-proposition-dimension-intersection} finishes the proof of the first half of the statement.

    Now suppose $\ts$, $\tss$ and $\tsss$ are pairwise non-isomorphic, non-Steinberg Serre weights such that $\dim \cX_{\ts} \cap \cX_{\tss} \cap \cX_{\tsss} = [K:\Qp]-1$. Then by \Cref{geometry-proposition-dimension-intersection}, there exists at least one maximal family of representations $\cF_{\psi_2^{-1}, \psi_1^{-1}}$ of dimension $[K:\Qp]-1$ in $|\cX_{\ts} \cap \cX_{\tss} \cap \cX_{\tsss}|$, where $\psi_1, \psi_2$ are $G_K$-characters. Therefore, for any pair of non-isomorphic Serre weights in $\{\ts, \tss, \tss\}$, $(\psi_1|_{I_K}, \psi_2|_{I_K})$ witnesses either a type I or a type II intersection. As there can be only one highest weight for $(\psi_1|_{I_K}, \psi_2|_{I_K})$, at least one of these is a type II intersection.
\end{proof}

\section{Type I intersections}\label{type1}
In this section and the next, we will explicitly determine the existence of type I and type II intersections between pairs of Serre weights. First, we introduce some more notation and conventions.
When comparing $f$-tuples $\mathbf{s}$ and $\mathbf{s'}$, we will often and without additional comment take the indices of the components of $\mathbf{s}$ and $\mathbf{s}'$ to be integers valued in $[0, f-1]$ instead of element of the set $\Z/f\Z$, by identifying each element of $\Z/f\Z$ with its representative in $[0, f-1]$. The chosen set of indices, $[0, f-1]$ versus $\Z/f\Z$, will be clear from the context. While making a comparison, we will only state the values of $s_{i}$ and $s'_{i}$ that have specific constraints or are potentially different from each other. If we say $(\dots, s_i, \dots) = (\dots, \in [a, b], \dots)$ or $(\dots, s_i, \dots) = (\dots, c, \dots)$, we mean respectively that $s_i$ can take any value $\in [a, b]$ or $s_i =c$. 
If for some $a, b \in [0, p-1]$ and $j \in [-1, f-2]$, a subsequence $(s_i, \dots, s_{i+j})$ is set to equal $(a, \dots, a)$ and $(s'_i, \dots, s'_{i+j})$ is set to equal $(b, \dots, b)$, then this means that for each $k \in [i, i+j]$ $s_k = a$ and $s'_k = b$. Here, the interval $[i, i+j]$ is to be interpreted to be the empty set if $j = -1$. 
If the values of $s_{i}$ and $s'_{i}$ are not specified, then we assume that $s_i$ can take any value in $[0, p-1]$ and $s'_i = s_i$. 
Similar notational norms apply with the roles of $s_i$ and $s'_i$ interchanged. When $f=1$, we will omit the indices $i$ in the tuples, and also write the character $\omega_i$ as simply $\omega$. 


We will retain the symbols $m_i$, $y_i$, $z_i$, $\mathcal{I}_i$, $\lambda_i$, $\xi_i$, $\alpha$ and $J_{\mathrm{max}}$ as defined in \Cref{defn:m_i}, \Cref{definition-y-z,definition-interval,definition-xi}, \Cref{defn:alpha-i-u,J-max} for $\sigma_{\mathbf{t}, \mathbf{s}}$, and will replace them respectively with $m'_i, y'_i$, $z'_i$, $\mathcal{I}'_i$, $\lambda'_i$, $\xi'_i$, $\alpha'$ and $J'_{\mathrm{max}}$ for $\sigma_{\mathbf{t'}, \mathbf{s'}}$ and with $m''_i, y''_i$, $z''_i$, $\mathcal{I}''_i$, $\lambda''_i$, $\xi''_i$, $\alpha''$ and $J''_{\mathrm{max}}$ for $\sigma_{\mathbf{t''}, \mathbf{s''}}$.

Let $\ts$ and $\tss$ be non-isomorphic, non-Steinberg Serre weights. We will now compute criteria for existence of an ordered pair of $I_K$-characters $(\chi_1, \chi_2)$ with highest weight $\ts$ witnessing a type I intersection between $\ts$ and $\tss$.
As we saw in \Cref{ramified-codim-1}, $|J^{\mathrm{AH}}_{\sigma_{\mathbf{t'}, \mathbf{s'}}}(\chi_1, \chi_2)| = ef-1$ can happen in one of three ways. 

\subsection{Type I intersections when \texorpdfstring{$f=1$}{f equals 1}}

\subsubsection{Case 1: \texorpdfstring{
$\tss$ satisfies \Cref{ramified-codim-1}\ref{ramified-codim-1-a}}{J-tss via i or ii}.}


\begin{proposition}\label{prop:tot-ram-type1a}
    Let $f=1$. There exists an ordered pair of $I_K$-characters $(\chi_1, \chi_2)$ with highest weight $\ts$ and the statement of \Cref{ramified-codim-1}\ref{ramified-codim-1-a} holding true for $\tss$ (in particular, $|J^{\mathrm{AH}}_{\tss}(\chi_1, \chi_2)| = ef-1$) if and only if 
    \begin{align}\label{eqn:tot-ram-type1a}
    s' = p-s-3 \qquad \text{ and } \qquad t' \equiv t - s' -1 \mod p-1.\end{align}
    This in turn happens if and only if \Cref{ext-gamma}\ref{ext-gamma-totally-ramified}\ref{ext-gamma-totally-ramified-1} is true.
\end{proposition}
\begin{proof}
Existence of such a pair $(\chi_1, \chi_2)$ is equivalent to
\begin{align*}
  \chi_1 = \omega ^{e-1 + t'} = \omega^{s+e+t}, \qquad \chi_2 = \omega^{s' + 1 + t'} = \omega^{t} \qquad \text{ and } \qquad s'\leq p-3.
  \end{align*}
Comparing two different ways of writing $\chi_2^{-1} \chi_1$, we obtain the following equivalences mod $p-1$:
\begin{align*} 
	s+e &\equiv e-2-s' \\
	\iff s' &\equiv - 2 - s \equiv p-3-s.\\
\end{align*}
If $s = p-2$, then $s' = p-2$, a contradiction. Therefore, $s \leq p-3$ and $s=p-s'-3$. Comparing the expressions for $\chi_2$, we get the desired expression for $t'$. 
\end{proof}


\begin{remark}\label{rem:tot-ram-type1a-number}
The criterion in \Cref{eqn:tot-ram-type1a} is symmetric in $\ts$ and $\tss$. Therefore, when it holds, there exist two distinct ordered pairs of $I_K$-characters, $(\chi_1, \chi_2)$ with highest weight $\ts$ and $(\chi'_1, \chi'_2)$ with highest weight $\tss$, both witnessing type I intersections between $\ts$ and $\tss$. 

When $K = \Qp$, we get $$\chi_1 = \omega^{s+1+t} = \omega^{t'} = \chi'_2 \quad \text{ and } \quad \chi_2 = \omega^{t} = \omega^{s'+t'+1} = \chi_1.$$ Therefore, $\chi_1 \oplus \chi_2 \cong \chi'_1 \oplus \chi'_2$.

\end{remark}


\subsubsection{Case 2: \texorpdfstring{
$\tss$ satisfies \Cref{ramified-codim-1}\ref{ramified-codim-1-b}}{J-tss via iii}.}
\begin{proposition}\label{prop:tot-ram-type1c}
    Let $f=1$ and $e>1$. There exists an ordered pair of $I_K$-characters $(\chi_1, \chi_2)$ with highest weight $\ts$ and  the statement of \Cref{ramified-codim-1}\ref{ramified-codim-1-b} holding true for $\tss$ (in particular, $|J^{\mathrm{AH}}_{\tss}(\chi_1, \chi_2)| = ef-1$) if and only if 
    \begin{align}\label{eqn:tot-ram-type1c}
        s' &= \begin{cases}
            s+2 &\text{ and } s<p-3, \text{ or }\\
            1 &\text{ and } s=p-2,
        \end{cases} \nonumber \\
        t' &\equiv t-1 \mod p-1
    \end{align} 
    If the above holds, then if $s<p-3$, then \Cref{H1-criteria} is true, and if $s=p-2$, then \Cref{ext-gamma}\ref{ext-gamma-totally-ramified}\ref{ext-gamma-totally-ramified-2} is true.
\end{proposition}
\begin{proof}
Existence of such a pair $(\chi_1, \chi_2)$ is equivalent to 
\begin{align*}
  \chi_1 = \omega ^{s' + e-1+ t'} = \omega^{s+e+t}, \qquad \chi_2 = \omega^{1 + t'} =\omega^{t} \qquad \text{ and } \qquad s'\neq 0.
  \end{align*}

Comparing the two ways of writing $\chi_2^{-1} \chi_1$, we obtain the following equivalence mod $p-1$
\begin{align*} 
	s+e &\equiv e-2+s'  \\
	\iff s' &\equiv s + 2.
\end{align*}
Thus, the requirement that $s' \neq 0$ gives the conditions on $s$ and $s'$ as desired. The expression for $t'$ follows from the expressions for $\chi_2$.
\end{proof}
\begin{remark}\label{rem:tot-ram-type1c-number}
    The criteria in \Cref{eqn:tot-ram-type1c} are not symmetric when  $p>3$ and symmetric when $p=3$. Therefore, if these criteria hold and $p>3$ (resp. $p=3$), then there exists exactly one (resp. two) ordered pair(s) of $I_K$-characters witnessing a type 1 intersection between $\ts$ and $\tss$.
\end{remark}

\subsection{Type I intersections when \texorpdfstring{$f>1$}{f is bigger than 1}}

\subsubsection{Case 1: \texorpdfstring{
$\tss$ satisfies \Cref{ramified-codim-1}\ref{ramified-codim-1-a} or \ref{unramified-codim-1}}{J-tss via i or ii}.} 



\begin{proposition}\label{proposition-unram-type-1}
     Let $f>1$. There exists an ordered pair of $I_K$-characters $(\chi_1, \chi_2)$ with highest weight $\ts$ and the statement of \Cref{ramified-codim-1}\ref{ramified-codim-1-a} or \ref{unramified-codim-1} holding true for $\tss$ (in particular, $|J^{\mathrm{AH}}_{\tss}(\chi_1, \chi_2)| = ef-1$) if and only if, after translating the indices for the components of $\mathbf{t}$, $\mathbf{s}$, $\mathbf{t}'$ and $\mathbf{t}'$ by adding a common fixed element of $\Z/f\Z$ if necessary, $\mathbf{s}$ and $\mathbf{s}'$ satisfy one of the conditions in \Cref{table:unram-type-1} and \begin{align} \label{type-1-case-1-det-twist}
	\sum_{i \in \Z/f\Z} p^{f-1-i} t'_i  \equiv -1-s'_{f-1} + \sum_{i \in \Z/f\Z} p^{f-1-i} t_i \mod p^f-1.
\end{align}
     
\end{proposition}
\begin{proof}
Existence of such a pair $(\chi_1, \chi_2)$ is equivalent to 
 \begin{align*}
 \chi_1 \qquad & = && \omega_{j} ^{e-1} \prod_{i \neq j} \omega_i^{s'_i + e} \prod_{i \in \Z/f\Z} \omega_i^{t'_i} &= \qquad &\prod_{i \in \Z/f\Z} \omega_i^{s_i + e} \prod_{i \in \Z/f\Z} \omega_i^{t_i}, \\
 \chi_2 \qquad &= && \omega_{j}^{s'_j + 1} \prod_{i \in \Z/f\Z} \omega_i^{t'_i} &= \qquad &\prod_{i \in \Z/f\Z} \omega_i^{t_i},
 \end{align*}
  and either $s'_{j} \leq p-2$, or $e=1$, $s'_{j-1} \neq 0$ and $s'_{j} = p-1$, for some $j \in \Z/f\Z$.
Translate all indices by a fixed element if necessary so that $j=f-1$. Comparing the two ways of writing $\chi_2^{-1} \chi_1$, we obtain the following equivalences mod $p^{f}-1$:
\begin{align*} 
	&\sum_{i \in \Z/f\Z} p^{f-1-i} (s_i + e) &\equiv &\hspace{0.2cm} e-2-s'_{f-1} + \sum_{i= 0}^{f-2} p^{f-1-i} (s'_i + e)\\
	\iff &\sum_{j \in \Z/f\Z} p^{f-1-i} s_i &\equiv &\hspace{0.2cm} -2-s'_{f-1} + \sum_{i=0}^{f-2} p^{f-1-i} s'_i \\
	&& \equiv &\hspace{0.2cm} p -s'_{f-1} - 2 + p (s'_{f-2} -1) + \sum_{i =0}^{f-3} p^{f-1-i} s'_i
\end{align*}

Therefore, for any fixed $\mathbf{s}'$, we can find a unique $\mathbf{s}$ satisfying the equation above. The uniqueness of $\mathbf{s}$ follows from the requirement that each $s_i \in [0, p-1]$, and not all $s_i$ can be $p-1$. Taking all possible (non-Steinberg) values of $\mathbf{s}'$, we obtain the pairs $\mathbf{s}, \mathbf{s}'$ in \Cref{table:unram-type-1}. 
Comparing the two ways of writing $\chi_2$, we obtain mod $p^{f}-1$,
\begin{align*}
	\sum_{i \in \Z/f\Z} p^{f-1-i} t'_i  + (s'_{f-1} + 1) &\equiv \sum_{j \in \Z/f\Z} p^{f-1-i} t_i \\
	\iff \sum_{i \in \Z/f\Z} p^{f-1-i} t'_i  &\equiv -1-s'_{f-1} + \sum_{i \in \Z/f\Z} p^{f-1-i} t_i
 \end{align*} finishing the proof.
\end{proof}

\vspace{0.5cm}

\subsubsection{Case 2: \texorpdfstring{
$\tss$ satisfies \Cref{ramified-codim-1}\ref{ramified-codim-1-b}}{J-tss via iii}.} 

\begin{proposition}\label{summary-ramified-codim-1-b}
     Let $f>1$, $e>1$. There exists an ordered pair of $I_K$-characters $(\chi_1, \chi_2)$ with highest weight $\ts$ and the statement of \Cref{ramified-codim-1}\ref{ramified-codim-1-b} holding true for $\tss$ (in particular, $|J^{\mathrm{AH}}_{\tss}(\chi_1, \chi_2)| = ef-1$) if and only if, after translating the indices for the components of $\mathbf{t}$, $\mathbf{s}$, $\mathbf{t}'$ and $\mathbf{t}'$ by adding a common fixed element of $\Z/f\Z$ if necessary, $\mathbf{s}$ and $\mathbf{s}'$ satisfy one of the conditions in \Cref{table:ram-type-1b} and \begin{align} \label{ramified-codim-1-b-chi2}
	\sum_{i \in \Z/f\Z} p^{f-1-i} t'_i  \equiv -1 + \sum_{i \in \Z/f\Z} p^{f-1-i} t_i \mod p^f-1.
\end{align}
\end{proposition}
\begin{proof}
Existence of such a pair $(\chi_1, \chi_2)$ is equivalent to 
\begin{align*}
 \chi_1 \qquad & = && \omega_{j} ^{s'_j + e-1} \prod_{i \neq j} \omega_i^{s'_i + e} \prod_{i \in \Z/f\Z} \omega_i^{t'_i} &= \qquad &\prod_{i \in \Z/f\Z} \omega_i^{s_i + e} \prod_{i \in \Z/f\Z} \omega_i^{t_i}, \\
 \chi_2 \qquad &= && \omega_{j}^{1} \prod_{i \in \Z/f\Z} \omega_i^{t'_i} &= \qquad &\prod_{i \in \Z/f\Z} \omega_i^{t_i},
 \end{align*}
  and $s'_{j} \neq 0$, for some $j \in \Z/f\Z$.
Translate all indices by a fixed element if necessary so that $j=f-1$. 
Writing $\chi_2^{-1}\chi_1$ in two different ways, we have the following equivalences mod $p^{f}-1$:
\begin{align}\label{ramified-codim-1-b-chi} 
	&\sum_{i \in \Z/f\Z} p^{f-1-i} (s_i + e) &\equiv &&e-2 + s'_{f-1} + \sum_{i \neq f-1} p^{f-1-i} (s'_i + e) \nonumber \\ 
	\iff 
	&\sum_{i \in \Z/f\Z} p^{f-1-i} s_i &\equiv &&s'_{f-1} - 2  + \sum_{i \neq f-1} p^{f-1-i} s'_i
\end{align}
We plug in all possible (non-Steinberg) values of $\mathbf{s}'$ so that $s'_{f-1} \neq 0$ and compute the unique (non-Steinberg) value of $\mathbf{s}$ so that the above equation is satisfied obtaining the pairs in \Cref{table:ram-type-1b}. 

Comparing the two ways of writing $\chi_2$, we obtain \Cref{ramified-codim-1-b-chi2}.
\end{proof}

\begin{remark}\label{type-1-summary}
Let $f>1$ and let $\ts$ and $\tss$ be weakly regular. Then there exists a type I intersection between $\ts$ and $\tss$ if and only if after translating all indices by a fixed element of $\Z/f\Z$ and exchanging $\ts$ and $\tss$ if necessary, either  \Cref{table:unram-type-1}\ref{item-unram-generic} and \Cref{type-1-case-1-det-twist} are true, or \Cref{table:ram-type-1b}\ref{ram-type-1b-i} and \Cref{ramified-codim-1-b-chi2} are true. Either way, the criteria are not symmetric, and there can exist at most one pair of $I_K$-characters $(\chi_1, \chi_2)$ witnessing a type I intersection between $\ts$ and $\tss$.

\end{remark}

\section{Type II intersections}\label{type2}

Let $\tss$ and $\tsss$ be non-isomorphic, non-Steinberg Serre weights. In this section, we will compute criteria for existence of a pair of characters $(\chi_1, \chi_2)$ witnessing a type II intersection between Serre weights $\sigma_{\mathbf{t'}, \mathbf{s'}}$ and $\sigma_{\mathbf{t''}, \mathbf{s''}}$. When such a pair $(\chi_1, \chi_2)$ exists, we will denote by $\ts$ its highest weight. When comparing tuples  $\mathbf{s}$, $\mathbf{s}'$ and $\mathbf{s}''$, we will continue to use the notational conventions we set at the beginning of \Cref{type1}. 

A type II intersection between $\tss$ and $\tsss$ witnessed by $(\chi_1, \chi_2)$ implies the existence of type I intersections between $\ts$ and $\tss$, and between $\ts$ and $\tsss$. Therefore, each of $\tss$ and $\tsss$ must satisfy one of the three conditions in \Cref{ramified-codim-1}, giving rise to different cases as we will see. 

\subsection{Type II intersections when \texorpdfstring{$f=1$}{f equals 1}}
\subsubsection{Case 1: \texorpdfstring{$\tss$ and $\tsss$ both satisfy \Cref{ramified-codim-1}\ref{ramified-codim-1-a} or both satisfy \Cref{ramified-codim-1}\ref{ramified-codim-1-b}}{1a or 1b}}\label{typeII-f-1-Case1} 
If $\tss$ and $\tsss$ both satisfy \Cref{ramified-codim-1}\ref{ramified-codim-1-a} or $e>1$ and both satisfy \Cref{ramified-codim-1}\ref{ramified-codim-1-b}, then one verifies immediately that $\tss \cong \tsss$, a contradiction.
 
\subsubsection{Case 2: \texorpdfstring{$\tss$ satisfies \Cref{ramified-codim-1}\ref{ramified-codim-1-a} and $\tsss$ satisfies \Cref{ramified-codim-1}\ref{ramified-codim-1-b}}{1a-and-1b}} 


\begin{proposition}\label{type-II-tot-ram}
Let $f=1$. There exists a type II intersection between $\tss$ and $\tsss$ witnessed by an ordered pair of $I_K$-characters $(\chi_1, \chi_2)$ with highest weight $\ts$ if and only if $e>1$, $p>3$ and after exchanging $\tss$ and $\tsss$ if necessary, $\tss$ satisfies \Cref{ramified-codim-1}\ref{ramified-codim-1-a}, $\tsss$ satisfies \Cref{ramified-codim-1}\ref{ramified-codim-1-b},
\begin{align}
\label{ramified-f=1-codim-1-s-s'}
&s' = p-3-s,  &&\qquad s'' = s+2,  &&\qquad s'' = p-1-s', \text{ and } \\
\label{ramified-f=1-codim-1-d-d'} &t' \equiv t + s + 1,  &&\qquad t'' \equiv -1 +t,  &&\qquad t'' \equiv t'+s' \mod p-1.
\end{align}
\end{proposition}
\begin{proof}
    
If there is to be a type II intersection witnessed by $(\chi_1, \chi_2)$, then by \Cref{typeII-f-1-Case1}, after exchanging $\tss$ and $\tsss$ if necessary, we must have $e>1$, $\tss$ satisfies \Cref{ramified-codim-1}\ref{ramified-codim-1-a} and $\tsss$ satisfies \Cref{ramified-codim-1}\ref{ramified-codim-1-b}. Therefore,
\begin{align*}
  &\chi_1 = \omega ^{e-1 + t'} = \omega ^{s'' + e-1+ t''} = \omega^{s+e + t}, \\
  &\chi_2 = \omega^{s' + 1 + t'} = \omega^{1 + t''} = \omega^{t}, \\
  &s' \leq p-3, \\
  &s'' \neq 0.
\end{align*}  
Comparing different ways of writing $\chi_2^{-1} \chi_1$ and $\chi_2$, we have the following equivalences mod $p-1$:	
\begin{align*}
	&e-2-s' \equiv e-2 + s'' \equiv s+e \\
	\iff
	&s'' \equiv p-1-s', \qquad
	s' \equiv p-3-s, \qquad
    s'' \equiv s+ 2.\\
    \\
	&s'+1+t' \equiv 1+t'' \equiv t \\
	\iff
 &t'' \equiv t' + s', \qquad
		t' \equiv t + s + 1, \qquad
		t'' \equiv -1 + t.\\
\end{align*}

If $s'=0$, then we find that $\tss \cong \tsss$, a contradiction. Therefore, $s' > 0$, \Cref{ramified-f=1-codim-1-s-s',ramified-f=1-codim-1-d-d'} are true and $s<p-3$. In particular, $p>3$.





Imposing the above conditions, we next calculate $y'$, $y''$, $z'$, $z''$, $\mathcal{I}'$, $\mathcal{I}''$, $\xi'$ and $\xi''$, and compare $J^{\mathrm{AH}}_{\sigma_{t', s'}}(\chi_1, \chi_2)$ with $J^{\mathrm{AH}}_{\sigma_{t'', s''}}(\chi_1, \chi_2)$.

By the proof of \Cref{ramified-codim-1}, $y'=s'+1$ and $z' = e-1$, while $y'' = 1$ and $z'' = s'' + e-1$. We obtain
\begin{align*}
    &\mathcal{I}' = [0, e-2] \\
    &\mathcal{I}'' = \{1\} \cup [s''+1, s''+e-2] \\
    &\xi' = (p-1)(e-1) + (e-2-s') \\
    &\xi''= (p-1)(s''+e-1) + (e-2+s'')
\end{align*}
We have equalities of sets
$$\{\xi' - u'(p-1) \: | \: u' \in \cI'\}= \{(p-1)v' + (e-2-s') \: | \: v'\in [1, e-1]\}$$ and
\begin{align*}
    &\{\xi'' - u''(p-1) \: | \: u'' \in \cI''\} \\
    &= \{(p-1)v'' + (e-2+s'') \: | \: v''\in [1, e-2] \cup \{s''+e-2\}\} \\
    &= \{(p-1)v'' + (e-2-s') + (p-1) \: | \: v''\in [1, e-2] \cup \{p-3-s'+e\}\} \\
    &= \{(p-1)v'' + (e-2-s') \: | \: v''\in [2, e-1] \cup \{p-2-s'+e\}\}.
\end{align*}

Hence, $J^{\mathrm{AH}}_{\sigma_{t', s'}}(\chi_1, \chi_2) = J^{\mathrm{AH}}_{\sigma_{t'', s''}}(\chi_1, \chi_2)$ if and only if for all $v' \in [1, e-1]$, there exists a $v'' \in [2, e-1] \cup \{p-2-s'+e\}$ such that:

\begin{align}\label{ramified-f=1-computing-J-V}
	\frac{(p-1)v' + (e-2-s')}{p^{\nu'}} = \frac{(p-1)v'' + (e-2-s')}{p^{\nu''}}
\end{align}
where $\nu'$ is the $p$-adic valuation of the numerator on L.H.S, while $\nu''$ is that of the numerator on R.H.S.

The only thing to check then is that \Cref{ramified-f=1-computing-J-V} holds for $v'=1$ and $v''=p-2-s'+e$. Plugging in,


\begin{align*}
	\mathrm{R.H.S.} &= \frac{(p-1)(p-2-s'+e) + e - 2 - s'}{p^{\nu''}} \\
	&=\frac{p(p-3-s'+e)}{p^{\nu'+1}} =\mathrm{L.H.S.}
\end{align*} finishing the proof.
\end{proof}

\begin{remark}\label{summary-totally-ramified-type-2}
If $s'' \leq p-3$, the relationship between $\tss$ and $\tsss$ is symmetric in \Cref{ramified-f=1-codim-1-s-s',ramified-f=1-codim-1-d-d'} but if $\tss$ and $\tsss$ are exchanged, $\ts$ changes. Therefore, if the criteria hold, there exist two distinct ordered pairs of $I_K$-characters witnessing a type II intersection between $\tss$ and $\tsss$. 

On the other hand, if $s'' = p-2$ and $s' = 1$, then the roles of $\tss$ and $\tsss$ cannot be exchanged since $p>3$. Therefore, if \Cref{ramified-f=1-codim-1-s-s',ramified-f=1-codim-1-d-d'} hold in this setting, then there exists exactly one pair $(\chi_1, \chi_2)$ witnessing a type II intersection between $\tss$ and $\tsss$.

We note also that \Cref{ramified-f=1-codim-1-s-s',ramified-f=1-codim-1-d-d'} are true for $\tss$ and $\tsss$ (for suitable $\ts$) if and only if they satisfy the conditions of \Cref{ext-gamma}\ref{ext-gamma-totally-ramified}\ref{ext-gamma-totally-ramified-2}.
\end{remark}

\subsection{Type II intersections when \texorpdfstring{$f > 1$, $e=1$}{f is bigger than 1, e equals 1}}

\begin{proposition}\label{type-2-unram}
Let $f>1$, $e=1$. Then there exists a type II intersection between $\tss$ and $\tsss$ 
if and only if after translating all indices by adding a fixed element of $\Z/f\Z$ and exchanging $\tss$ and $\tsss$ if necessary, $\mathbf{s}'$ and $\mathbf{s}''$ are as in one of the rows of \Cref{table:unram-type-2} and 
\begin{align}\label{det-twist-type2-unram}
     s'_{f-1} + 1 + \sum_{i \in \Z/f\Z} t'_i p^{f-1-i}  \equiv (s''_l +1) p^{f-1-l} + \sum_{i \in \Z/f\Z} t''_i p^{f-1-i}
\end{align} where $l \in \Z/f\Z$ is as described in each row of \Cref{table:unram-type-2}.
\end{proposition}
\begin{proof}
By \Cref{proposition-unram-type-1}, the existence of a pair of $I_K$-characters $(\chi_1, \chi_2)$ witnessing a type II intersection between $\tss$ and $\tsss$ is equivalent to the existence of $(\chi_1, \chi_2)$ with highest weight $\ts$ so that for some $j, l \in \Z/f\Z$,
    \begin{itemize}
        \item after translating all indices by adding $f-1-j$ so that $j$ translates to $f-1$,
        $\ts$ and $\tss$ satisfy \Cref{type-1-case-1-det-twist} and one of the conditions in \Cref{table:unram-type-1},
        \item after translating all indices in $\Z/f\Z$ by adding $f-1-l$ so that $l$ translates to $f-1$ 
        and replacing 
         $\tss$ with $\tsss$ in the statements in \Cref{type-1-case-1-det-twist} and \Cref{table:unram-type-1},
        $\ts$ and $\tsss$ satisfy \Cref{type-1-case-1-det-twist} and one of the conditions in \Cref{table:unram-type-1}, and
        \item $\mu(J'_{\mathrm{max}}) = \mu(J''_{\mathrm{max}})$.
    \end{itemize}  
   We may assume without loss of generality that $j=f-1$ (otherwise translate all indices by adding $f-1-j$ so that $j$ can be taken to be $f-1$). 

Thus, \Cref{det-twist-type2-unram} follows immediately from considering  \Cref{type-1-case-1-det-twist} and for the rest, we need to compare pairs of (same or distinct) rows in \Cref{table:unram-type-1} so that after translating the indices in one of the rows by a fixed element of $\Z/f\Z$ if necessary, the tuples $\bf{s}$ in the two rows match up and so do the entries in the last column of \Cref{table:unram-type-1}. 
The pairs of rows that can thus be compared so that we get non-isomorphic $\tss$ and $\tsss$ are \ref{item-unram-generic} and \ref{item-unram-7}, 
\ref{item-unram-generic} and \ref{item-unram-5}, \ref{item-unram-3} and \ref{item-unram-7}, \ref{item-unram-3} and \ref{item-unram-5}, \ref{item-unram-6} and \ref{item-unram-5}, \ref{item-unram-6} and \ref{item-unram-4}, \ref{item-unram-7} and \ref{item-unram-5}, \ref{item-unram-5} and \ref{item-unram-5}, and finally, \ref{item-unram-5} and \ref{item-unram-4}. Explicitly, $\bf{s}'$, $\bf{s}''$ and $\bf{s}$ are exactly those as in the rows of \Cref{table:unram-type-2}.
\end{proof}

\begin{remark}\label{rem:summary-regular-type2-unram}
    Note that in each pair of $\mathbf{s'}$ and $\mathbf{s''}$ featuring in \Cref{table:unram-type-2}, at least one is not weakly regular. Therefore, when $f>1$ and $e=1$, every type II intersection involves at least one Serre weight that is not weakly regular.
\end{remark}

\subsection{Type II intersections when \texorpdfstring{$f > 1$, $e>1$}{f and e are bigger than 1}}


We set up some extra notation that we will need in the following calculations. Let $(\chi_1, \chi_2)$ be an ordered pair of $I_K$-characters so that if $\tilde{\chi}_1$ and $\tilde{\chi}_2$ are extensions of $\chi_1$ and $\chi_2$ respectively to $G_K$, then $L^{\star}_{\tss}(\tilde{\chi}_1, \tilde{\chi}_2)$ and $L^{\star}_{\tsss}(\tilde{\chi}_1, \tilde{\chi}_2)$ are non-empty. In particular, for each $i \in \Z/f\Z$, $\xi'_i \equiv \lambda'_i \equiv \lambda''_i \equiv \xi''_i \mod f$. Therefore, for each $i$, we can define $V'_i, V''_i \subset \Z$ so that 
\begin{align*}
	\{\xi'_i - u(p^f - 1)\; | \; u \in \mathcal{I}'_i\} &= \{(p^f - 1)v + \lambda'_i \; | \; v \in V'_i\} \text{ and }\\
    \{\xi''_i - u(p^f - 1) \; | \; u \in \mathcal{I}''_j\} &= \{(p^f - 1)v + \lambda'_i \; | \; v \in V''_i\}
    \end{align*}



and further define $P', P'' \subset \Z/f\Z \times \mathbb{Z}$ as follows:
\begin{align*}
		P' &:= \{(i, v) \in \Z/f\Z \times \mathbb{Z} \; | \; v \in V'_i\} \: \text{ and}\\
		P'' &:= \{(i, v) \in \Z/f\Z \times \mathbb{Z} \; | \; v \in V''_i\}.
	\end{align*}


Define a function $$\beta \colon P' \cup P'' \to \Z \times [0, f''-1]$$ by setting 
\begin{align*}
\beta(i, v) = \begin{cases}
    \alpha'(i, u) &\text{ if } (i, v) \in P', \\
    \alpha''(i, u) &\text{ if } (i, v) \in P''
\end{cases}
\end{align*}
where $u$ satisfies
\begin{align*}
    \xi'_i - u(p^f - 1) = (p^f - 1)v + \lambda'_i \: &\text{ if } (i, v) \in P', \text{ and }\\
    \xi''_i - u(p^f - 1) = (p^f - 1)v + \lambda'_i \: &\text{ if } (i, v) \in P''.
\end{align*}

Note that there is no ambiguity in the definition of $\beta(i,v)$ when $(i, v) \in P' \cap P''$. Furthermore, $\beta|_{P'}$ and $\beta|_{P''}$ are injective functions and $J^{\text{AH}}_{\tss}(\chi_1, \chi_2) = \beta(P')$ while $J^{\text{AH}}_{\tsss}(\chi_1, \chi_2) = \beta(P'')$.

\begin{lemma}\label{obvious-match}
    $J^{\text{AH}}_{\tss}(\chi_1, \chi_2) = J^{\text{AH}}_{\tsss}(\chi_1, \chi_2)$ if and only if $\beta(P' \smallsetminus P'') = \beta(P'' \smallsetminus P')$.
\end{lemma}
\begin{proof}
    Clear.
\end{proof}




\begin{lemma}\label{reduce-problem-of-J-V}
	
Fix $(i, v) \in P'$. There exists $(j, w) \in P''$ such that $\beta(i,v) = \beta(j,w)$ if and only if the following are true:
\begin{enumerate}
    \item $(p^f-1)v + \lambda'_{i} = p^{\nu} ((p^f-1)w + \lambda'_{j})$ for some $\nu \in \Z$, and
    \item $ i - j \equiv \nu \mod f$.
\end{enumerate}
\end{lemma}
\begin{proof}
Let $\beta(i, v) = (n_1, \kappa_1)$ and $\beta(j, w) = (n_2, \kappa_2)$, and let $\nu_1$ be the $p$-adic valuation of $(p^{f}-1)v + \lambda'_i$ and $\nu_2 = \nu_1 + \nu$ be the $p$-adic valuation of $(p^{f}-1)w + \lambda'_j$. Further, let $i_{n_1}$ and $i_{n_2}$ be the unique representatives in $[0, f'-1]$ of $i - \nu_1$ and $j - \nu_2$ respectively. From the definition of $n_1$ and $n_2$, we see that the equality $n_1 = n_2$ holds if and only if $(p^f-1)v + \lambda'_{i} = p^{\nu} ((p^f-1)w + \lambda'_{j})$. This in turn implies $i_{n_1} = i_{n_2}$ and $$\kappa_1 - \kappa_2 \equiv \frac{(i - j) - \nu}{f'} \mod f''.$$
Therefore, if $n_1 = n_2$ holds, then $\kappa_1 = \kappa_2$ if and only if $i - j \equiv \nu \mod f.$
\end{proof}

\begin{corollary}\label{easy-match}
    Let $(i, v) \in P'$ and $(j, w) \in P''$ satisfy one of the following:
    \begin{enumerate}
        \item $j = i+1$ and $w=pv + z'_{i+1} - y'_{i+1}$, or
        \item $j = i-1$ and $v = pw + z'_{i} - y'_i$.
    \end{enumerate}
    Then $\beta(i,v) = \beta(j,w)$.
\end{corollary}
\begin{proof}
    Plugging into the expression for $\lambda'_i$ and $\lambda'_j$, in the first situation we get $$(p^{f} - 1)v + \lambda'_i = \frac{(p^{f} - 1)w + \lambda'_j}{p}$$ and in the second, we get $$(p^{f} - 1)v + \lambda'_i = p((p^{f} - 1)w + \lambda'_j).$$ An application of \Cref{reduce-problem-of-J-V} finishes the proof.
\end{proof}

\subsubsection{Case 1: \texorpdfstring{$\tss$ and $\tsss$ both satisfy \Cref{ramified-codim-1}\ref{ramified-codim-1-a}.}{1a}}\label{typeII-arb-Case1} 
\begin{proposition}\label{ram-type2-aa-candidates}
    There exists a pair of $I_K$-characters $\chi_1, \chi_2$ so that for some $j, l \in \Z/f\Z$, \begin{align*}
 \chi_1 \qquad & = && \omega_{j} ^{ e-1} \prod_{i \neq j} \omega_i^{s'_i + e} \prod_{i \in \Z/f\Z} \omega_i^{t'_i} &= \qquad & \omega_l^{e-1} \prod_{i \neq l} \omega_i^{s''_i + e} \prod_{i \in \Z/f\Z} \omega_i^{t''_i}, \\
 \chi_2 \qquad &= && \omega_{j}^{s'_j + 1} \prod_{i \in \Z/f\Z} \omega_i^{t'_i} &= \qquad & \omega_{l}^{s''_l + 1} \prod_{i \in \Z/f\Z} \omega_i^{t''_i}
 \end{align*} and $s'_{j}, s'_l \leq p-2$ if and only if after exchanging $\tss$ and $\tsss$ if necessary and translating indices by adding a fixed element in $\Z/f\Z$ if necessary so that $j = f-1$, $\bf{s}'$ and $\bf{s}''$ are as in one of the rows of \Cref{table:ram-type2-aa-candidates} and 
 \begin{align}\label{det-twist-type2-ram-aa}
     s'_{f-1} + 1 + \sum_{i \in \Z/f\Z} t'_i p^{f-1-i}  \equiv (s''_l +1) p^{f-1-l} + \sum_{i \in \Z/f\Z} t''_i p^{f-1-i}
\end{align}
where $l$ is described in \Cref{table:ram-type2-aa-candidates}.
\end{proposition}

\begin{proof}
    By \Cref{proposition-unram-type-1}, the existence of a pair of $I_K$-characters $(\chi_1, \chi_2)$ with highest weight $\ts$ and satisfying the statement of the Proposition is equivalent to the existence of $\ts$ so that for some $j, l \in \Z/f\Z$, the following two statements hold:
    \begin{itemize}
        \item After translating all indices by adding $f-1-j$ so that $j$ translates to $f-1$,
        $\ts$ and $\tss$ satisfy \Cref{type-1-case-1-det-twist} and one of the conditions in rows \ref{item-unram-generic},\ref{item-unram-3} and \ref{item-unram-6} of \Cref{table:unram-type-1}.
        \item After translating all indices in $\Z/f\Z$ by adding $f-1-l$ so that $l$ translates to $f-1$ 
        and replacing 
         $\tss$ with $\tsss$ in the statements in \Cref{type-1-case-1-det-twist} and \Cref{table:unram-type-1},
        $\ts$ and $\tsss$ satisfy \Cref{type-1-case-1-det-twist} and one of the conditions in rows \ref{item-unram-generic},\ref{item-unram-3} and \ref{item-unram-6} of \Cref{table:unram-type-1}.
    \end{itemize}  

   We may assume without loss of generality that $j=f-1$ (otherwise translate all indices by adding $f-1-j$ so that $j$ can be taken to be $f-1$). The equation \Cref{det-twist-type2-ram-aa} follows immediately from considering  \Cref{type-1-case-1-det-twist} or alternatively, the two expressions for $\chi_2$ in the statement of the Proposition. For the rest, we need to compare pairs of (same or distinct) rows from amongst rows \ref{item-unram-generic},\ref{item-unram-3} and \ref{item-unram-6} of \Cref{table:unram-type-1} so that after translating the indices in one of the rows by a fixed element of $\Z/f\Z$ if necessary, the tuples $\bf{s}$ in the two rows match up. The requirement that $\tss$ and $\tsss$ are non-isomorphic necessitates that the only comparisons to be made are between row \ref{item-unram-generic} and translated versions of itself.
Explicitly, $\bf{s}'$, $\bf{s}''$ and $\bf{s}$ are exactly those as in the rows of \Cref{table:ram-type2-aa-candidates}.
\end{proof}

\begin{proposition}\label{summary-ramified-aa}
    For each pair ($\tss$, $\tsss$) as in the statement of \Cref{ram-type2-aa-candidates}, $J_{\tss}^{\text{AH}}(\chi_1, \chi_2) \neq J_{\tsss}^{\text{AH}}(\chi_1, \chi_2)$.
\end{proposition}
\begin{proof}
    Using \Cref{obvious-match}, we need to show that for $\tss$ and $\tsss$ as in each row of \Cref{table:ram-type2-aa-candidates} and \Cref{det-twist-type2-ram-aa}, $\beta(P' \smallsetminus P'') \neq \beta(P'' \smallsetminus P')$.

    We have
    \begin{align*}
	&y'_i = 
	\begin{cases}
 0 &  \text{if  } i \neq f-1,\\
		s'_i + 1 & \text{if  } i = f-1
	\end{cases} &&\qquad z'_i = 
	\begin{cases}
 	s'_i + e &  \text{if  } i \neq f-1 \\
		e-1 & \text{if  } i = f-1
	\end{cases} \\
 	&y''_i = 
	\begin{cases}
 0 &  \text{if  } i \neq l,\\
		s''_i + 1 & \text{if  } i = l
	\end{cases} &&\qquad z''_i = 
	\begin{cases}
 		s''_i + e &  \text{if  } i \neq l \\
		e-1 & \text{if  } i = l
	\end{cases} 
\end{align*} and further,
\begin{align*}
         \mathcal{I}'_i &= 
	\begin{cases}
		\{0\} \cup [s'_i + 1, s'_i + e-1] & \text{if  } i \neq f-1\\
  [0, e-2] &\text{if  } i = f-1
	\end{cases} \\
        \mathcal{I}''_i &= 
	\begin{cases}
		\{0\} \cup [s''_i + 1, s''_i + e-1] & \text{if  } i \neq l \\
  [0, e-2] &\text{if  } i = l.
	\end{cases}\\
 V'_i &=
	\begin{cases}
            [1, e-1] \cup \{z'_i - y'_i\} & \qquad \text{if  } i \neq f-1 \\
		[1, e-1] & \qquad \text{if  } i = f-1
	\end{cases}
 \end{align*}

    Let's first assume the setting of \Cref{table:ram-type2-aa-candidates}\ref{ramified-a-meets-a-i}. 
    Rewriting $z''_i$ in terms of $y'_i, z'_i, y''_i$, we get
    \begin{align*}
	z''_i = 
	\begin{cases}
 		e + s'_i = (z'_i - y'_i)&  \text{if  } i \in [0, m-1] \cup [l+1, r-1] \\
   		e + 1 + s'_i = 1 + (z'_i - y'_i)&  \text{if  } i = m \\
     		e = -(p-1) + (z'_i - y'_i)&  \text{if  } i \in [m+1, l-1] \\
       		e-1 = - p + (z'_i - y'_i) + y''_i&  \text{if  } i = l \\
         		e - 1 + s'_i = -1 + (z'_i - y'_i)&  \text{if  } i = r \\
           		p + e - 1 = p-1 + (z'_i - y'_i) &  \text{if  } i \in [r+1, f-2] \\
		p + e - 2 - s'_{i} = p + (z'_{i} - y'_{i}) & \text{if  } i = f-1
	\end{cases}
\end{align*}
Computing $V''_i$ in terms of $y'_i$ and $z'_i$ for each $i \in \Z/f\Z$, we obtain
\begin{align*}
	V''_i &= 
	\begin{cases}
 [1, e-1] \cup \{z'_i - y'_i\}&  \text{if  } i \in  [0, m-1] \cup [l+1, r-1]\\
 [0, e-2] \cup \{z'_i - y'_i\}&  \text{if  } i = m \\
[0, e-1]&  \text{if  } i \in [m+1, l-1] \\
[1, e-1] &  \text{if  } i = l \\
		[2, e] \cup \{z'_i - y'_i\}&  \text{if  } i = r \\
  [2, e] \cup \{p + z'_{i} - y'_{i}\} &  \text{if  } i \in [r+1, f-2] \\
		[1, e-1] \cup \{p + z'_{i} - y'_{i}\} & \text{if  } i = f-1 \\	
	\end{cases}
\end{align*}

Therefore, \begin{align*}
    P' \smallsetminus P'' &= &&\{(m, e-1)\} \cup \{(i, z'_i - y'_i)\; | \; i \in [m+1, l]\} \cup \{(i,1) \; | \; i \in [r, f-2]\}, \text{ and} \\
    P'' \smallsetminus P' &= &&\{(i, 0)\;|\; i \in [m, l-1]\} \cup \{(r, e)\} \cup \{(i, p+z'_i - y'_i) \; | \; i \in [r+1, f-1]\}.
\end{align*}
Using \Cref{easy-match}, we find that $\beta(\{(i, z'_i - y'_i)\; | \; i \in [m+1, l]\} \cup \{(i,1) \; | \; i \in [r, f-2]\}) = \beta(\{(i, 0)\;|\; i \in [m, l-1]\} \cup \{(i, p+z'_i - y'_i) \; | \; i \in [r+1, f-1]\})$. Therefore, we only need to show that $\beta(m, e-1) \neq \beta(r,e)$. We note that while $m \neq r$, the $p$-adic valuations of $(p^{f}-1)(e-1) + \lambda'_m$ and of $(p^{f}-1)e + \lambda'_r$ are both $0$ using $s'_m \neq p-1$ and $s'_r \neq 0$. An application of \Cref{reduce-problem-of-J-V} finishes this calculation.




Next, assume the setting of \Cref{table:ram-type2-aa-candidates}\ref{ramified-a-meets-a-ii}. We have
\begin{align*}
	z''_i &= 
	\begin{cases}
 e = -(p-1) + (z'_i - y'_i)&  \text{if  } i \in [0, l-1]\\
 		e-1 = - p + (z'_i - y'_i) + y''_i&  \text{if  } i = l \\
		e + s'_i = (z'_i - y'_i)&  \text{if  } i \in [l+1, r-1]\\
  		e - 1 + s'_i = -1 + (z'_i - y'_i)&  \text{if  } i = r \\
	p + e - 1 = p-1 + (z'_i - y'_i) &  \text{if  } i \in [r+1, f-2] \\
		p + e - 1 - s'_{i} = p + (z'_{i} - y'_{i}) + 1 & \text{if  } i = f-1 \\
	\end{cases}\\
V''_j &= 
	\begin{cases}
		[0, e-1]&  \qquad \qquad \text{if  } i \in [0, l-1]\\
  		[1, e-1] & \qquad \qquad \text{if  } i = l \\
    [1, e-1] \cup \{s'_i + e\}& \qquad \qquad  \text{if  } i \in [l+1, r-1] \\
[2, e] \cup \{s'_i + e\}&  \qquad \qquad \text{if  } i = r \\
	[2, e] \cup \{p + z'_{i} - y'_{i}\} &  \qquad \qquad \text{if  } i \in [r+1, f-2] \\			
  [0, e-2] \cup \{p + z'_{i} - y'_{i}\} & \qquad \qquad \text{if  } i = f-1
		\end{cases}
\end{align*}
Applying similar reasoning as before, we need to show that $\beta(f-1, e-1) \neq \beta(r,e)$. Since $s'_{f-1} \neq p-1$ and $s'_{r} \neq 0$, we once again note that while $f-1 \neq r$, the $p$-adic valuations of $(p^f-1)(e-1) + \lambda'_{f-1}$ and $(p^f-1)e + \lambda'_r$ are both $0$. An application of \Cref{reduce-problem-of-J-V} finishes this calculation as well.

Finally, the proof in the setting of \Cref{table:ram-type2-aa-candidates}\ref{ramified-a-meets-a-iv} is similar to the two cases above and boils down to showing that $\beta(f-1, e-1) \neq \beta(l, e)$ using precisely the same strategy as above. We leave the details to the reader.\end{proof}

\subsubsection{Case 2: \texorpdfstring{$\tss$ and $\tsss$ both satisfy \Cref{ramified-codim-1}\ref{ramified-codim-1-b}.}{1b}}\label{typeII-arb-Case2} 
\begin{proposition}\label{ram-type2-bb-candidates}
    There exists a pair of $I_K$-characters $\chi_1, \chi_2$ so that for some $j, l \in \Z/f\Z$, \begin{align*}
 \chi_1 \qquad & = && \omega_{j} ^{s'_j + e-1} \prod_{i \neq j} \omega_i^{s'_i + e} \prod_{i \in \Z/f\Z} \omega_i^{t'_i} &= \qquad & \omega_l^{s''_l + e-1} \prod_{i \neq l} \omega_i^{s''_i + e} \prod_{i \in \Z/f\Z} \omega_i^{t''_i}, \\
 \chi_2 \qquad &= && \omega_{j} \prod_{i \in \Z/f\Z} \omega_i^{t'_i} &= \qquad & \omega_{l} \prod_{i \in \Z/f\Z} \omega_i^{t''_i}
 \end{align*} and $s'_{j}, s'_l \neq 0$ if and only if after exchanging $\tss$ and $\tsss$ if necessary and translating indices by adding a fixed element in $\Z/f\Z$ if necessary so that $j = f-1$, $\bf{s}'$ and $\bf{s}''$ are as in one of the rows of \Cref{table:ram-type2-bb-candidates} and 
 \begin{align}\label{det-twist-type2-ram-bb}
     1 + \sum_{i \in \Z/f\Z} t'_i p^{f-1-i}  \equiv p^{f-1-l} + \sum_{i \in \Z/f\Z} t''_i p^{f-1-i}
\end{align} where $l$ is described in \Cref{table:ram-type2-bb-candidates}.
\end{proposition}

\begin{proof}
    By \Cref{proposition-unram-type-1}, the existence of a pair of $I_K$-characters $(\chi_1, \chi_2)$ with highest weight $\ts$ and satisfying the statement of the Proposition is equivalent to the existence of $\ts$ so that for some $j, l \in \Z/f\Z$, the following two statements hold:
    \begin{itemize}
        \item After translating all indices by adding $f-1-j$ so that $j$ translates to $f-1$,
        $\ts$ and $\tss$ satisfy \Cref{ramified-codim-1-b-chi2} and one of the rows in \Cref{table:ram-type-1b}.
        \item After translating all indices in $\Z/f\Z$ by adding $f-1-l$ so that $l$ translates to $f-1$ 
        and replacing 
         $\tss$ with $\tsss$ in the statements in \Cref{ramified-codim-1-b-chi2} and \Cref{table:ram-type-1b},
        $\ts$ and $\tsss$ satisfy \Cref{ramified-codim-1-b-chi2} and one of the rows in \Cref{table:ram-type-1b}.
    \end{itemize}  

   We may assume without loss of generality that $j=f-1$ (otherwise translate all indices by adding $f-1-j$ so that $j$ can be taken to be $f-1$). The equation \Cref{det-twist-type2-ram-bb} follows immediately from considering  \Cref{ramified-codim-1-b-chi2}  or alternatively, the two expressions for $\chi_2$ in the statement of the Proposition. For the rest, we need to compare pairs of (same or distinct) rows in \Cref{table:ram-type-1b} so that after translating the indices in one of the rows by a fixed element of $\Z/f\Z$ if necessary, the tuples $\bf{s}$ in the two rows match up. The rows that can thus be compared so that we get non-isomorphic $\tss$ and $\tsss$ are \ref{ram-type-1b-i} and \ref{ram-type-1b-i}, \ref{ram-type-1b-i} and \ref{ram-type-1b-ii}, \ref{ram-type-1b-ii} and \ref{ram-type-1b-ii}, and \ref{ram-type-1b-ii} and \ref{ram-type-1b-iii}. Explicitly, $\bf{s}'$, $\bf{s}''$ and $\bf{s}$ are exactly those as in the rows of \Cref{table:ram-type2-bb-candidates}.
\end{proof}

\begin{proposition}\label{summary-ramified-bb}
    For each pair ($\tss$, $\tsss$) as in the statement of \Cref{ram-type2-bb-candidates}, $J_{\tss}^{\text{AH}}(\chi_1, \chi_2) = J_{\tsss}^{\text{AH}}(\chi_1, \chi_2)$ if and only if after exchanging $\tss$ and $\tsss$ if necessary and translating indices by adding a fixed element of $\Z/f\Z$ if necessary, $\bf{s}'$ and $\bf{s}''$ are those that feature in rows \ref{ramified-bb-iii}, \ref{ramified-bb-v} or \ref{ramified-bb-vi} of \Cref{table:ram-type2-bb-candidates}.
\end{proposition}

\begin{proof}
    Using \Cref{obvious-match}, we need to show that for $\tss$ and $\tsss$ satisfying \Cref{det-twist-type2-ram-bb} and $\bf{s}'$ and $\bf{s}''$ as in the rows \ref{ramified-bb-i}, \ref{ramified-bb-ii} and \ref{ramified-bb-iv} of \Cref{table:ram-type2-bb-candidates}, $\beta(P' \smallsetminus P'') \neq \beta(P'' \smallsetminus P')$, while for the remaining rows, $\beta(P' \smallsetminus P'') = \beta(P'' \smallsetminus P')$.

    We have
    \begin{align*}
	&y'_i = 
	\begin{cases}
 0 &  \text{if  } i \neq f-1,\\
		1 & \text{if  } i = f-1
	\end{cases} &&\qquad z'_i = 
	\begin{cases}
 	s'_i + e &  \text{if  } i \neq f-1 \\
		s'_i + e-1 & \text{if  } i = f-1
	\end{cases} \\
 	&y''_i = 
	\begin{cases}
 0 &  \text{if  } i \neq l,\\
		1 & \text{if  } i = l
	\end{cases} &&\qquad z''_i = 
	\begin{cases}
 		s''_i + e &  \text{if  } i \neq l \\
		s''_i + e-1 & \text{if  } i = l
	\end{cases} 
\end{align*} and further,
\begin{align*}
         \mathcal{I}'_i &= 
	\begin{cases}
		\{0\} \cup [s'_i + 1, s'_i + e-1] & \text{if  } i \neq f-1\\
  \{1\} \cup [s'_i + 1, s'_i + e-2] &\text{if  } i = f-1
	\end{cases} \\
        \mathcal{I}''_i &= 
	\begin{cases}
		\{0\} \cup [s''_i + 1, s''_i + e-1] & \text{if  } i \neq l \\
  \{1\} \cup [s''_i + 1, s''_i + e-2] &\text{if  } i = l.
	\end{cases}
 \end{align*}   
Assume first the setting of \Cref{table:ram-type2-bb-candidates}\ref{ramified-bb-i}. We have
\begin{align*}
	z''_i &= 
	\begin{cases}
 e + s'_i = z'_i - y'_i&  \text{if  } i \neq l, f-1 \\
 e+1+s'_i = (z'_i - y'_i) + y''_i& \text{if  } i = l \\
		e - 2 + s'_{i} = z'_{i} - y'_{i} & \text{if  } i = f-1 
	\end{cases} \\
 V'_i &=
	\begin{cases}
 [1, e-1] \cup \{z'_i - y'_i\} & \quad \qquad \text{if  } i \neq f-1 \\
		[1, e-2] \cup \{z'_i - y'_i\} &\quad \qquad \text{if  } i = f-1
	\end{cases} \\
 V''_i &= 
	\begin{cases}
 [1, e-1] \cup \{z'_i - y'_i\} & \quad \qquad \text{if  } i \neq l, f-1\\
 [1, e-2] \cup \{z'_i - y'_i\} & \quad \qquad \text{if  } i =l \\
		[1, e-1] \cup \{z'_i - y'_i\} & \quad \qquad \text{if  } i = f-1 \\
	\end{cases}
\end{align*}
We obtain
$$
    P' \smallsetminus P'' = \{l, e-1\}, \qquad
    P'' \smallsetminus P' = \{f-1, e-1\}.$$
Since $s'_l \neq p-1$ and $s'_{f-1} \neq 1$, the $p$-adic valuations of $(p^{f}-1)(e-1) + \lambda'_l$ and $(p^{f}-1)(e-1) + \lambda'_{f-1}$ are both $0$. Since $l \neq f-1$, \Cref{reduce-problem-of-J-V} shows that $J^{\mathrm{AH}}_{\sigma_{\mathbf{t''}, \mathbf{s''}}}(\chi_1, \chi_2)$. 

Next, assume the setting of \Cref{table:ram-type2-bb-candidates}\ref{ramified-bb-ii}. We have
\begin{align*}
	z''_i &= 
	\begin{cases}
 e + s'_i = z'_i - y'_i&  \text{if  } i \not \in \{l\} \cup [r, f-1]\\
 e+1+s'_i = (z'_i - y'_i) + y''_i&  \text{if  } i = l \\
 e-1+s'_i = -1 + (z'_{i} - y'_{i}) & \text{if  } i =r\\
 p + e - 1 = (p-1) + (z'_{i} - y'_{i}) & \text{if  } i \in [r+1, f-2]\\
p + e - 2 + s'_{i} = p + (z'_{i} - y'_{i}) & \text{if  } i = f-1 \\	
	\end{cases}\\
	V'_i &=
	\begin{cases}
 		[1, e-1] \cup \{z'_i - y'_i\} & \qquad\qquad \text{if  } i \neq f-1\\
		[1, e-1] &\qquad\qquad\text{if  } i = f-1
	\end{cases}\\
	V''_i &= 
	\begin{cases}
 [1, e-1] \cup \{z'_i - y'_i\} & \qquad\text{if  } i \not \in \{l\} \cup [r, f-1] \\
 	[1, e-2] \cup \{z'_i - y'_i\}&  \qquad\text{if  } i = l \\
  [2, e] \cup \{z'_i - y'_i\} & \qquad\text{if  } i = r \\
	[2, e] \cup \{p + z'_i - y'_i\} &\qquad \text{if  } i \in [r+1, f-2] \\
		[1, e-1] \cup \{p + z'_{i} - y'_i\} &\qquad \text{if  } i = f-1 \\
	\end{cases}
\end{align*}

Using \Cref{easy-match}, we only need to show that $\beta(\{(l, e-1), (f-1, z'_{f-1} - y'_{f-1})\}) \neq \beta(\{(r, e), (f-1, e-1)\})$. Assume otherwise. The $p$-adic valuations of both $(p^{f}-1)(e-1) + \lambda'_l$ and $(p^{f}-1)e + \lambda'_r$ equal $0$ while $l \neq r$. Therefore, 
we must have $\beta(f-1, z'_{f-1} - y'_{f-1}) = \beta(r,e)$. Equivalently, by \Cref{reduce-problem-of-J-V}, $\mathrm{val}_{p}((p^{f}-1)(z'_{f-1}-y'_{f-1}) + \lambda'_{f-1}) \equiv f-1-r \mod f$. Hence, 
\begin{align*}
    (p^f-1)e + \lambda'_r &\leq \frac{(p^{f}-1)(z'_{f-1}-y'_{f-1}) + \lambda'_{f-1}}{p^{f-1-r}} \\
    &= \frac{\sum_{j=-1}^{f-2} p^{f-1-j}(z'_{j} - y'_{j})}{p^{f-1-r}} \\
    &= \sum_{j = -1}^{r} p^{r-j}(z'_{j} -y'_j) + \sum_{j = r+1}^{f-2} \frac{z'_j - y'_j}{p^{j-r}}\\
    &= \sum_{j = 0}^{r+1} p^{j}(z'_{r-j} -y'_{r-j}) +  p^{-f} \sum_{j=r+2}^{f-1} p^j(z'_{r-j} - y'_{r-j}) \\
    &<(p^f-1)e + \lambda'_r
\end{align*}
giving rise to a contradiction.

Next, we assume that $\bf{s}'$ and $\bf{s}''$ satisfy one of the \Cref{table:ram-type2-bb-candidates} rows \ref{ramified-bb-iii}, \ref{ramified-bb-v} and \ref{ramified-bb-vi}. With this assumption,
\begin{align*}
	z''_i &= 
	\begin{cases}
 e + s'_i = z'_i - y'_i&  \text{if  } i \in [0, l-1] \\
 e+s'_i =  z'_{i} - y'_{i} & \text{if  } i =l\\
 p + e - 1 = (p-1) + (z'_{i} - y'_{i}) & \text{if  } i \in [l+1, f-2]\\
p + e - 2 + s'_{i} = p + (z'_{i} - y'_{i}) & \text{if  } i = f-1		
	\end{cases} \\
 	V'_i &=
	\begin{cases}
 [1, e-1] \cup \{z'_i - y'_i\} & \qquad \qquad \text{if  } i \not= f-1 \\
		[1, e-1] & \qquad \qquad \text{if  } i = f-1
	\end{cases}\\
 	V''_i &= 
	\begin{cases}
 [1, e-1] \cup \{z'_i - y'_i\} &\qquad  \text{if  } i \in [0, l-1] \\
 [2, e-1] \cup \{z'_i - y'_i\} & \qquad \text{if  } i = l \\
[2, e] \cup \{p + z'_i - y'_i\} &\qquad \text{if  } i \in [l+1, f-2] \\
		[1, e-1] \cup \{p + z'_{i} - y'_i\} & \qquad \text{if  } i = f-1
	\end{cases}
\end{align*}

Using \Cref{easy-match}, we find that $\beta(P' \smallsetminus P'') = \beta(P'' \smallsetminus P')$, as desired. 

Finally, we assume the setting of \Cref{table:ram-type2-bb-candidates}\ref{ramified-bb-iv}. We have
\begin{align*}
    z''_i &= \begin{cases}
        e + s'_i = z'_i - y'_i & \text{if  } i \in [0, m-1] \cup [l+1, r-1] \\
        e+ s'_i + 1 = (z'_i - y'_i) + 1 & \text{if  } i =m\\
        e = (z'_i - y'_i) + 1-p & \text{if  } i \in [m+1, l] \\
        e + s'_i -1  = (z'_i - y'_i) - 1 & \text{if  } i =r \\
        e + p-1 = (z'_i - y'_i) + p & \text{if  } i \in [r+1, f-1]
    \end{cases} \\
    V'_i &= \begin{cases}
        [1, e-1] \cup \{z'_i - y'_i\} & \qquad \text{if  } i \neq f-1 \\
        [1, e-1] & \qquad \text{if  } i = f-1
    \end{cases} \\
    V''_i &= \begin{cases}
         [1, e-1] \cup \{z'_i - y'_i\} & \text{if  } i \in [0, m-1] \cup [l+1, r-1] \\
         [0, e-2] \cup \{z'_i - y'_i \}& \text{if  } i =m\\
         [0, e-1] & \text{if  } i \in [m+1, l-1] \\
         [1, e-1] & \text{if  } i =l\\
         [2, e] \cup \{z'_i - y'_i\} & \text{if  } i =r \\
         [2, e] \cup \{p+ z'_i - y'_i\} & \text{if  } i \in [r+1, f-2] \\
         [1, e-1] \cup \{p + z'_i - y'_i \}& \text{if  } i = f-1
    \end{cases}
\end{align*}

Using \Cref{easy-match}, one simply needs to check that $\beta(m, e-1) \neq \beta(r,e)$. The $p$-adic valuations of both $(p^f-1)(e-1) + \lambda'_{m}$ and $(p^{f}-1)e + \lambda'_r$ are $0$, since $s'_m \neq p-1$ and $s'_r \neq 0$. Therefore, $m \neq r$ implies $\beta(m, e-1) \neq \beta(r,e)$, finishing the proof.
\end{proof}

\subsubsection{Case 3: \texorpdfstring{$\tss$ satisfies \Cref{ramified-codim-1}\ref{ramified-codim-1-a} while $\tsss$ satisfies \Cref{ramified-codim-1}\ref{ramified-codim-1-b}.}{1a-vs-1b}}\label{typeII-arb-Case3}

\begin{proposition}\label{ram-type2-ab-candidates}
    There exists a pair of $I_K$-characters $\chi_1, \chi_2$ so that for some $j, l \in \Z/f\Z$, \begin{align*}
 \chi_1 \qquad & = && \omega_{j} ^{e-1} \prod_{i \neq j} \omega_i^{s'_i + e} \prod_{i \in \Z/f\Z} \omega_i^{t'_i} &= \qquad & \omega_l^{s''_l + e-1} \prod_{i \neq l} \omega_i^{s''_i + e} \prod_{i \in \Z/f\Z} \omega_i^{t''_i}, \\
 \chi_2 \qquad &= && \omega_j^{s'_j + 1} \prod_{i \in \Z/f\Z} \omega_i^{t'_i} &= \qquad & \omega_{l} \prod_{i \in \Z/f\Z} \omega_i^{t''_i},
 \end{align*}$s'_{j} \leq p-2$ and $s'_l \neq 0$ if and only if after exchanging $\tss$ and $\tsss$ if necessary and translating indices by adding a fixed element in $\Z/f\Z$ if necessary so that $j = f-1$, $\bf{s}'$ and $\bf{s}''$ are as in one of the rows of \Cref{table:ram-type2-ab-candidates} and 
 \begin{align}\label{det-twist-type2-ram-ab}
     s'_{f-1} + 1 + \sum_{i \in \Z/f\Z} t'_i p^{f-1-i}  \equiv p^{f-1-l} + \sum_{i \in \Z/f\Z} t''_i p^{f-1-i}
\end{align} where $l$ is described in \Cref{table:ram-type2-ab-candidates}.
\end{proposition}

\begin{proof}
    By \Cref{proposition-unram-type-1}, the existence of a pair of $I_K$-characters $(\chi_1, \chi_2)$ with highest weight $\ts$ and satisfying the statement of the Proposition is equivalent to the existence of $\ts$ so that for some $j, l \in \Z/f\Z$, the following two statements hold:
    \begin{itemize}
        \item After translating all indices by adding $f-1-j$ so that $j$ translates to $f-1$,
        $\ts$ and $\tss$ satisfy \Cref{type-1-case-1-det-twist} and one of the rows \ref{item-unram-generic}, \ref{item-unram-3} and \ref{item-unram-6} in \Cref{table:unram-type-1}.
        \item After translating all indices in $\Z/f\Z$ by adding $f-1-l$ so that $l$ translates to $f-1$ 
        and replacing 
         $\tss$ with $\tsss$ in the statements in \Cref{ramified-codim-1-b-chi2} and \Cref{table:ram-type-1b},
        $\ts$ and $\tsss$ satisfy \Cref{ramified-codim-1-b-chi2} and one of the rows in \Cref{table:ram-type-1b}.
    \end{itemize}  

   We may assume without loss of generality that $j=f-1$ (otherwise translate all indices by adding $f-1-j$ so that $j$ can be taken to be $f-1$). The equation \Cref{det-twist-type2-ram-ab} follows immediately from considering  \Cref{type-1-case-1-det-twist,ramified-codim-1-b-chi2}  or alternatively, the two expressions for $\chi_2$ in the statement of the Proposition. For the rest, we need to compare rows \ref{item-unram-generic}, \ref{item-unram-3} and \ref{item-unram-6} in \Cref{table:unram-type-1} with the rows of \Cref{table:ram-type-1b} so that after translating the indices in the rows of \Cref{table:ram-type-1b} by a fixed element of $\Z/f\Z$ if necessary, the tuples $\bf{s}$ in the two rows match up. The rows of \Cref{table:unram-type-1} and \Cref{table:ram-type-1b} that can thus be compared are respectively \ref{item-unram-generic} and \ref{ram-type-1b-i},  \ref{item-unram-generic} and \ref{ram-type-1b-ii}, \ref{item-unram-3} and \ref{ram-type-1b-i}, \ref{item-unram-3} and \ref{ram-type-1b-ii}, \ref{item-unram-6} and \ref{ram-type-1b-ii}, and finally, \ref{item-unram-6} and \ref{ram-type-1b-iii}. Explicitly, $\bf{s}'$, $\bf{s}''$ and $\bf{s}$ are exactly those as in the rows of \Cref{table:ram-type2-ab-candidates}.
\end{proof}

\begin{proposition}\label{summary-ramified-ab}
    For each pair ($\tss$, $\tsss$) as in the statement of \Cref{ram-type2-bb-candidates}, $J_{\tss}^{\text{AH}}(\chi_1, \chi_2) = J_{\tsss}^{\text{AH}}(\chi_1, \chi_2)$ if and only if after translating indices by adding a fixed element of $\Z/f\Z$ if necessary, $\bf{s}'$ and $\bf{s}''$ are those that feature in rows \ref{ramified-ab-ii}, \ref{ramified-ab-vi}, \ref{ramified-ab-vii}, \ref{ramified-ab-viii}, \ref{ramified-ab-ix}, \ref{ramified-ab-x} or \ref{ramified-ab-xi}  of \Cref{table:ram-type2-ab-candidates}.
\end{proposition}

\begin{proof}
    Using \Cref{obvious-match}, we need to show that for $\tss$ and $\tsss$ satisfying \Cref{det-twist-type2-ram-bb} and $\bf{s}'$ and $\bf{s}''$ as in the rows \ref{ramified-ab-i}, \ref{ramified-ab-iii}, \ref{ramified-ab-iv} and \ref{ramified-ab-v} of \Cref{table:ram-type2-ab-candidates}, $\beta(P' \smallsetminus P'') \neq \beta(P'' \smallsetminus P')$, while for the remaining rows, $\beta(P' \smallsetminus P'') = \beta(P'' \smallsetminus P')$.

  We have
    \begin{align*}
	&y'_i = 
	\begin{cases}
 0 &  \text{if  } i \neq f-1,\\
		s'_{i} + 1 & \text{if  } i = f-1
	\end{cases} &&\qquad z'_i = 
	\begin{cases}
 	s'_i + e & \qquad \text{if  } i \neq f-1 \\
		e-1 & \qquad \text{if  } i = f-1
	\end{cases} \\
 	&y''_i = 
	\begin{cases}
 0 &  \qquad \text{if  } i \neq l,\\
		1 & \qquad \text{if  } i = l
	\end{cases} &&\qquad z''_i = 
	\begin{cases}
 		s''_i + e &  \text{if  } i \neq l \\
		s''_i + e-1 & \text{if  } i = l
	\end{cases} 
\end{align*} and further,
\begin{align*}
         \mathcal{I}'_i &= 
	\begin{cases}
		\{0\} \cup [s'_i + 1, s'_i + e-1] & \text{if  } i \neq f-1\\
  [0, e-2] &\text{if  } i = f-1
	\end{cases} \\
        \mathcal{I}''_i &= 
	\begin{cases}
		\{0\} \cup [s''_i + 1, s''_i + e-1] & \text{if  } i \neq l \\
  \{1\} \cup [s''_i + 1, s''_i + e-2] &\text{if  } i = l.
	\end{cases} \\
 V'_i &= 
	\begin{cases}
		[1, e-1] \cup \{z'_i - y'_i\} & \text{if  } i \neq f-1\\
  [1, e-1] &\text{if  } i = f-1
	\end{cases} \\
 \end{align*}   

First, assume the setting of \Cref{table:ram-type2-ab-candidates}\ref{ramified-ab-i}.

We have
\begin{align*}
	z''_i &= 
	\begin{cases}
 e + s'_i = z'_i - y'_i&  \text{if  } i \in [0, l-1] \cup [l+1, r-1] \\
 e+1+s'_i = (z'_i - y'_i) + y''_i& \text{if  } i = l \\
		e -1 + s'_{i} = (z'_{i} - y'_{i}) - 1& \text{if  } i = r \\
  e + p-1 = (z'_i -y'_i) + p-1 &\text{if  } i \in [r+1, f-2] \\
  e + p - 2 - s'_i = (z'_i - y'_i) + p &\text{if  } i = f-1
	\end{cases} \\
 V''_i &= 
	\begin{cases}
 [1, e-1] \cup \{z'_i - y'_i\} & \qquad \text{if  } i \in [0,l-1] \cup [l+1, r-1]\\
 [1, e-2] \cup \{z'_i - y'_i\} & \qquad \text{if  } i =l\\
 [2, e] \cup \{z'_i - y'_i\} & \qquad \text{if  } i =r \\
 [2, e] \cup \{p+z'_i - y'_i\} & \qquad \text{if  } i \in [r+1, f-2] \\
		[1, e-1] \cup \{p+z'_i - y'_i\} & \qquad \text{if  } i = f-1 \\
	\end{cases}
\end{align*}

We need to check that $\beta(l, e-1) \neq \beta(r, e)$. This follows from \Cref{easy-match} and the fact that $s'_l \neq p-1$ and $s'_r \neq 0$ implying that the $p$-adic valuations of $(p^f - 1)(e-1) + \lambda'_l$ and of $(p^f - 1)e + \lambda'_r$ are both $0$.

Next, assume the setting of \Cref{table:ram-type2-ab-candidates}\ref{ramified-ab-ii}, \ref{ramified-ab-vi}, \ref{ramified-ab-viii}, \ref{ramified-ab-ix} or \ref{ramified-ab-xi}.
We have
\begin{align*}
	z''_i &= 
	\begin{cases}
 e + s'_i = z'_i - y'_i&  \text{if  } i \in [0, l-1] \\
 e+s'_i = z'_i - y'_i & \text{if  } i = l \\
  e + p-1 = (z'_i -y'_i) + p-1 &\text{if  } i \in [l+1, f-2] \\
  e + p - 2 - s'_i = (z'_i - y'_i) + p &\text{if  } i = f-1
	\end{cases} \\
 V''_i &= 
	\begin{cases}
 [1, e-1] \cup \{z'_i - y'_i\} & \qquad \text{if  } i \in [0,l-1]\\
 [2, e-1] \cup \{z'_i - y'_i\} & \qquad \text{if  } i =l \\
 [2, e] \cup \{p+z'_i - y'_i\} & \qquad \text{if  } i \in [l+1, f-2] \\
		[1, e-1] \cup \{p+z'_i - y'_i\} & \qquad \text{if  } i = f-1 \\
	\end{cases}
\end{align*}
An application of \Cref{easy-match} shows that $\beta(P'\smallsetminus P'') = \beta(P'' \smallsetminus P')$, as desired.

Assuming the setting of \Cref{table:ram-type2-ab-candidates}\ref{ramified-ab-iii}, we have
\begin{align*}
	z''_i &= 
	\begin{cases}
 e + s'_i = z'_i - y'_i&  \text{if  } i \in [0, r-1] \\
 e+s'_i-1 = (z'_i - y'_i) -1 & \text{if  } i = r \\
  e + p-1 = (z'_i -y'_i) + p-1 &\text{if  } i \in [r+1, f-2] \\
  e + p - 1 - s'_i = (z'_i - y'_i) + p + y''_i &\text{if  } i = l = f-1
	\end{cases} \\
 V''_i &= 
	\begin{cases}
 [1, e-1] \cup \{z'_i - y'_i\} & \qquad \text{if  } i \in [0,r-1]\\
 [2, e] \cup \{z'_i - y'_i\} & \qquad \text{if  } i =r \\
 [2, e] \cup \{p+z'_i - y'_i\} & \qquad \text{if  } i \in [r+1, f-2] \\
		[1, e-2] \cup \{p+z'_i - y'_i\} & \qquad \text{if  } i = l = f-1 \\
	\end{cases}
\end{align*}

We find that $\beta(f-1, e-1) \neq \beta(r, e)$ as $s'_{f-1} \neq p-1$ and $s'_r \neq 0$ imply that the $p$-adic valuations of $(p^f - 1)(e-1) + \lambda'_{f-1}$ and of $(p^f - 1)e + \lambda'_r$ are both $0$, while $r \neq f-1$. Therefore, $\beta(P'\smallsetminus P'') \neq \beta(P'' \smallsetminus P')$

Assuming the setting of \Cref{table:ram-type2-ab-candidates}\ref{ramified-ab-iv}, we have
\begin{align*}
	z''_i &= 
	\begin{cases}
 e + s'_i = z'_i - y'_i&  \text{if  } i \in [0, m-1] \cup [l+1, r-1] \\
 e + s'_i + 1 = z'_i - y'_i + 1&  \text{if  } i = m \\
 e = (z'_i - y'_i) - p + 1&  \text{if  } i \in [m+1, l-1] \\
 e = (z'_i - y'_i) + y''_i - p &  \text{if  } i =l \\
 e+s'_i-1 = (z'_i - y'_i) -1 & \text{if  } i = r \\
  e + p-1 = (z'_i -y'_i) + p-1 &\text{if  } i \in [r+1, f-2] \\
  e + p - 1 - s'_i = (z'_i - y'_i) + p + y''_i &\text{if  } i = f-1
	\end{cases} \\
 V''_i &= 
	\begin{cases}
 [1, e-1] \cup \{z'_i - y'_i\} & \quad \qquad \text{if  } i \in [0, m-1] \cup [l+1, r-1]\\
  [0, e-2] \cup \{z'_i - y'_i\} & \quad \qquad \text{if  } i =m\\
   [0, e-1] & \quad \qquad \text{if  } i \in [m+1, l-1]\\
    [1, e-1] & \quad \qquad \text{if  } i =l\\
 [2, e] \cup \{z'_i - y'_i\} & \quad \qquad \text{if  } i =r \\
 [2, e] \cup \{p+z'_i - y'_i\} & \quad \qquad \text{if  } i \in [r+1, f-2] \\
		[1, e-1] \cup \{p+z'_i - y'_i\} & \quad \qquad \text{if  } i = f-1 \\
	\end{cases}
\end{align*}

After applying \Cref{easy-match}, we need to show that $\beta(m, e-1) \neq \beta(r,e)$. Since $s'_m \neq p-1$ and $s'_r \neq 0$, the $p$-adic valuations of $(p^f - 1)(e-1) + \lambda'_m$ and $(p^{f}-1)e + \lambda'_r$ are both $0$. Thus, from the inequality $m \neq r$, we obtain $\beta(m, e-1) \neq \beta(r,e)$.

Now, assume the setting of \Cref{table:ram-type2-ab-candidates}\ref{ramified-ab-v}. We have
\begin{align*}
	z''_i &= 
	\begin{cases}
 e  = (z'_i - y'_i) - p + 1&  \text{if  } i \in [0, l-1] \\
 e = (z'_i - y'_i) - p + y''_i &  \text{if  } i = l \\
 e + s'_i = (z'_i - y'_i) &  \text{if  } i \in [l+1, r-1] \\
 e+s'_i-1 = (z'_i - y'_i) -1 & \text{if  } i = r \\
  e + p-1 = (z'_i -y'_i) + p-1 &\text{if  } i \in [r+1, f-2] \\
  e + p - 1 - s'_i = (z'_i - y'_i) + p + 1 &\text{if  } i = f-1
	\end{cases} \\
 V''_i &= 
	\begin{cases}
 [0, e-1] & \quad \qquad \text{if  } i \in [0, l-1]\\
  [1, e-1] & \quad \qquad \text{if  } i =l\\
   [1, e-1] \cup \{z'_i - y'_i\} & \quad \qquad \text{if  } i \in [l+1, r-1]\\
 [2, e] \cup \{z'_i - y'_i\} & \quad \qquad \text{if  } i =r \\
 [2, e] \cup \{p+z'_i - y'_i\} & \quad \qquad \text{if  } i \in [r+1, f-2] \\
		[0, e-2] \cup \{p+z'_i - y'_i\} & \quad \qquad \text{if  } i = f-1 \\
	\end{cases}
\end{align*}
After applying \Cref{easy-match}, we need to show that $\beta(f-1, e-1) \neq \beta(r,e)$. This follows from the fact that $s'_{f-1} \neq p -1$ and $s'_r \neq 0$ and similar argument as before.

Finally, assume the setting of \Cref{table:ram-type2-ab-candidates}\ref{ramified-ab-vii} or \ref{ramified-ab-x}. We have
\begin{align*}
	z''_i &= 
	\begin{cases}
 e + p-1 = (z'_i - y'_i) + p - 1&  \text{if  } i \in [0, f-2] \\
  e + p - 2 - s'_i = (z'_i - y'_i) + y''_i + p - 1 &\text{if  } i = f-1
	\end{cases} \\
 V''_i &= 
	\begin{cases}
 [2, e] \cup \{p + z'_i - y'_i\} & \quad \qquad \text{if  } i \in [0, f-2]\\
		[2, e-1] \cup \{p+z'_i - y'_i\} & \quad \qquad \text{if  } i = f-1 \\
	\end{cases}
\end{align*}

An application of \Cref{easy-match} shows that $\beta(P' \smallsetminus P'') = \beta(P'' \smallsetminus P')$, finishing the proof.
\end{proof}

\begin{remark}\label{type-II-regular}
	Let $e>1$, $f>1$ and $\tss$ and $\tsss$ be a pair of weakly regular Serre weights. An examination of \Cref{table:ram-type2-bb-candidates,table:ram-type2-ab-candidates} shows that there exists a type II intersection between $\tss$ and $\tsss$ if and only if $\mathrm{Ext}^1_{\Fbar[\GL_2(k)]}(\sigma_{\mathbf{t'}, \mathbf{s'}}, \sigma_{\mathbf{t''}, \mathbf{s''}}) \neq 0$ via the criterion described in \Cref{ext-gamma}\ref{ext-k-i}. Moreover, the relationship between $\tss$ and $\tsss$ is asymmetric, and therefore, there exists at most one pair of $I_K$-characters witnessing a type II intersection between them.

\end{remark}

\section{Conclusion}\label{conclusion}

From the findings of \Cref{type1,type2}, we get our main results:

\begin{theorem}\label{Ext-group-intersection}
		Let $\ts$ and $\tss$ be a pair of non-isomorphic, non-Steinberg Serre weights. Then
  \begin{align}\label{one-way}
  \mathrm{Ext}^1_{\Fbar[\GL_2(k)]}(\ts, \tss) \neq 0 \implies \dim \cX_{\ts} \cap \cX_{\tss} = [K:\Qp]-1.\end{align} If $\ts$ and $\tss$ are further assumed to be weakly regular, then \Cref{one-way} upgrades to $$\mathrm{Ext}^1_{\Fbar[\GL_2(k)]}(\ts, \tss) \neq 0 \iff \dim \cX_{\ts} \cap \cX_{\tss} = [K:\Qp]-1.$$
  \end{theorem}

\begin{proof}
Consider the explicit numerical criterion for type I and II intersections in \Cref{prop:tot-ram-type1a,prop:tot-ram-type1c,type-II-tot-ram} when $f=1$, in \Cref{proposition-unram-type-1,type-2-unram} when $f>1$ and $e=1$, and in \Cref{proposition-unram-type-1,summary-ramified-codim-1-b,summary-ramified-aa,summary-ramified-bb,summary-ramified-ab} when $f>1$ and $e>1$.

The result follows from considering these numerical criteria in conjunction with \Cref{criteria-codim-arbitrary} as well as the criteria for 
$\mathrm{Ext}^1_{\Fbar[\GL_2(k)]}(\sigma_{\mathbf{t}, \mathbf{s}}, \sigma_{\mathbf{t'}, \mathbf{s'}})$ to be non-zero given in \Cref{ext-gamma,H1-criteria} and \Cref{ext-GL_2(O_K)-unramified-generic,GL_2(O_K)-ramified-splits}. 

We note that a complete (and non-empty) list of pairs of Serre weights $\ts$ and $\tss$ satisfying $\dim \cX_{\ts} \cap \cX_{\tss} = [K:\Qp]-1$ and $\mathrm{Ext}^1_{\Fbar[\GL_2(k)]}(\sigma_{\mathbf{t}, \mathbf{s}}, \sigma_{\mathbf{t'}, \mathbf{s'}}) = 0$ (and so, necessarily, one of $\ts$ and $\tss$ is not weakly regular and $f>1$) is obtained by combining data from \Cref{table:unram-type-1,table:unram-type-2} when $e=1$, and \Cref{table:unram-type-1,table:ram-type-1b} along with the rows marked with an asterisk in \Cref{table:ram-type2-bb-candidates,table:ram-type2-ab-candidates} when $e>1$. 
\end{proof}

\begin{theorem}\label{thm-number}
		Let $\ts$ and $\tss$ be a pair of weakly regular and non-isomorphic Serre weights, so that $\dim \mathcal{X}_{\sigma_{\mathbf{t}, \mathbf{s}}} \cap \mathcal{X}_{\sigma_{\mathbf{t'}, \mathbf{s'}}} = [K:\Qp]-1$. Then the following are true:
  \begin{enumerate}
      \item  Let $n$ be the number of top-dimensional irreducible components in $\mathcal{X}_{\sigma_{\mathbf{t}, \mathbf{s}}} \cap \mathcal{X}_{\sigma_{\mathbf{t'}, \mathbf{s'}}}$, that is, irreducible components of dimension $[K:\Qp]-1$. If $e=1$, then $n=1$. If $e>1$, then \begin{align*}
		        n = \begin{cases}
		            2 &\text{ if }  \mathrm{Ext}^1_{\Fbar[\GL_2(k)]}(\sigma_{\mathbf{t}, \mathbf{s}}, \sigma_{\mathbf{t'}, \mathbf{s'}}) \neq 0, \\
		            1 &\text{ if }  \mathrm{Ext}^1_{\Fbar[\GL_2(k)]}(\sigma_{\mathbf{t}, \mathbf{s}}, \sigma_{\mathbf{t'}, \mathbf{s'}}) = 0.
		        \end{cases}
		    \end{align*}
      \item Let $\cZ$ be a top-dimensional irreducible component in $\mathcal{X}_{\sigma_{\mathbf{t}, \mathbf{s}}} \cap \mathcal{X}_{\sigma_{\mathbf{t'}, \mathbf{s'}}}$. If $e=1$, then $\cZ \subset \cX_{\sigma}$ for a non-Steinberg Serre weight $\sigma$ implies that $\sigma$ is isomorphic to either $\ts$ or $\tss$. If either $e>1$, $f=1$ and $s, s'< p-3$, or $e>1$ and $f>1$, then there exists a non-Steinberg $\sigma$ not isomorphic to either $\ts$ or $\tss$ so that $\cZ \subset \cX_{\ts} \cap \cX_{\tss} \cap \cX_{\sigma}$.
  \end{enumerate}

\end{theorem}
\begin{proof}
   The first statement follows from \Cref{rem:tot-ram-type1a-number} when $f=1$ and $e=1$, \Cref{rem:tot-ram-type1a-number,rem:tot-ram-type1c-number,summary-totally-ramified-type-2} when $f=1$ and $e>1$, and \Cref{type-1-summary,rem:summary-regular-type2-unram,type-II-regular} when $f>1$, by comparing with \Cref{criteria-codim-arbitrary} and \Cref{ext-gamma}.

For the second statement, suppose first that $K \neq \Qp$. The proof of \Cref{geometry-proposition-dimension-intersection} shows that there exist $G_K$-characters
$\tilde{\chi}_1$ and $\tilde{\chi}_2$ lifting $I_K$-characters $\chi_1$ and $\chi_2$ respectively so that
$\cZ$ is in the closure of a family $\cF_{\tilde{\chi}_2^{-1}, \tilde{\chi}_1^{-1}}$ contained in $|\cX_{\ts} \cap \cX_{\tss}|$. Using the proof of \Cref{criteria-codim-arbitrary}, the existence of such a top-dimensional component $\cZ$ and such $G_K$-characters $\tilde{\chi}_1$ and $\tilde{\chi_2}$ is equivalent to $\ts$ and $\tss$ having either a type I or a type II intersection witnessed by $(\chi_1, \chi_2)$. 

If it is a type II intersection, let $\sigma$ be the highest weight (see \Cref{defn:highest-wt}) for the pair $(\chi_1, \chi_2)$. Then $\sigma$ is non-Steinberg and not isomorphic to either $\ts$ or $\tss$, and moreover, $\cZ \subset \cX_{\sigma}$ as described in the proof of \Cref{prop:triple}. When $e=1$, $\ts$ and $\tss$ do not have a type II intersection by \Cref{rem:summary-regular-type2-unram}. This deals with the case of type II intersections between $\ts$ and $\tss$, and we may now assume that the pair $(\chi_1, \chi_2)$ witnesses a type I intersection between $\ts$ and $\tss$, with highest weight $\ts$. The existence of a non-Steinberg $\sigma$ not isomorphic to either $\ts$ or $\tss$ and such that $\cZ \subset \cX_{\sigma}$ is equivalent to $(\chi_1, \chi_2)$ witnessing a type II intersection between $\tss$ and $\sigma$. The desired statement then follows from comparing \Cref{prop:tot-ram-type1a,prop:tot-ram-type1c} with \Cref{type-II-tot-ram} when $f=1$, \Cref{table:unram-type-1}\ref{item-unram-generic} with \Cref{table:unram-type-2} when $f>1$ and $e=1$, and \Cref{table:unram-type-1}\ref{item-unram-generic} and \Cref{table:ram-type-1b}\ref{ram-type-1b-i} with \Cref{table:ram-type2-ab-candidates}\ref{ramified-ab-ii} when $f>1$ and $e>1$.

Next, suppose $K=\Qp$. If $\cZ$ is in the closure of a family $\cF_{\tilde{\chi}_2^{-1}, \tilde{\chi}_1^{-1}}$ as above, then the argument above goes through unchanged. Otherwise, by the proof of \Cref{geometry-proposition-dimension-intersection}, there exists an irreducible two-dimensional $G_K$-representation $\rhobar$ so that $\cZ$ is in the scheme-theoretic image of the map $f_{\rhobar}$ defined in \Cref{defn:irreducible-maps}. By \Cref{wt-irred} and \Cref{irred-intersection}, the scheme-theoretic image of $f_{\rhobar}$ can be contained in at most two irreducible components of $\cX$ and therefore, there does not exist $\sigma$ non-Steinberg and not isomorphic to either $\ts$ or $\tss$ such that $\cZ \subset \cX_{\sigma}$. This finishes the proof. \end{proof}

\begin{center}
 \begin{table}[hp]
 \begin{threeparttable}
 \resizebox{\textwidth}{!}{%
\begin{tabular}{ |N|l|l|l| }
\hline
\multicolumn{1}{|c|}{} & \multicolumn{1}{c|}{$\mathbf{s},\mathbf{s}'$} & \multicolumn{1}{c|}{Ext criterion} & \multicolumn{1}{c|}{$\Z/f\Z \smallsetminus \mu(J'_{\mathrm{max}})$ if $e=1$}\\
\hline
\label{item-unram-generic} & 
\begin{tabular}{l}
For some $i \in [1, f-1]$, \\
$(s_{f-1-i}, s_{f-i}, \dots, s_{f-2}, s_{f-1})=$ \\
$(\in [0, p-2], p-1, \dots, p-1, \in [0, p-2])$;\\[0.15cm]
$(s'_{f-1-i}, s'_{f-i}, \dots, s'_{f-2}, s'_{f-1})=$ \\
$(s_{f-1-i} + 1, 0, \dots, 0,  p-s_{f-1}-2)$ \\
\end{tabular} & \begin{tabular}{l}
Iff $i=1$,\\
\Cref{ext-gamma}\ref{ext-k-i}
\end{tabular} & \begin{tabular}{l} $f-i$ \end{tabular}\\
\hline
\label{item-unram-3} & 
\begin{tabular}{l}
$(s_{0}, \dots, s_{f-2}, s_{f-1})=$ \\
$(p-1, \dots, p-1, \in [0, p-3])$;\\[0.15cm]
$(s'_0, \dots, s'_{f-2}, s'_{f-1}) =(0, \dots, 0,  p-s_{f-1}-3)$ \\
\end{tabular} & \begin{tabular}{l} None \end{tabular}  & \begin{tabular}{l} $0$ \end{tabular}\\
\hline
\label{item-unram-6} & 
\begin{tabular}{l}
$(s_{0}, \dots, s_{f-3}, s_{f-2}, s_{f-1})= $\\
$(p-1, \dots, p-1, p-2, p-1)$;\\[0.15cm]
$(s'_0, \dots, s'_{f-3}, s'_{f-2}, s'_{f-1})= (0, \dots, 0,  0, p-2)$ 
\end{tabular} & \begin{tabular}{l} Iff $f=2$,\\
\Cref{ext-gamma}\ref{ext-k-i} \end{tabular} &\begin{tabular}{l} $f-1$ \end{tabular}   \\
\hline
\label{item-unram-7} & 
\begin{tabular}{l}
$e=1$,\\
$(s_{f-2}, s_{f-1}) = (\in [0, p-3],p-1)$; \\[0.15cm]
$(s'_{f-2}, s'_{f-1}) = (s_{f-2} + 2, p-1)$
\end{tabular} & \begin{tabular}{l} ? \end{tabular} & \begin{tabular}{l} $f-1$ \end{tabular}\\
\hline
\label{item-unram-5} & 
\begin{tabular}{l}
$f>2$, $e=1$ and for some $i \in [2, f-1]$,\\
$(s_{f-1-i}, s_{f-i}, \dots, s_{f-3}, s_{f-2}, s_{f-1}) =$\\
$ (\in [0, p-2], p-1, \dots, p-1, p-1, p-1)$; \\[0.15cm]
$(s'_{f-1-i}, s'_{f-i}, \dots, s'_{f-3}, s'_{f-2}, s'_{f-1}) =$ \\
$ (s_{f-1-i} + 1, 0, \dots, 0,  1, p-1)$
\end{tabular} & \begin{tabular}{l} None \end{tabular} &\begin{tabular}{l} $f-i$ \end{tabular}  \\
\hline
\label{item-unram-4} & 
\begin{tabular}{l}
$e=1$,\\
$(s_{0}, \dots, s_{f-3}, s_{f-2}, s_{f-1}) =$\\
$ (p-1, \dots, p-1, p-1, p-2)$; \\[0.15cm]
$(s'_{0}, \dots, s'_{f-3}, s'_{f-2}, s'_{f-1}) =$\\
$(0, \dots, 0, 1, p-1)$
\end{tabular} & \begin{tabular}{l} None \end{tabular} & \begin{tabular}{l} $0$ \end{tabular} \\
\hline
\end{tabular}}
\caption{Type 1 intersection when $f>1$, highest weight is $\ts$ and $\tss$ satisfies \Cref{ramified-codim-1}\ref{ramified-codim-1-a} or \ref{unramified-codim-1} with $j=f-1$. The column `Ext criterion' indicates which, if any, of the criteria for $\Ext^{1}_{\GL_2(\cO_K)}(\ts, \tss)$ to be non-zero hold. The notation ? indicates that it's not known if the group of extensions between the Serre weights is non-zero. By \Cref{ramified-codim-1}, $J'_{\mathrm{max}} = \Z/f\Z \smallsetminus \{f-1\}$ and the last column computes $\mu(J'_{\mathrm{max}})$ if $e=1$ for later use in the proof of \Cref{type-2-unram}.}\label{table:unram-type-1}
\end{threeparttable}
\end{table}
\end{center}
\setcounter{rownumber}{0}

 \begin{center}
 \begin{table}[hp]
 \begin{threeparttable}
 \resizebox{\textwidth}{!}{%
\begin{tabular}{ |N|l|l| }
\hline
\multicolumn{1}{|c|}{} & \multicolumn{1}{c|}{$\mathbf{s},\mathbf{s}'$} & \multicolumn{1}{c|}{Ext criterion} \\
\hline
\label{ram-type-1b-i} & 
\begin{tabular}{l}
$s_{f-1} \leq p-3$; \: \: $s'_{f-1} = s_{f-1} + 2$ \\
\end{tabular} & \begin{tabular}{l}
\Cref{H1-criteria}
\end{tabular}\\
\hline
\label{ram-type-1b-ii}
& \begin{tabular}{l}
For some $i \in [1, f-1]$,\\
$(s_{f-1-i}, s_{f-i}, \dots, s_{f-2}, s_{f-1})=(\in[0, p-2], p-1, \dots, p-1, p-1)$;\\[0.15cm]
$(s'_{f-1-i}, s'_{f-i}, \dots, s'_{f-2}, s'_{f-1}) =(s_{f-1-i}+ 1, 0, \dots, 0,  1)$ \\
\end{tabular} & \begin{tabular}{l} None \end{tabular}  \\
\hline
\label{ram-type-1b-iii}
& \begin{tabular}{l}
$(s_{0}, \dots, s_{f-2}, s_{f-1})= $
$(p-1, \dots, p-1, p-2)$;\\[0.15cm]
$(s'_0, \dots, s'_{f-2}, s'_{f-1})= $
$(0, \dots, 0,  1)$ \\
\end{tabular} & \begin{tabular}{l} None \end{tabular}   \\
\hline
\end{tabular}}
\caption{Type 1 intersection when $f>1$, $e>1$, the highest weight is $\ts$ and $\tss$ satisfies \Cref{ramified-codim-1}\ref{ramified-codim-1-b} with $j=f-1$. The column `Ext criterion' indicates which, if any, of the criteria for $\Ext^{1}_{\GL_2(\cO_K)}(\ts, \tss)$ to be non-zero hold.}\label{table:ram-type-1b}
\end{threeparttable}
\end{table}
\end{center}
\setcounter{rownumber}{0}

 \begin{table}[hp]
 \begin{threeparttable}
 \centering
 \resizebox{\textwidth}{!}{%
\begin{tabular}{ |N|l|l|l| }
\hline
\multicolumn{1}{|c|}{} & \multicolumn{1}{c|}{$\mathbf{s}',\mathbf{s}'', \bf{s}$} &  \multicolumn{1}{c|}{\Cref{table:unram-type-1} rows} & \multicolumn{1}{c|}{\begin{tabular}{l}
     Ext criterion 
\end{tabular}}\\
\hline
\label{item-1-with-4} & 
\begin{tabular}{l}
$f>2$ and for some $l \in [1,f-2]$, \\
$(s'_{l-1}, s'_l, \dots, s'_{f-2}, s'_{f-1})=(\in [1,p-2], 0, \dots, 0, \in [0, p-2])$;\\[0.15cm]
$(s''_{l-1}, s''_l, \dots, s''_{f-2}, s''_{f-1})=$\\
$(s'_{l-1} + 1, p-1 \dots, p-1, p-s'_{f-1}-2)$;\\[0.15cm]
$(s_{l-1}, s_l, \dots, s_{f-2}, s_{f-1}) = $\\
$(s'_{l-1}-1, p-1, \dots, p-1, p-s'_{f-1}-2)$
\\
\end{tabular} & \begin{tabular}{l} \ref{item-unram-generic} and \ref{item-unram-7} \end{tabular} & \begin{tabular}{l}
None
\end{tabular}\\
\hline
\label{item-1-with-5} & 
\begin{tabular}{l}
$f>3$ and for some $l \in [2, f-2]$ and $m \in [0, l-2]$, \\
$(s'_{m}, s'_{m+1}, \dots, s'_{f-2}, s'_{f-1}) = ([1, p-1], 0, \dots, 0, \in [0, p-2])$;\\[0.15cm]
$(s''_{m}, s''_{m+1}, \dots, s''_{l-2}, s''_{l-1}, s''_l, \dots, s''_{f-2}, s''_{f-1}) = $\\
$(s'_{m}, 0, \dots, 0, 1, p-1, \dots, p-1, p-s'_{f-1}-2)$;\\[0.15cm]
$(s_{m}, s_{m+1}, \dots, s_{f-2}, s_{f-1}) = (s'_{m}-1, p-1, \dots, p-1, p-s'_{f-1}-2)$\\

\end{tabular}& \begin{tabular}{l}
\ref{item-unram-generic} and \ref{item-unram-5} \end{tabular} & \begin{tabular}{l}
None
\end{tabular}\\
\hline
\label{item-2-with-4} & 
\begin{tabular}{l}
$p>3$, $l = 0$, \\
$(s'_0, \dots, s'_{f-2}, s'_{f-1}) = (0, \dots, 0, \in [1, p-3])$;\\[0.15cm]
$(s''_0, \dots, s''_{f-2}, s''_{f-1}) = (p-1, \dots, p-1, p - s'_{f-1}-1)$;\\[0.15cm]
$(s_0, \dots, s_{f-2}, s_{f-1}) = (p-1, \dots, p-1, p-s'_{f-1}-3)$\\

\end{tabular} & \begin{tabular}{l} \ref{item-unram-3} and \ref{item-unram-7} \end{tabular} & \begin{tabular}{l} None \end{tabular}  \\
\hline
\label{item-2-with-5} & 
\begin{tabular}{l}
$f>2$ and for some $l \in [1, f-2]$,\\
$(s'_0, \dots, s'_{f-2}, s'_{f-1}) = (0, \dots, 0, \in [0, p-3])$;\\[0.15cm]
$(s''_0, \dots, s''_{l-2}, s''_{l-1}, s''_l, \dots, s''_{f-2}, s''_{f-1}) = $\\
$(0, \dots, 0, 1, p-1, \dots, p-1, p-s'_{f-1}-2)$;\\[0.15cm]
$(s_0, \dots, s_{f-2}, s_{f-1}) = (p-1, \dots, p-1, p-s'_{f-1}-3)$\\

\end{tabular} & \begin{tabular}{l} \ref{item-unram-3} and \ref{item-unram-5} \end{tabular} & \begin{tabular}{l} None \end{tabular}   \\
\hline
\label{item-3-with-5a} & 
\begin{tabular}{l}
$f>2$ and for some $l \in [1, f-2]$,\\
$(s'_0, \dots, s'_{l-1}, s'_l, \dots, s'_{f-3}, s'_{f-2}, s'_{f-1})=$\\
$(p-1, \dots, p-1, 0, \dots, 0, 1, p-1)$;\\[0.15cm]
$(s''_0, \dots, s''_{l-1}, s''_{l}, s''_{l+1}, \dots, s''_{f-1}) =(0, \dots, 0, p-2, 0, \dots, 0)$;\\[0.15cm]
$(s_0, \dots, s_{l-2}, s_{l-1}, s_{l}, \dots, s_{f-1}) =$\\
$(p-1, \dots, p-1, p-2, p-1, \dots p-1)$\\



\end{tabular} & \begin{tabular}{l} \ref{item-unram-5} and \ref{item-unram-6} \end{tabular} & \begin{tabular}{l} None \end{tabular}  \\
\hline
%
\label{item-3-with-6a} & 
\begin{tabular}{l}
$l=0$,\\
$(s'_0, \dots, s'_{f-3}, s'_{f-2}, s'_{f-1}) = (0, \dots, 0, 1, p-1)$;\\[0.15cm]
$(s''_0, s''_1, \dots, s''_{f-2}, s''_{f-1}) = (p-2, 0, \dots, 0, 0)$;\\[0.15cm]
$(s_0, \dots, s_{f-2}, s_{f-1}) = (p-1, \dots, p-1, p-2)$\\


\end{tabular}& \begin{tabular}{l} \ref{item-unram-4} and \ref{item-unram-6} \end{tabular} & \begin{tabular}{l} None \end{tabular}   \\
\hline
%

\label{item-4-with-5} & 
\begin{tabular}{l}
$f>2$ and for some $l \in [1, f-2]$\\

$(s'_{l-1}, s'_{l}, \dots, s'_{f-3}, s'_{f-2}, s'_{f-1}) =(\in [1, p-2], 0, \dots, 0, 1, p-1)$;\\[0.15cm]
$(s''_{l-1}, s''_{l}, \dots, s''_{f-3}, s''_{f-2}, s''_{f-1}) =$\\
$(s'_{l-1} +1, p-1, \dots, p-1, p-1, p-1)$;\\[0.15cm]
$(s_{l-1}, s_{l}, \dots, s_{f-3}, s_{f-2}, s_{f-1}) =$\\
$(s'_{l-1}-1, p-1, \dots, p-1, p-1, p-1)$\\

\end{tabular} & \begin{tabular}{l} \ref{item-unram-5} and \ref{item-unram-7} \end{tabular} & \begin{tabular}{l} None \end{tabular}  \\
\hline
\label{item-5-with-5} & 
\begin{tabular}{l}
$f>3$ and for some $l \in [2, f-2]$, $m \in [0, l-2]$\\
$(s'_{m}, s'_{m+1}, \dots, s'_{f-3}, s'_{f-2}, s'_{f-1}) =(\in [1, p-1], 0, \dots, 0, 1, p-1)$;\\[0.15cm]
$(s''_{m}, s''_{m+1}, \dots, s''_{l-2}, s''_{l-1}, s''_l, \dots, s''_{f-1}) =$\\
$(s'_m, 0, \dots, 0, 1, p-1, \dots, p-1)$;\\[0.15cm]
$(s_{m}, s_{m+1}, \dots, s_{f-1}) =(s'_{m}-1, p-1, \dots, p-1)$\\


\end{tabular} & \begin{tabular}{l} \ref{item-unram-5} and \ref{item-unram-5} \end{tabular} & \begin{tabular}{l} None \end{tabular}  \\
\hline
\label{item-5-with-6} & 
\begin{tabular}{l}
$f>2$, $l=0$,\\
$(s'_{0}, s'_1, \dots, s'_{f-3}, s'_{f-2}, s'_{f-1}) = (p-1, 0, \dots, 0, 1, p-1)$;\\[0.15cm]
$(s''_{0}, s''_1, \dots, s''_{f-2}, s''_{f-1}) = (p-1, 0, \dots, 0, 1)$;\\[0.15cm]
$(s_0, s_1, \dots, s_{f-1}) = (p-2, p-1, \dots, p-1)$\\

\end{tabular} & \begin{tabular}{l} \ref{item-unram-5} and \ref{item-unram-4} \end{tabular} & \begin{tabular}{l} None \end{tabular}  \\
\hline
\end{tabular}}
\caption{When $f>1$ and $e=1$, pairs of Serre weights $\tss$ and $\tsss$ with a type II intersection witnessed by a pair of $I_K$-characters with highest weight $\ts$ are precisely those for which, after exchanging $\tss$ and $\tsss$ if necessary and translating all indices by some fixed element in $\Z/f\Z$ if necessary, $\mathbf{s}'$, $\mathbf{s}''$ and $\bf{s}$ are as in the rows of this table and $\bf{t}'$, $\bf{t}''$ and $\bf{t}$ satisfy \Cref{type-1-case-1-det-twist,det-twist-type2-unram}.
The column `\Cref{table:unram-type-1} rows' indicates respectively the row numbers in \Cref{table:unram-type-1} that $(\mathbf{s}, \mathbf{s}')$, and after further translating indices by adding $f-1-l \in \Z/f\Z$, $(\mathbf{s}, \mathbf{s}'')$ satisfy. The column `Ext criterion' indicates which, if any, of the criteria for $\Ext^{1}_{\GL_2(\cO_K)}(\tss, \tsss)$ to be non-zero hold.}  \label{table:unram-type-2}
\end{threeparttable}
\end{table}
\setcounter{rownumber}{0}

 \begin{center}
 \begin{table}[ht]
 \begin{threeparttable}
 \resizebox{\textwidth}{!}{%
\begin{tabular}{ |N|l| }
\hline
\multicolumn{1}{|c|}{} & \multicolumn{1}{c|}{$\mathbf{s}', \mathbf{s}'', \bf{s}$}\\
\hline
\label{ramified-a-meets-a-i} & 
\begin{tabular}{l}
$f>3$ and for some $l \in [1,f-3], r \in [l+1, f-2], m \in [0, l-1]$,\\
$(s'_{m}, s'_{m+1}, \dots, s'_{l-1}, s'_{l}) = (\in [0, p-2], p-1, \dots, p-1, \in [0,p-2])$,\\
$(s'_{r}, s'_{r+1}, \dots, s'_{f-2}, s'_{f-1}) = (\in [1, p-1], 0, \dots, 0, \in [0, p-2])$;\\[0.15cm]
$(s''_{m}, s''_{m+1}, \dots, s''_{l-1}, s''_{l}) = (s'_m +1, 0, \dots, 0, p-s'_l - 2)$,\\
$(s''_{r}, s''_{r+1}, \dots, s''_{f-2}, s''_{f-1}) = (s'_r - 1, p-1, \dots, p-1, p-s'_{f-1}-2)$;\\[0.15cm]
$(s_{m}, s_{m+1}, \dots, s_{l-1}, s_{l}) = (s'_m, p-1, \dots, p-1, s'_{l})$,\\
$(s_{r}, s_{r+1}, \dots, s_{f-2}, s_{f-1}) = (s'_r - 1, p-1, \dots, p-1, p-s'_{f-1}-2)$\\


\end{tabular}\\
\hline
\label{ramified-a-meets-a-ii} & 
\begin{tabular}{l}
$f>2$ and for some $l \in [0,f-3], r \in [l+1, f-2]$,\\
$(s'_{0}, \dots, s'_{l-1}, s'_{l}) = (p-1, \dots, p-1, \in [0,p-2])$,\\
$(s'_{r}, s'_{r+1}, \dots, s'_{f-2}, s'_{f-1}) = (\in [1, p-1], 0, \dots, 0, \in [0, p-2])$;\\[0.15cm]
$(s''_{0}, \dots, s''_{l-1}, s''_{l}) = (0, \dots, 0, p-s'_l - 2)$,\\
$(s''_{r}, s''_{r+1}, \dots, s''_{f-2}, s''_{f-1}) = (s'_r - 1, p-1, \dots, p-1, p-s'_{f-1}-1)$;\\[0.15cm]
$(s_{0}, \dots, s_{l-1}, s_{l}) = (p-1, \dots, p-1, s'_{l})$,\\
$(s_{r}, s_{r+1}, \dots, s_{f-2}, s_{f-1}) = (s'_r - 1, p-1, \dots, p-1, p-s'_{f-1}-2)$\\


\end{tabular} \\
\hline
\label{ramified-a-meets-a-iv} & 
\begin{tabular}{l}
For some $l \in [0,f-2]$,\\
$(s'_{0}, \dots, s'_{l-1}, s'_{l}, s'_{l+1}, \dots, s'_{f-2}, s'_{f-1}) = (p-1, \dots, p-1, \in [1, p-1], 0, \dots, 0, \in [0, p-2])$;\\[0.15cm]
$(s''_{0}, \dots, s''_{l-1}, s''_{l}, s''_{l+1}, \dots, s''_{f-2}, s''_{f-1}) = (0, \dots, 0, p-s'_l - 1, p-1, \dots, p-1, p-s'_{f-1}-1)$;\\[0.15cm]
$(s_{0}, \dots, s_{l-1}, s_{l}, s_{l+1}, \dots, s_{f-2}, s_{f-1}) = (p-1, \dots, p-1, s'_l - 1, p-1, \dots, p-1, p-s'_{f-1}-2)$\\
\end{tabular}   \\
\hline
\end{tabular}}
\caption{Let $f>1$, $e>1$. Consider all pairs $(\tss, \tsss)$ of non-isomorphic, non-Steinberg Serre weights so that that for a pair of $I_K$-characters $(\chi_1, \chi_2)$ with highest weight $\ts$, $\tss$ as well as $\tsss$ satisfy \Cref{ramified-codim-1}\ref{ramified-codim-1-a}. After exchanging $\tss$ and $\tsss$ if necessary and translating all indices by a fixed element of $\Z/f\Z$ if necessary, the triples $\bf{s}'$, $\bf{s}''$ and $\bf{s}$ are precisely those as in the rows of this table. As seen in \Cref{summary-ramified-aa}, none of the rows correspond to pairs $(\tss, \tsss)$ with type II intersections between them.}\label{table:ram-type2-aa-candidates}

\end{threeparttable}
\end{table}
\end{center}
\setcounter{rownumber}{0}

 \begin{center}
 \begin{table}[hp]
 \begin{threeparttable}
 \resizebox{\textwidth}{!}{%
\begin{tabular}{ |N|l|l|l| }
\hline
\multicolumn{1}{|c|}{} & \multicolumn{1}{c|}{$\mathbf{s}', \mathbf{s}'', \bf{s}$} &  \multicolumn{1}{c|}{\Cref{table:ram-type-1b} rows} & \multicolumn{1}{c|}{Ext criterion}\\
\hline
\label{ramified-bb-i} & 
\begin{tabular}{l}
For some $l \in [0,f-2]$, \\
$s'_{l} \in [0, p-3], s'_{f-1} \in [2, p-1]$;\\[0.15cm]
$s''_{l} = s'_l + 2, s''_{f-1} = s'_{f-1}-2$;\\[0.15cm]
$s_l = s'_l, s_{f-1} = s'_{f-1}-2$\\
\end{tabular} & \begin{tabular}{l}
\ref{ram-type-1b-i} and \ref{ram-type-1b-i} 
\end{tabular} & \begin{tabular}{l}
\end{tabular}\\
\hline
\label{ramified-bb-ii} & 
\begin{tabular}{l}
$f>2$ and for some $l \in [0, f-3], r\in [l+1, f-2]$,\\
$s'_{l} \in [0, p-3]$, $(s'_r, s'_{r+1}, \dots, s'_{f-2}, s'_{f-1}) = (\in [1, p-1], 0, \dots, 0, 1)$;\\[0.15cm]
$s''_{l} =s'_l +2$, $(s''_r, s''_{r+1}, \dots, s''_{f-1}) = (s'_r - 1, p-1, \dots, p-1)$;\\[0.15cm]
$s_{l} =s'_l$, $(s_r, s_{r+1}, s_{f-1}) = (s'_r - 1, p-1, \dots, p-1)$;\\



\end{tabular} & \begin{tabular}{l}
\ref{ram-type-1b-ii} and \ref{ram-type-1b-i} 
\end{tabular} & \begin{tabular}{l}
\end{tabular}\\
\hline
*\label{ramified-bb-iii} & 
\begin{tabular}{l}
For some $l \in [0,f-2]$,\\
$(s'_{l}, s'_{l+1}, \dots, s'_{f-2}, s'_{f-1}) = (\in [1, p-2], 0, \dots, 0, 1)$;\\[0.15cm]
$(s''_{l}, s''_{l+1}, \dots,  s''_{f-1}) = (s'_l + 1, p-1, \dots, p-1)$;\\[0.15cm]
$(s_{l}, s_{l+1}, \dots, s_{f-2}, s_{f-1}) = (s'_l - 1, p-1, \dots, p-1, p-1)$\\
\end{tabular}  & \begin{tabular}{l}
\ref{ram-type-1b-ii} and \ref{ram-type-1b-i} 
\end{tabular} & \begin{tabular}{l} None
\end{tabular} \\
\hline
\label{ramified-bb-iv} & 
\begin{tabular}{l}
$f>3$ and for some $l \in [1,f-3], r \in [l+1, f-2], m \in [0, l-1]$,\\
$(s'_m, s'_{m+1}, \dots, s'_{l-1}, s'_l) = (\in [0, p-2], p-1, \dots, p-1, p-1)$,\\
$(s'_{r}, s'_{r+1}, \dots, s'_{f-2}, s'_{f-1}) = (\in [1, p-1], 0, \dots, 0, 1)$;\\[0.15cm]
$(s''_m, s''_{m+1}, \dots, s''_{l-1}, s''_l) = (s'_m + 1, 0, \dots, 0, 1)$,\\
$(s''_{r}, s''_{r+1}, \dots, s''_{f-2}, s''_{f-1}) = (s'_r - 1, p-1, \dots, p-1, p-1)$;\\[0.15cm]
$(s_m, s_{m+1}, \dots, s_{l-1}, s_l) = (s'_m, p-1, \dots, p-1, p-1)$,\\
$(s_{r}, s_{r+1}, \dots, s_{f-2}, s_{f-1}) = (s'_l - 1, p-1, \dots, p-1, p-1)$\\
\end{tabular}  & \begin{tabular}{l}
\ref{ram-type-1b-ii} and \ref{ram-type-1b-ii} 
\end{tabular} & \begin{tabular}{l}
\end{tabular} \\
\hline
*\label{ramified-bb-v} & 
\begin{tabular}{l}
$f>2$ and for some $l \in [1,f-2], m \in [0, l-1]$,\\
$(s'_m, s'_{m+1}, \dots, s'_{f-2}, s'_{f-1}) = (\in [1, p-1], 0, \dots, 0, 1)$,\\[0.15cm]
$(s''_m, s''_{m+1}, \dots, s''_{l-1}, s''_l, s''_{l+1}, \dots, s''_{f-2}, s''_{f-1}) =$\\
$ (s'_m, 0, \dots, 0, 1, p-1, \dots, p-1, p-1)$;\\[0.15cm]
$(s_m, s_{m+1}, \dots, s_{f-1}) = (s'_m-1, p-1, \dots, p-1)$\\
\end{tabular} & \begin{tabular}{l}
\ref{ram-type-1b-ii} and \ref{ram-type-1b-ii} 
\end{tabular}  & \begin{tabular}{l} None
\end{tabular} \\
\hline
*\label{ramified-bb-vi} & 
\begin{tabular}{l}
$l \in [0,f-2]$,\\
$(s'_0, \dots, s'_{f-2}, s'_{f-1}) = (0, \dots, 0, 1)$,\\[0.15cm]
$(s''_0, \dots, s''_{l-1}, s''_l, s''_{l+1}, \dots, s''_{f-1}) = (0, \dots, 0, 1, p-1, \dots, p-1)$;\\[0.15cm]
$(s_0, \dots, s_{f-2}, s_{f-1}) = (p-1, \dots, p-1, p-2)$\\
\end{tabular}  & \begin{tabular}{l}
\ref{ram-type-1b-iii} and \ref{ram-type-1b-ii} 
\end{tabular}  & \begin{tabular}{l} None
\end{tabular}\\
\hline
\end{tabular}}
\caption{Let $f>1$, $e>1$. Consider all pairs $(\tss, \tsss)$ of non-isomorphic, non-Steinberg Serre weights so that that for a pair of $I_K$-characters $(\chi_1, \chi_2)$ with highest weight $\ts$, $\tss$ as well as $\tsss$ satisfy \Cref{ramified-codim-1}\ref{ramified-codim-1-b}. After exchanging $\tss$ and $\tsss$ if necessary and translating all indices by a fixed element of $\Z/f\Z$ if necessary, the triples $\bf{s}'$, $\bf{s}''$ and $\bf{s}$ are precisely those as in the rows of this table. If there exists a type II intersection between $\tss$ and $\tsss$, then the corresponding row number is marked with an asterisk. The column `\Cref{table:ram-type-1b} rows' indicates respectively the row numbers in \Cref{table:ram-type-1b} that $(\mathbf{s}, \mathbf{s}')$, and after further translating indices by adding $f-1-l \in \Z/f\Z$, $(\mathbf{s}, \mathbf{s}'')$ satisfy. The column `Ext criterion' indicates which, if any, of the criteria for $\Ext^{1}_{\GL_2(\cO_K)}(\tss, \tsss)$ to be non-zero hold in the rows with an asterisk. }\label{table:ram-type2-bb-candidates}
\end{threeparttable}
\end{table}
\end{center}
\setcounter{rownumber}{0}

 \begin{center}
 \begin{table}[hp]
 \begin{threeparttable}
 \resizebox{\textwidth}{!}{%
\begin{tabular}{ |N|l|l|l| }
\hline
\multicolumn{1}{|c|}{} & \multicolumn{1}{c|}{$\mathbf{s}', \mathbf{s}'', \bf{s}$} &  \multicolumn{1}{c|}{\begin{tabular}{l}
     Rows in \\
\Cref{table:unram-type-1,table:ram-type-1b} 
\end{tabular}} & \multicolumn{1}{c|}{\begin{tabular}{l}
     Ext \\
     criterion
\end{tabular}}\\
\hline
\label{ramified-ab-i} & 
\begin{tabular}{l}
$f>2$ and for some $l \in [0,f-3]$, $r \in [l+1, f-2]$, \\
$s'_{l} \in [0, p-3]$, $(s'_r, s'_{r+1}, \dots, s'_{f-2}, s'_{f-1}) = ([1, p-1], 0, \dots, 0, [0, p-2])$;\\[0.15cm]
$s''_{l} = s'_l + 2$, \\
$(s''_r, s''_{r+1}, \dots, s''_{f-2}, s''_{f-1}) = (s'_r - 1, p-1, \dots, p-1, p - s'_{f-1} -2)$;\\[0.15cm]
$s_{l} = s'_l$, $(s_r, s_{r+1}, \dots, s_{f-2}, s_{f-1}) = (s'_r - 1, p-1, \dots, p-1, p - s'_{f-1} -2)$\\
\end{tabular} & \begin{tabular}{l}
\ref{item-unram-generic} and \ref{ram-type-1b-i} 
\end{tabular} & \begin{tabular}{l} \end{tabular}\\
\hline
*\label{ramified-ab-ii} & 
\begin{tabular}{l}
For some $l \in [0,f-2]$,\\
$(s'_l, s'_{l+1}, \dots, s'_{f-2}, s'_{f-1}) = ([1, p-2], 0, \dots, 0, [0, p-2])$;\\[0.15cm]
$(s''_l, s''_{l+1}, \dots, s''_{f-2}, s''_{f-1}) = (s'_l + 1, p-1, \dots, p-1, p - s'_{f-1} -2)$;\\[0.15cm]
$(s_l, s_{l+1}, \dots, s_{f-2}, s_{f-1}) = (s'_l - 1, p-1, \dots, p-1, p - s'_{f-1} -2)$\\
\end{tabular} & \begin{tabular}{l}
\ref{item-unram-generic} and \ref{ram-type-1b-i} 
\end{tabular}& \begin{tabular}{l} Iff $l =f-2$, \\
\Cref{ext-gamma}\ref{ext-k-i}
\end{tabular}\\
\hline
\label{ramified-ab-iii} & 
\begin{tabular}{l}
For $l=f-1$ and some $r \in [0, f-2]$,\\
$(s'_r, s'_{r+1}, \dots, s'_{f-2}, s'_{f-1}) = ([1, p-1], 0, \dots, 0, [1, p-2])$;\\[0.15cm]
$(s''_r, s''_{r+1}, \dots, s''_{f-2}, s''_{f-1}) = (s'_r - 1, p-1, \dots, p-1, p - s'_{f-1})$;\\[0.15cm]
$(s_r, s_{r+1}, \dots, s_{f-2}, s_{f-1}) = (s'_r - 1, p-1, \dots, p-1, p - s'_{f-1} -2)$\\
\end{tabular} & \begin{tabular}{l}
\ref{item-unram-generic} and \ref{ram-type-1b-i} 
\end{tabular}& \begin{tabular}{l} \end{tabular}\\
\hline
\label{ramified-ab-iv} & 
\begin{tabular}{l}
$f>3$ and for some $l \in [1,f-3], m \in [0, l-1], r\in [l+1, f-2]$,\\
$(s'_m, s'_{m+1}, \dots, s'_l) = (\in [0, p-2], p-1, \dots, p-1)$,\\
$(s'_r, s'_{r+1}, \dots, s'_{f-2}, s'_{f-1}) = ([1, p-1], 0, \dots, 0, [0, p-2])$;\\[0.15cm]
$(s''_m, s''_{m+1}, \dots, s''_{l-1}, s''_l) = (s'_m + 1, 0, \dots, 0, 1)$,\\
$(s''_r, s''_{r+1}, \dots, s''_{f-2}, s''_{f-1}) = (s'_r -1, p-1, \dots, p-1, p - s'_{f-1} -2)$;\\[0.15cm]
$(s_m, s_{m+1}, \dots, s_l) = (s'_m, p-1, \dots, p-1)$,\\
$(s_r, s_{r+1}, \dots, s_{f-2}, s_{f-1}) = (s'_r - 1, p-1, \dots, p-1, p - s'_{f-1} -2)$\\
\end{tabular} & \begin{tabular}{l}
\ref{item-unram-generic} and \ref{ram-type-1b-ii} 
\end{tabular}& \begin{tabular}{l} \end{tabular}\\
\hline
\label{ramified-ab-v} & 
\begin{tabular}{l}
$f>2$ and for some $l \in [0,f-3], r\in [l+1, f-2]$,\\
$(s'_0, \dots, s'_l) = (p-1, \dots, p-1)$,\\
$(s'_r, s'_{r+1}, \dots, s'_{f-2}, s'_{f-1}) = ([1, p-1], 0, \dots, 0, [0, p-2])$;\\[0.15cm]
$(s''_0, \dots, s''_{l-1}, s''_l) = (0, \dots, 0, 1)$,\\
$(s''_r, s''_{r+1}, \dots, s''_{f-2}, s''_{f-1}) = (s'_r -1, p-1, \dots, p-1, p - s'_{f-1} -1)$;\\[0.15cm]
$(s_0, \dots, s_l) = (p-1, \dots, p-1)$,\\
$(s_r, s_{r+1}, \dots, s_{f-2}, s_{f-1}) = (s'_r - 1, p-1, \dots, p-1, p - s'_{f-1} -2)$\\
\end{tabular} & \begin{tabular}{l}
\ref{item-unram-generic} and \ref{ram-type-1b-ii} 
\end{tabular}& \begin{tabular}{l} \end{tabular}\\
\hline
*\label{ramified-ab-vi} & 
\begin{tabular}{l}
$f>2$ and for some $l \in [1,f-2], m\in [0, l-1]$,\\
$(s'_m, s'_{m+1}, \dots, s'_{f-2}, s'_{f-1}) = ([1, p-1], 0, \dots, 0, [0, p-2])$;\\[0.15cm]
$(s''_m, s''_{m+1}, \dots, s''_{l-1}, s''_l, s''_{l+1}, \dots, s''_{f-2}, s''_{f-1})=$\\
$(s'_m, 0, \dots, 0, 1, p-1, \dots, p-1, p-s'_{f-1}-2)$;\\[0.15cm]
$(s_m, s_{m+1}, \dots, s_{f-2}, s_{f-1}) = (s'_m - 1, p-1, \dots, p-1, p - s'_{f-1} -2)$\\
\end{tabular} & \begin{tabular}{l}
\ref{item-unram-generic} and \ref{ram-type-1b-ii} 
\end{tabular}& \begin{tabular}{l}Iff $l =f-2$, \\
\Cref{ext-gamma}\ref{ext-k-i} \end{tabular}\\
\hline
*\label{ramified-ab-vii} & 
\begin{tabular}{l}
$l=f-1$,\\
$(s'_0, \dots, s'_{f-2}, s'_{f-1}) = (0, \dots, 0, [1, p-3])$;\\[0.15cm]
$(s''_0, \dots, s''_{f-2}, s''_{f-1})=(p-1, \dots, p-1, p-s'_{f-1}-1)$;\\[0.15cm]
$(s_0, \dots, s_{f-2}, s_{f-1}) = (p-1, \dots, p-1, p-s'_{f-1}-3)$\\
\end{tabular} & \begin{tabular}{l}
\ref{item-unram-3} and \ref{ram-type-1b-i} 
\end{tabular}& \begin{tabular}{l} None \end{tabular}\\
\hline
*\label{ramified-ab-viii} & 
\begin{tabular}{l}
$l\in[0,f-2]$,\\
$(s'_0, \dots, s'_{f-2}, s'_{f-1}) = (0, \dots, 0, [0, p-3])$;\\[0.15cm]
$(s''_0, \dots, s''_{l-1}, s''_l, s''_{l+1}, \dots, s''_{f-3}, s''_{f-2}, s''_{f-1})=$\\$(0, \dots, 0, 1, p-1, \dots, p-1, p-s'_{f-1}-2)$;\\[0.15cm]
$(s_0, \dots, s_{f-2}, s_{f-1}) = (p-1, \dots, p-1, p-s'_{f-1}-3)$\\
\end{tabular} & \begin{tabular}{l}
\ref{item-unram-3} and \ref{ram-type-1b-ii} 
\end{tabular}& \begin{tabular}{l} Iff $l =f-2$, \\
\Cref{ext-gamma}\ref{ext-k-i} \end{tabular}\\
\hline
*\label{ramified-ab-ix} & 
\begin{tabular}{l}
$f>2$ and $l\in[0,f-3]$,\\
$(s'_0, \dots, s'_{f-3}, s'_{f-2}, s'_{f-1}) = (0, \dots, 0, 0, p-2)$;\\[0.15cm]
$(s''_0, \dots, s''_{l-1}, s''_l, s''_{l+1}, \dots, s''_{f-3}, s''_{f-2}, s''_{f-1})=$\\
$(0, \dots, 0, 1, p-1, \dots, p-1, p-1, 0)$;\\[0.15cm]
$(s_0, \dots, s_{f-3}, s_{f-2}, s_{f-1}) = (p-1, \dots, p-1, p-2, p-1)$\\
\end{tabular} & \begin{tabular}{l}
\ref{item-unram-6} and \ref{ram-type-1b-ii} 
\end{tabular}& \begin{tabular}{l} None \end{tabular}\\
\hline
*\label{ramified-ab-x} & 
\begin{tabular}{l}
$l=f-1$, \\
$(s'_0, \dots, s'_{f-2}, s'_{f-1}) = (0, \dots, 0, p-2)$; \\[0.15cm]
$(s''_0, \dots, s''_{f-2}, s''_{f-1})=(p-1, \dots, p-1, 1)$;\\[0.15cm]
$(s_0, \dots, s_{f-3}, s_{f-2}, s_1) = (p-1, \dots, p-1, p-2, p-1)$\\
\end{tabular} & \begin{tabular}{l}
\ref{item-unram-6} and \ref{ram-type-1b-ii} 
\end{tabular}& \begin{tabular}{l} None \end{tabular}\\
\hline
*\label{ramified-ab-xi} & 
\begin{tabular}{l}
$l=f-2$,\\
$(s'_0, \dots, s'_{f-3}, s'_{f-2}, s'_{f-1}) = (0, \dots, 0, 0, p-2)$;\\[0.15cm]
$(s''_0, \dots, s''_{f-3}, s''_{f-2}, s''_{f-1})=(0, \dots, 0, 1, 0)$;\\[0.15cm]
$(s_0, \dots, s_{f-3}, s_{f-2}, s_{f-1}) = (p-1, \dots, p-1, p-2, p-1)$\\
\end{tabular} & \begin{tabular}{l}
\ref{item-unram-6} and \ref{ram-type-1b-iii} 
\end{tabular}& \begin{tabular}{l} Iff $l =f-2$, \\
\Cref{ext-gamma}\ref{ext-k-i} \end{tabular}\\
\hline
\end{tabular}}
\caption{Let $f>1$, $e>1$. Consider all pairs $(\tss, \tsss)$ of non-isomorphic, non-Steinberg Serre weights so that for a pair of $I_K$-characters $(\chi_1, \chi_2)$ with highest weight $\ts$, $\tss$ satisfies \Cref{ramified-codim-1}\ref{ramified-codim-1-a}, while $\tsss$ satisfies \Cref{ramified-codim-1}\ref{ramified-codim-1-b}. After translating all indices by a fixed element of $\Z/f\Z$ if necessary, the triples $\bf{s}'$, $\bf{s}''$ and $\bf{s}$ are precisely those as in the rows of this table, with the rows corresponding to a type II intersection between $\tss$ and $\tsss$ marked with an asterisk. The column `Rows in \Cref{table:unram-type-1,table:ram-type-1b}' indicates respectively the row numbers in \Cref{table:unram-type-1,table:ram-type-1b} that $(\mathbf{s}, \mathbf{s}')$, and after further translating indices by adding $f-1-l \in \Z/f\Z$, $(\mathbf{s}, \mathbf{s}'')$ satisfy. The column `Ext criterion' indicates which, if any, of the criteria for $\Ext^{1}_{\GL_2(\cO_K)}(\tss, \tsss)$ to be non-zero hold in the rows with an asterisk. }\label{table:ram-type2-ab-candidates}
\end{threeparttable}
\end{table}
\end{center}
\setcounter{rownumber}{0}

\clearpage
\bibliographystyle{math}
\bibliography{my}
\end{document}